\documentclass[oneside,a4paper,reqno]{amsart}
\usepackage{pdfsync, enumerate}
\usepackage{stmaryrd}
\usepackage{mathrsfs}
\usepackage{hyperref}
\usepackage{tikz-cd}
\usepackage{amssymb}
\usepackage{stackrel}
\usepackage{adjustbox}
\usepackage{mathtools}
\usepackage{multicol}
\usepackage{amsmath, amsthm, amscd, amssymb, latexsym, eucal}
\usepackage[all]{xy}
\def\serieslogo@{} \def\@setcopyright{} \makeatother

\usepackage{multienum}

\usepackage[colorinlistoftodos]{todonotes}

\usepackage{hyperref}
\usepackage{color}
\usepackage{cite}

\makeatletter
\renewcommand*\env@matrix[1][c]{\hskip -\arraycolsep
	\let\@ifnextchar\new@ifnextchar
	\array{*\c@MaxMatrixCols #1}}
\makeatother

\usepackage{color}

\numberwithin{equation}{section}
\newtheorem{thm}{Theorem}[section]
\newtheorem*{main-thm}{Theorem}
\newtheorem*{Auslander-thm}{Auslander's Theorem}

\newtheorem{cor}[thm]{Corollary}
\newtheorem{lem}[thm]{Lemma}
\newtheorem{prop}[thm]{Proposition}

\newtheorem*{thmA}{Theorem~A}
\newtheorem*{thmB}{Theorem~B}

\theoremstyle{definition}
\newtheorem{defn}[thm]{Definition}

\newtheorem{rem}[thm]{Remark}
\newtheorem{exmp}[thm]{Example}

\newtheorem*{ackn}{Acknowledgment}
\newtheorem*{connot}{Conventions and Notation}
\newtheorem*{outline}{Outline of the paper}
 
\newtheorem*{krauserecollement}{Big singularity categories and Krause's recollement} 
\newtheorem*{adjointpairs}{Some useful facts}

\newtheorem{fact}[thm]{Fact}

\DeclareMathOperator*{\Ker}{\mathsf{Ker}}

\DeclareMathOperator{\pd}{\mathsf{pdim}\!}
\DeclareMathOperator{\fd}{\mathsf{fdim}\!}
\DeclareMathOperator*{\id}{\mathsf{idim}\!}

\DeclareMathOperator*{\gd}{\mathsf{gl.dim}\!}

\DeclareMathOperator*{\Mod}{\mathsf{Mod}-\!}

\DeclareMathOperator*{\smod}{\mathsf{mod}-\!}

\DeclareMathOperator*{\lMod}{\!-\mathsf{Mod}}
\DeclareMathOperator*{\lInj}{\!-\mathsf{Inj}}

\DeclareMathOperator*{\umod}{\underline{\mathsf{mod}}-\!}

\DeclareMathOperator*{\proj}{\mathsf{proj}-\!}

\DeclareMathOperator*{\GProj}{\mathsf{GProj}-\!}
\DeclareMathOperator*{\GInj}{\mathsf{GInj}-\!}
\DeclareMathOperator*{\Gproj}{\mathsf{Gproj}-\!}

\DeclareMathOperator*{\Inj}{\mathsf{Inj}-\!}

\DeclareMathOperator*{\uGProj}{\underline{\mathsf{GProj}}-\!}
\DeclareMathOperator*{\Proj}{\mathsf{Proj}-\!}

\DeclareMathOperator*{\uGproj}{\underline{\mathsf{Gproj}}-\!}
\DeclareMathOperator*{\oGInj}{\overline{\mathsf{GInj}}-\!}
\DeclareMathOperator*{\oGinj}{\overline{\mathsf{Ginj}}-\!}

\DeclareMathOperator*{\Tor}{\mathsf{Tor}}

\usepackage{stackengine}

\newsavebox{\proofbox}
\savebox{\proofbox}{\begin{picture}(7,7)%
	\put(0,0){\framebox(7,7){}}\end{picture}}

\usepackage{latexsym}
\usepackage{pstricks}
\usepackage{comment}

\usepackage[greek,english]{babel}

\begin{document}

\title{Cleft extensions of rings and singularity categories}

\author[Kostas]{Panagiotis Kostas}
\address{Department of Mathematics, Aristotle University of Thessaloniki, Thessaloniki 54124, Greece}
\email{pkostasg@math.auth.gr}

\date{\today}   

\keywords{%
Cleft extension of rings, Cleft extensions of abelian categories, Iwanaga-Gorenstein algebras, Gorenstein categories, Gorenstein projective modules, Singularity categories, Homotopy categories of injectives, Perfect endofunctors, Perfect bimodules.}

\subjclass[2020]{%
16E, 
16E30, 
16E65, 
16E10, 
16G50, 
18E, 
18G80
\!}

\begin{abstract}
This paper provides a systematic treatment of Gorenstein homological aspects for cleft extensions of rings. In particular, we investigate Goresnteinness, Gorenstein projective modules and singularity categories in the context of cleft extensions of rings. This setting includes triangular matrix rings, trivial extension rings and tensor rings, among others. Under certain conditions, we prove singular equivalences between the algebras in a cleft extension, unifying an abundance of known results. Moreover, we compare the big singularity categories of cleft extensions of rings in the sense of Krause. 
\end{abstract}
\maketitle

\setcounter{tocdepth}{1} \tableofcontents

\section{Introduction} 

Singularity categories are nowadays heavily studied in algebra and geometry. The singularity category of a Noetherian ring $\Lambda$, denoted by $\mathsf{D}_{\mathsf{sg}}(\Lambda)$, was introduced by Buchweitz \cite{buchweitz} as the Verdier quotient $\mathsf{D}^{\mathsf{b}}(\smod\Lambda)/\mathsf{K}^{\mathsf{b}}(\proj\Lambda)$ and it is a measure of regularity, in the sense that - at least for Artin algebras or local commutative rings - the category $\mathsf{D}_{\mathsf{sg}}(\Lambda)$ is trivial if and only if $\gd\Lambda<\infty$. Independently, Orlov \cite{orlov} introduced the singularity category of an algebraic variety $\mathsf{X}$, denoted by $\mathsf{D}_{\mathsf{sg}}(\mathsf{X})$, as the Verdier quotient $\mathsf{D}^{\mathsf{b}}(\mathsf{coh}\mathsf{X})/\mathsf{perf}(\mathsf{X})$, which turns out to be of great geometric interest. For instance, if $\mathsf{X}$ is smooth, then $\mathsf{D}_{\mathsf{sg}}(\mathsf{X})$ is trivial.

Krause \cite{krause} introduced the big singularity category of $\Lambda$ as the homotopy category of acyclic complexes of injectives, denoted by $\mathsf{K}_{\mathsf{ac}}(\Inj\Lambda)$. This is a category with coproducts and, in fact, compactly generated with $\mathsf{K}_{\mathsf{ac}}(\Inj\Lambda)^{\mathsf{c}}\simeq \mathsf{D}_{\mathsf{sg}}(\Lambda)$ up to summands.

In the study of singularity categories, Iwanaga-Gorenstein rings play a key role. A fundamental theorem of Buchweitz states that there is a fully faithful functor $i\colon\uGproj\Lambda\rightarrow \mathsf{D}_{\mathsf{sg}}(\Lambda)$ where the left-hand side denotes the stable category of the Gorenstein projective modules over $\Lambda$. Moreover, if $\Lambda$ is Iwanaga-Gorenstein, then the functor $i$ is a triangle equivalence.
\[
* \ \ * \ \ *
\]

The aim of this paper is to study the behaviour of the above triangulated categories under a \emph{cleft extension of rings}, i.e.\ for two rings $\Gamma$ and $\Lambda$ for which there are ring homomorphisms $f\colon \Lambda\rightarrow \Gamma$ and $g\colon \Gamma\rightarrow \Lambda$ with $fg=\mathsf{Id}_\Gamma$. The terminology ``cleft" follows MacLane \cite{maclane} and the same notion appears with different names, for instance ``split extensions'' \cite{pierce} and ``split quotiens" \cite{diracca_koenig}.

The goal, as described above, is in the spirit of Oppermann-Psaroudakis-Stai \cite{OPS}, where the same problem was investigated under a ring homomorphism. By following closely ideas of \cite{OPS} and the theory of \cite{beligiannis, beligiannis2}, it turns out that for cleft extensions we can say much more, compared to a single change of rings. 

The motivation for working with cleft extensions of rings is simple: there is an abundance of constructions that are of (great) interest in representation theory and occur as such. For instance, trivial extensions, tensor rings and, more generally, positively graded rings and $\theta$-extensions \cite{marmaridis}. In a dull way, every finite dimensional algebra over a field is a cleft extension of its semisimple part, while more sophisticated (and explicit) constructions are studied in \cite{chen_koenig, diracca_koenig, arrow2, monomial_arrow, arrow}. 

A cleft extension of rings $\Lambda\xrightarrow{f} \Gamma$ and $\Gamma\xrightarrow{g}\Lambda$ gives rise to the following diagram: 
\begin{equation} \label{cleft}
\begin{tikzcd}
\Mod\Gamma \arrow[rr, "\mathsf{i}"] &  & \Mod \Lambda \arrow[rr, "\mathsf{e}"] \arrow[ll, "-\otimes_{\Lambda}\Gamma"', bend right] \arrow[ll, "{\mathsf{Hom}_{\Lambda}(\Gamma,-)}", bend left] &  & \Mod\Gamma \arrow[ll, "-\otimes_{\Gamma}\Lambda"', bend right] \arrow[ll, "{\mathsf{Hom}_{\Gamma}(\Lambda,-)}", bend left] \arrow["-\otimes_{\Gamma}M"', loop, distance=2em, in=125, out=55] \arrow["{\mathsf{Hom}_{\Gamma}(M,-)}"', loop, distance=2em, in=305, out=235]
\end{tikzcd}
\end{equation}
where the functor $\mathsf{i}$ denotes the restriction induced by $f$, the functor $\mathsf{e}$ denotes the restriction induced by $g$ and $M=\mathsf{ker}f$. The left adjoints of (1.1) form a cleft extension of module categories and the right adjoints of (1.1) form a cleft coextension of module categories. Both belong in the framework of cleft (co)extensions of abelian categories in the sense of Beligiannis \cite{beligiannis, beligiannis2}. This language provides a purely homological approach to our problem.  

We have outlined the existence of $\mathsf{F}=-\otimes_{\Gamma}M$ and $\mathsf{F}'=\mathsf{Hom}_{\Gamma}(M,-)$; they are prominent in our theory, as imposing homological conditions on them imply homological properties for the rest of the functors of (1.1). The assumption that we will consider throughout asks for the functor $\mathsf{F}$ to be perfect and nilpotent, see Definition \ref{perfect_functor}. This will allow for good homological study of the cleft extension part of (1.1). The notion of a perfect endofunctor is the categorical analogue of the concept of a perfect bimodule, introduced in \cite{perfect}. In case $\Gamma$ is two-sided Noetherian and $M$ finitely generated on both sides, the functor $\mathsf{F}=-\otimes_{\Gamma}M$ is perfect and nilpotent precisely when the $\Gamma$-bimodule $M$ is perfect and nilpotent, see Lemma \ref{tensor_is_perfect}. We also introduce the dual notion of a coperfect endofunctor, see Definition \ref{coperfect_functor}, which allows for good homological study of the cleft coextension part of (1.1), as long as $\mathsf{F}'$ satisfies this property. Whenever the functor $\mathsf{F}=-\otimes_{\Gamma}M$ is perfect and nilpotent, it follows that the functor $\mathsf{F}'=\mathsf{Hom}_{\Gamma}(M,-)$ is coperfect and nilpotent, see Lemma \ref{perfect_implies_coperfect}. 
\[
* \ \ * \ \ *
\]
Recall that given a Noetherian ring $\Gamma$, there is a left adjoint to the inclusion functor $\mathsf{K}(\Inj\Gamma)\hookrightarrow \mathsf{K}(\Mod\Gamma)$ \cite{krause,neeman4}, which we denote by $\lambda_{\Gamma}$. 

Our first results, stated in the context of two-sided Noetherian rings, are summarised in Theorem A, below. This is a combination of Corollary \ref{middle_part}, Corollary \ref{cleft_of_stable}, Corollary \ref{cleft_of_big_singularity} and Corollary \ref{cleft_of_big_singularity_for_perfect}. 

\begin{thmA}
Let $\Gamma$ and $\Lambda$ be two-sided Noetherian rings such that $\Lambda$ is a cleft extension of $\Gamma$. Assume further that the functor $\mathsf{F}$, as in the diagram \textnormal{(1.1)}, is perfect and nilpotent. Then there is a commutative diagram of triangle functors 
\[
\begin{tikzcd}
\uGproj\Gamma \arrow[rr, "\underline{\mathsf{i}}", dashed] \arrow[dd, hook]    &  & \uGproj\Lambda \arrow[rr, "\underline{\mathsf{e}}", dashed] \arrow[ll, "\mathsf{q}"', bend right] \arrow[dd, hook]                                                      &  & \uGproj\Gamma \arrow[dd, hook] \arrow[ll, "\mathsf{l}"', bend right] \arrow["\underline{\mathsf{F}}"', dashed, loop, distance=2em, in=125, out=55]                                                                              \\
                                                                               &  &                                                                                                                                                                         &  &                                                                                                                                                                                                                                 \\
\mathsf{D}_{\mathsf{sg}}(\Gamma) \arrow[rr, "\mathsf{i}"] \arrow[dd, hook]     &  & \mathsf{D}_{\mathsf{sg}}(\Lambda) \arrow[ll, "\mathbb{L}_{\mathsf{sg}}\mathsf{q}"', bend right] \arrow[rr, "\mathsf{e}"] \arrow[ll, dashed, bend left] \arrow[dd, hook] &  & \mathsf{D}_{\mathsf{sg}}(\Gamma) \arrow[dd, hook] \arrow[ll, dashed, bend left] \arrow[ll, "\mathbb{L}_{\mathsf{sg}}\mathsf{l}"', bend right] \arrow["\mathbb{L}_{\mathsf{sg}}\mathsf{F}"', loop, distance=2em, in=35, out=325] \\
                                                                               &  &                                                                                                                                                                         &  &                                                                                                                                                                                                                                 \\
\mathsf{K}_{\mathsf{ac}}(\Inj\Gamma) \arrow[rr, "\lambda_{\Lambda}\mathsf{i}"] &  & \mathsf{K}_{\mathsf{ac}}(\Inj\Lambda) \arrow[rr, "\lambda_{\Gamma}\mathsf{e}"] \arrow[ll, bend right] \arrow[ll, "\mathsf{p}", bend left]                               &  & \mathsf{K}_{\mathsf{ac}}(\Inj\Gamma) \arrow[ll, "\mathsf{r}", bend left] \arrow[ll, bend right] \arrow["\lambda_{\Gamma}\mathsf{F}'"', loop, distance=2em, in=305, out=235]                                                    
\end{tikzcd}
\]
satisfying the following: 
\begin{itemize}
    \item[(i)] The composition of any two consecutive horizontal functors is the identity. 
    \item[(ii)] Any two consecutive parallel functors form an adjoint pair. 
\end{itemize}
\end{thmA}
To idea of the proof is to verify that under the perfectness assumptions that we impose, we may apply results of \cite{OPS} to the ring homomorphisms $\Gamma\rightarrow \Lambda$ and to $\Lambda\rightarrow \Gamma$. For the former, i.e. for the right hand side of the diagram, this is a direct verification while for the latter this is more subtle and the key tools are \cite[Proposition 6.6]{graded_injective_generation} and Proposition \ref{preserves_and_reflects}.

The dashed arrows, of Theorem A, on the level of stable categories of Gorenstein projectives exist, provided that the dashed arrows on the level of singularity categories exist, see Proposition \ref{sthnk}. We prove sufficient conditions for their existence in terms of the endofunctor $\mathsf{F}'$, see Proposition \ref{projectives_F'_inj} and Proposition \ref{equivalence_of_stable_for_theta}.

Under the setup of Theorem A, the algebras $\Gamma$ and $\Lambda$ are effectively related in terms of Gorensteiness. Further, imposing extra homological conditions yields singular equivalences. We state these results in Theorem B, below. Part (i) is proved in Corollary \ref{cor_for_gorenstein_rings}, part (ii) in Corollary \ref{equivalence_singularity_for_mod} and part (iii) in Corollary \ref{equivalence_of_big_singularity}.

\begin{thmB} \label{thmA}
Under the setup and assumptions of Theorem A, the following hold: 
\begin{itemize}
    \item[(i)] $\Lambda$ is Iwanaga-Gorenstein if and only if $\Gamma$ is Iwanaga-Gorenstein. 
    \item[(ii)] If $\mathbb{L}_{\mathsf{sg}}\mathsf{F}=0$ in $\mathsf{D}_{\mathsf{sg}}(\Gamma)$, then $\mathsf{e}\colon\mathsf{D}_{\mathsf{sg}}(\Lambda)\xrightarrow{\simeq}\mathsf{D}_{\mathsf{sg}}(\Gamma)$.  
    \item[(iii)] If $\lambda_{\Gamma}\mathsf{F}'=0$ in $\mathsf{K}_{\mathsf{ac}}(\Inj\Gamma)$, then $\lambda_{\Gamma}\mathsf{e}\colon\mathsf{K}_{\mathsf{ac}}(\Inj\Lambda)\xrightarrow{\simeq} \mathsf{K}_{\mathsf{ac}}(\Inj\Gamma)$.  
\end{itemize}
\end{thmB} 
By viewing tensor rings as cleft extensions, we recover the main theorem of \cite{perfect} as a direct consequence of part (i), see Corollary \ref{gorenstein_tensor_rings}. The arrow removal operation \cite{arrow} is a trivial extension with a nilpotent projective bimodule. In this way, parts (i) and (ii) recover results of \cite{arrow2}, see Corollary \ref{gorenstein_trivial_extensions} and Example \ref{equivalence_for_trivial_extensions}. Surprisingly, viewing triangular matrix rings as cleft extensions (see Example \ref{morita_context_rings}) can be fruitful. In this way, part (i) recovers the main theorem of \cite{xiong_zhang}, see Corollary \ref{gorenstein_triangular}, while from part (ii) we recover equivalences discussed in \cite{chen, PSS}, see Example \ref{equivalence_for_triangular}. In a similar fashion, we can view Morita context rings as cleft extensions (see Example \ref{morita_context_rings}) and deduce from part (i) a result of \cite{monomorphism_cat}, see Corollary \ref{gorenstein_morita_context}. 
\[
* \ \ * \ \ *
\]

\begin{outline}
In Section 2 we recall the definitions of a cleft extension and a cleft coextension of abelian categories, together with their basic homological properties. Section 3 is devoted to the study of right exact perfect endofunctors, left exact coperfect endofunctors and their internal description for module categories - meaning the relation with the notion of perfect bimodules \cite{perfect}.  

Section 4 deals with homological properties of cleft extensions. First, we compare Gorensteinness for abelian categories in a cleft extension that is also a cleft coextension. This is a categorical version of part (i) of Theorem B, see Theorem \ref{main_thm}. Then, we study properties of the functor $\mathsf{e}$ of a cleft extension. In Theorem \ref{thm2}, we provide sufficient conditions so that $\mathsf{e}$ is an eventually homological isomorphism and prove that whenever $\mathsf{F}$ is perfect and nilpotent, then it preserves and reflects objects with finite projective dimension, see Proposition \ref{preserves_and_reflects}. A dual result holds for cleft coextensions and injective dimension, see Proposition \ref{preserves_and_reflects_inj_dim}. Lastly, we apply the above to cleft extensions of module categories. This relies on the crucial observation that a cleft extension of module categories is also a cleft coextension and if the endofunctor $\mathsf{F}$, associated to the given cleft extension, is perfect and nilpotent, then the endofunctor $\mathsf{F}'$, associated to the cleft coextension, is coperfect and nilpotent, see Proposition \ref{cleft_of_modules_is_cocleft}. The main result of this section is Corollary \ref{main_thm_2}, which - when restricted to the case of two-sided Noetherian rings - is precisely part (i) of Theorem B.

Section 5 deals with examples. First, we provide an indirect description of the cleft extension structure of triangular matrix rings and Morita context rings via tuples, which makes for easy computations. Then, we explain the cleft extension structure of trivial extensions and tensor rings, by viewing them as $\theta$-extensions. For each example, we provide sufficient conditions for the functor $\mathsf{F}$ to be perfect. We also explain how to apply preceding results in order to compare Gorensteiness. In Section 6 we study singularity categories. We obtain the middle part of the commutative diagram of Theorem A, see Corollary \ref{middle_part}, which builds upon the theory of Section 3. Then, we prove an equivalence of singularity categories for abelian categories, see Theorem \ref{equivalence_singularity}, whose module-theoretic interpretation is part (i) of Theorem B. 

In Section 7 we study Gorenstein homological properties of a cleft extension diagram. In Corollary \ref{cleft_of_stable}, we obtain the upper part of the commutative diagram of Theorem A and prove an equivalence of stable categories of Gorenstein projective modules, see Proposition \ref{equivalence_of_stable}. We close this section with applications regarding CM-free rings and algebras of finite Cohen-Macaulay type. In Section 8 we study the big singularity categories of cleft extensions of Noetherian rings. We explain the lower part of the commutative diagram of Theorem A, see Proposition \ref{cleft_of_big_singularity}, and obtain an equivalence of big singularity categories, see Theorem \ref{fox}, from which part (iii) of Theorem B follows.
\end{outline}

\begin{connot}
All categories and functors are additive. Given a ring $\Lambda$ (which will always be unital), we denote by $\Mod \Lambda$ the category of right $\Lambda$-modules and by $\smod \Lambda$ the category of finitely presented right $\Lambda$-modules. By a Noetherian ring, we mean a right Noetherian ring. Unless otherwise stated, abelian categories are assumed to have enough projective and enough injective objects.
\end{connot}

\begin{ackn} 
The research project is implemented in the framework of H.F.R.I call ``Basic research Financing (Horizontal support of all Sciences)'' under the National Recovery and Resilience Plan “Greece 2.0” funded by the European Union – NextGenerationEU (H.F.R.I. Project Number: 16785).

I am grateful to my supervisor Chrysostomos Psaroudakis, for interesting discussions and comments. Part of this work was done at the University of Stuttgart during March and April of 2024. I would like to thank the members of the algebra group and especially Steffen Koenig for the warm hospitality. I also thank Georgios Dalezios and Odysseas Giatagantzidis for their comments. Finally, I would like to thank Yongyun Qin for questions and corrections that improved the paper and the anonymous referee for valuable suggestions and comments.
\end{ackn}

\section{Cleft extensions of abelian categories} 

In this section we recall the notions of a cleft extension and a cleft coextension of abelian categories, which are due to Beligiannis \cite{beligiannis, beligiannis2}. 

This constitutes a natural generalization of trivial extensions of abelian categories in the sense of Fossum-Griffith-Reiten \cite{reiten}. Many other categorical constructions fit in the framework of cleft extensions, for instance free categories, exterior categories and repetitive categories, see \cite{beligiannis}. We are interested in results about cleft extensions of rings. However, our proofs only rely on the homological structure of cleft extensions of abelian categories and since there are many examples outside the context of cleft extensions of rings, we will work with abelian categories and afterwards state the module-theoretic interpretation. 

Cleft extensions of abelian categories are heavily used in \cite{arrow, arrow2, monomial_arrow}. A self-contained treatment of the theory is provided in \cite[Section 2]{arrow}, and we will often refer to it.

\subsection{Cleft extensions} We begin with the definition of a cleft extension.

\begin{defn} \label{cleft_extension} (\!\!\cite[Definition 2.1]{beligiannis})
    A \emph{cleft extension} of an abelian category $\mathcal{B}$ is an abelian category $\mathcal{A}$ together with functors:
    \[
    \begin{tikzcd}
\mathcal{B} \arrow[rr, "\mathsf{i}"] &  & \mathcal{A} \arrow[rr, "\mathsf{e}"] &  & \mathcal{B} \arrow[ll, "\mathsf{l}"', bend right]
\end{tikzcd}
    \]
    satisfying the following: 
    \begin{itemize}
        \item[(i)] The functor $\mathsf{e}$ is faithful exact.
        \item[(ii)] The pair $(\mathsf{l},\mathsf{e})$ is an adjoint pair.
        \item[(iii)] There is a natural isomorphism $\mathsf{ei}\simeq \mathsf{Id}_{\mathcal{B}}$ of functors. 
    \end{itemize}
\end{defn}

From now on we will denote a cleft extension by $(\mathcal{B},\mathcal{A},\mathsf{i},\mathsf{e},\mathsf{l})$. The above data give rise to additional structure. For instance, it follows that the functor $\mathsf{i}$ is fully faithful and exact (see \cite[Lemma 2.2]{beligiannis},\cite[Lemma 2.2(ii)]{arrow}). Moreover, there is a functor $\mathsf{q}\colon\mathcal{A}\rightarrow\mathcal{B}$, which is left adjoint to $\mathsf{i}$ (see \cite[Proposition 2.3]{beligiannis} and \cite[Lemma 2.2(iv)]{arrow}). Then, $(\mathsf{ql},\mathsf{ei})$ is an adjoint pair and since $\mathsf{ei}\simeq \mathsf{Id}_{\mathcal{B}}$, it follows that $\mathsf{ql}\simeq \mathsf{Id}_{\mathcal{B}}$.

 We will now explain how to obtain certain endofunctors on the categories $\mathcal{A}$ and $\mathcal{B}$. This is of central importance, so we do things with more details, following the exposition of \cite{arrow}. Denote by $\nu\colon\mathsf{Id}_{\mathcal{B}}\rightarrow\mathsf{el}$ the unit and by $\mu\colon\mathsf{le}\rightarrow \mathsf{Id}_{\mathcal{A}}$ the counit of the adjoint pair $(\mathsf{l},\mathsf{e})$. For every $A\in\mathcal{A}$ and $B\in\mathcal{B}$, the following relations are satisfied: 
 \[
 \mathsf{Id}_{\mathsf{l}(B)}=\mu_{\mathsf{l}(B)}\mathsf{l}(\nu_B) \  \text{ and } \  \mathsf{Id}_{\mathsf{e}(A)}=\mathsf{e}(\mu_A)\nu_{\mathsf{e}(A)}.
 \]

Notice that $\mathsf{e}(\mu_A)$ is an epimorphism and since $\mathsf{e}$ is faithful exact, it follows that $\mu_{A}\colon \mathsf{le}(A)\rightarrow A$ is an epimorphism for every $A\in\mathcal{A}$. Therefore, we may consider the following short exact sequence: 
\[
0\rightarrow \mathsf{ker}\mu_{A}\rightarrow \mathsf{le}(A)\xrightarrow{\mu_A}A\rightarrow 0,
\]
for every $A\in\mathcal{A}$. The assignment $A\mapsto \mathsf{ker}\mu_A$ defines a right exact endofunctor $\mathsf{G}\colon \mathcal{A}\rightarrow \mathcal{A}$. Given an object $B\in\mathcal{B}$, we denote by $\mathsf{F}(B)$ the object $\mathsf{eGi}(B)$. The assignment $B\mapsto \mathsf{F}(B)$ defines a right exact endofunctor $\mathsf{F}\colon\mathcal{B}\rightarrow \mathcal{B}$. Note that $\mathsf{F}(B)$ belongs in the following short exact sequence: 
\[
0\rightarrow \mathsf{F}(B)\rightarrow \mathsf{el}(B)\xrightarrow{\mathsf{e}(\mu_{\mathsf{i}(B)})}B\rightarrow 0,
\]
(which occurs after identifying $\mathsf{elei}(B)$ with $\mathsf{el}(B)$ and $\mathsf{ei}(B)$ with $B$) and since $\mathsf{e}\mu_{\mathsf{i}(B)}$ is a split epimorphism, the above short exact sequence splits, meaning that there is a natural isomorphism
\[
\mathsf{el}\simeq \mathsf{F}\oplus \mathsf{Id}_{\mathcal{B}}.
\]
One can then prove that $\mathsf{F}^n\mathsf{e}\simeq \mathsf{e}\mathsf{G}^n$ for all $n\geq 1$ and in particular $\mathsf{F}$ is nilpotent if and only if $\mathsf{G}$ is nilpotent (see \cite[Lemma~2.4]{arrow}). It also follows that for every $A\in\mathcal{A}$ and $n\geq 1$, there is a short exact sequence 
\[
0\rightarrow \mathsf{G}^n(A)\rightarrow \mathsf{lF}^{n-1}\mathsf{e}(A)\rightarrow \mathsf{G}^{n-1}(A)\rightarrow 0.
\]
The above short exact sequences are key ingredients to the proofs of this paper.

From now on, given a cleft extension $(\mathcal{B},\mathcal{A},\mathsf{i},\mathsf{e},\mathsf{l})$, we will also consider the induced functors $\mathsf{q}$, $\mathsf{F}$ and $\mathsf{G}$ as part of the structure, as well as their properties - often without explicit mention. In the following lemma we collect some elementary homological properties of cleft extensions.

\begin{lem}  \textnormal{(\!\!\!\cite[Corollary 4.2]{beligiannis2})} \label{basic_homological_properties_of_cleft}
    Let $(\mathcal{B},\mathcal{A},\mathsf{i},\mathsf{e},\mathsf{l})$ be a cleft extension of abelian categories. Consider an object $X$ of $\mathcal{B}$. The following hold: 
    \begin{itemize}
        \item[(i)] $X\in \Proj\mathcal{B}$ if and only if $\mathsf{l}(X)\in\Proj\mathcal{A}$.
        \item[(ii)] $\mathbb{L}_i\mathsf{F}(X)\cong \mathsf{e}\mathbb{L}_i\mathsf{l}(X)$ for all $i\geq 1$. 
        \item[(iii)] If $\mathbb{L}_i\mathsf{F}(X)=0$ for all $i\geq 1$, then $\mathsf{Ext}_{\mathcal{A}}^i(\mathsf{l}(X),Y)\cong \mathsf{Ext}_{\mathcal{B}}^i(X,\mathsf{e}(Y))$ for all $i\geq 1$ and every object $Y$ of $\mathcal{A}$.
    \end{itemize} 
\end{lem} 
\begin{proof}
(i) If $X$ is a projective object of $\mathcal{B}$, then $\mathsf{l}(X)$ is a projective object of $\mathcal{A}$, since $(\mathsf{l},\mathsf{e})$ is an adjoint pair of functors and $\mathsf{e}$ is exact. If, on the other hand $\mathsf{l}(X)$ is a projective object of $\mathcal{A}$, since $(\mathsf{q},\mathsf{i})$ is an adjoint pair and $\mathsf{i}$ is exact, it follows that $X\cong \mathsf{ql}(X)$ is a projective object of $\mathcal{B}$.  

(ii) Consider an exact sequence $0\rightarrow K\xrightarrow{i}P\xrightarrow{a}X\rightarrow 0$ in $\mathcal{B}$ with $P$ projective. Applying the functor $\mathsf{F}$ to the above gives the following exact sequence in $\mathcal{B}$:
\[
0\rightarrow\mathbb{L}_1\mathsf{F}(X)\rightarrow\mathsf{F}(K)\xrightarrow{\mathsf{F}(i)}\mathsf{F}(P)\xrightarrow{\mathsf{F}(a)}\mathsf{F}(X)\rightarrow 0,
\]
while applying the functor $\mathsf{l}$ gives the following exact sequence in $\mathcal{A}$: 
\[
0\rightarrow\mathbb{L}_1\mathsf{l}(X)\rightarrow\mathsf{l}(K)\xrightarrow{\mathsf{l}(i)}\mathsf{l}(P)\xrightarrow{\mathsf{l}(a)}\mathsf{l}(X)\rightarrow 0.
\]
We have $\mathbb{L}_1\mathsf{F}(X)\cong \mathsf{ker(F}(i))$ and $\mathbb{L}_1\mathsf{l}(X)\cong \Ker{\mathsf{l}(i)}$. Moreover, since $\mathsf{e}$ is exact, it follows that $\mathsf{e}(\Ker(\mathsf{l}(i)))=\Ker(\mathsf{el}(i))$. Using the natural isomorphism $\mathsf{el}\simeq \mathsf{Id}_{\mathcal{B}}\oplus \mathsf{F}$, we infer that
\[
\mathsf{e}\mathbb{L}_1\mathsf{l}(X)\cong \mathsf{e}(\Ker(\mathsf{l}(i)) \cong \Ker(\mathsf{el}(i))\cong \Ker(\mathsf{F}(i))\cong \mathbb{L}_1\mathsf{F}(X).
\] 
In the same way it follows that $\mathbb{L}_i\mathsf{F}(X)\cong \mathsf{e}\mathbb{L}_i\mathsf{l}(X)$ for all $i\geq 2$. 

(iii) By (ii) and since $\mathsf{e}$ is faithful, we have that $\mathbb{L}_i\mathsf{l}(X)=0$ for all $i\geq 1$. Consider a projective resolution of $X$:
\[
P^{\bullet}: \ \cdots\rightarrow P_n\rightarrow \cdots\rightarrow P_0\rightarrow X\rightarrow 0.
\]
Applying the functor $\mathsf{l}$ to the latter gives an exact sequence 
\[
\mathsf{l}(P^{\bullet}): \ \cdots\rightarrow \mathsf{l}(P_n)\rightarrow\cdots\rightarrow\mathsf{l}(P_0)\rightarrow \mathsf{l}(X)\rightarrow 0,
\]
and since each $\mathsf{l}(P_k)$ is projective, the above is a projective resolution of $\mathsf{l}(X)$. By applying $\mathsf{Hom}_{\mathcal{B}}(-,\mathsf{e}(Y))$ to $P^{\bullet}$ and $\mathsf{Hom}_{\mathcal{A}}(-,Y)$ to $\mathsf{l}(P^{\bullet})$ and using the adjunction $(\mathsf{l},\mathsf{e})$, we derive the following commutative diagram:
\[
\adjustbox{scale=0.86}{
   \begin{tikzcd}
\cdots \arrow[r] & {\mathsf{Hom}_{\mathcal{A}}(\mathsf{l}(X),Y)} \arrow[r] \arrow[d, "\cong"] & {\mathsf{Hom}_{\mathcal{A}}(\mathsf{l}(P_0),Y)} \arrow[r] \arrow[d, "\cong"] & \cdots \arrow[r] & {\mathsf{Hom}_{\mathcal{A}}(\mathsf{l}(P_n),Y)} \arrow[r] \arrow[d, "\cong"] & \cdots \\
\cdots \arrow[r] & {\mathsf{Hom}_{\mathcal{B}}(X,\mathsf{e}(Y))} \arrow[r]                    & {\mathsf{Hom}_{\mathcal{B}}(P_0,\mathsf{e}(Y))} \arrow[r]                    & \cdots \arrow[r] & {\mathsf{Hom}_{\mathcal{B}}(P_n,\mathsf{e}(Y))} \arrow[r]                    & \cdots
\end{tikzcd}}
\]
From the above we infer that $\mathsf{Ext}_{\mathcal{A}}^i(\mathsf{l}(X),Y)\cong \mathsf{Ext}_{\mathcal{B}}^i(X,\mathsf{e}(Y))$ for all $i\geq 1$. 
\end{proof}

We end this subsection with a characterization of the projective objects of the middle category in a cleft extension. 

\begin{lem}  \label{projectives_in_cleft}
Let $(\mathcal{B},\mathcal{A},\mathsf{i},\mathsf{e},\mathsf{l})$ be a cleft extension of abelian categories. An object of $\mathcal{A}$ is projective if and only if it is a direct summand of $\mathsf{l}(P)$ for some $P\in\Proj\mathcal{B}$. 
\end{lem}
\begin{proof}
If $X$ is a direct summand of $\mathsf{l}(P)$, then it is projective since $\mathsf{l}(P)$ is projective. On the other hand, for any object $X$ of $\mathcal{A}$, there is a projective object $P\in\mathcal{B}$ and an epimorphism $P\twoheadrightarrow \mathsf{e}(X)$. Consequently, since $\mathsf{l}$ is right exact, there is an epimorphism $\mathsf{l}(P)\twoheadrightarrow \mathsf{le}(X)$. Further, there is an epimorphism $\mathsf{le}(X)\twoheadrightarrow X$, so composing the latter gives an epimorphism $\mathsf{l}(P)\twoheadrightarrow X$. If $X$ is assumed to be projective, then it is a direct summand of $\mathsf{l}(P)$. 
\end{proof}

\subsection{Cleft coextensions} In the theory of trivial extensions of abelian categories, there is the dual notion of trivial coextensions (or ``left" trivial extensions \cite{reiten}). Likewise, there is a dual notion to cleft extensions which we recall below. 

\begin{defn} \label{cleft_coextension} (\!\!\cite[Section 2]{beligiannis})
    A \emph{cleft coextension} of an abelian category $\mathcal{B}$ is an abelian category $\mathcal{A}$ together with functors:
    \[
    \begin{tikzcd}
\mathcal{B} \arrow[rr, "\mathsf{i}"] &  & \mathcal{A} \arrow[rr, "\mathsf{e}"] &  & \mathcal{B} \arrow[ll, "\mathsf{r}", bend left]
\end{tikzcd}
    \]
    satisfying the following: 
    \begin{itemize}
        \item[(i)] The functor $\mathsf{e}$ is faithful exact. 
        \item[(ii)] The pair $(\mathsf{e},\mathsf{r})$ is an adjoint pair. 
        \item[(iii)] There is a natural isomorphism $\mathsf{ei}\simeq \mathsf{Id}_{\mathcal{B}}$ of functors.
    \end{itemize}
\end{defn}

From now on we will denote a cleft coextension by $(\mathcal{B},\mathcal{A},\mathsf{i},\mathsf{e},\mathsf{r})$. Like with cleft extensions, there is more structure that arises from the above information. For example, it turns out that $\mathsf{i}$ is fully faithful exact and that there is a functor $\mathsf{p}\colon\mathcal{A}\rightarrow \mathcal{B}$ such that $(\mathsf{i},\mathsf{p})$ is an adjoint pair. Necessarily we have that $\mathsf{pr}\simeq\mathsf{Id}_{\mathcal{B}}$. Most importantly for the purposes of this paper, there are endofunctors $\mathsf{G}'\colon\mathcal{A}\rightarrow\mathcal{A}$ and $\mathsf{F}'\colon\mathcal{B}\rightarrow\mathcal{B}$ that appear in the following short exact sequences: 
\[
0\rightarrow \mathsf{Id}_{\mathcal{A}}\rightarrow \mathsf{re}\rightarrow\mathsf{G}'\rightarrow 0 \  \text{ and }  \ 0\rightarrow \mathsf{Id}_{\mathcal{B}}\rightarrow \mathsf{er}\rightarrow \mathsf{F'}\rightarrow 0.
\]
The second one splits. Further, $\mathsf{e}\mathsf{G}'^n\simeq \mathsf{F}'^n\mathsf{e}$ for every $n\geq 1$ and in particular, the functor $\mathsf{G}'$ is nilpotent if and only if $\mathsf{F}'$ is nilpotent. Moroever, for every $A\in\mathcal{A}$ and $n\geq 1$, there is a short exact sequence 
\[
0\rightarrow \mathsf{G}'^{n-1}(A)\rightarrow \mathsf{rF}'^{n-1}\mathsf{e}(A)\rightarrow \mathsf{G}'^{n}(A)\rightarrow 0.
\]
Every property of a cleft extension has a dual counterpart for cleft coextensions. In this case, we will usually leave the proofs of the second to the reader. The following lemmata are the first such instances and, for this reason, we provide full proofs. 

\begin{lem} \label{basic_homological_properties_of_cocleft}
    Let $(\mathcal{B},\mathcal{A},\mathsf{i},\mathsf{e},\mathsf{r})$ be a cleft coextension of abelian categories. Consider an object $X$ of $\mathcal{B}$. The following hold: 
    \begin{itemize}
        \item[(i)] $X\in\Inj \mathcal{B}$ if and only if $\mathsf{r}(X)\in\Inj\mathcal{A}$. 
        \item[(ii)] $\mathbb{R}^i\mathsf{F}'(X)\cong\mathsf{e}\mathbb{R}^i\mathsf{r}(X)$ for all $i\geq 1$. 
        \item[(iii)] If $\mathbb{R}^i\mathsf{F}'(X)=0$ for all $i\geq 1$, then $\mathsf{Ext}_{\mathcal{A}}^i(Y,\mathsf{r}(X))\cong \mathsf{Ext}^i_{\mathcal{B}}(\mathsf{e}(Y),X)$ for all $i\geq 1$ and every object $Y$ of $\mathcal{A}$.
    \end{itemize}
\end{lem}
\begin{proof}
(i) If $X$ is an injective object of $\mathcal{B}$, then $\mathsf{r}(X)$ is an injective object of $\mathcal{A}$, since $(\mathsf{e},\mathsf{r})$ is an adjoint pair and $\mathsf{i}$ is exact. If, on the other hand, $\mathsf{r}(X)$ is an injective object of $\mathcal{A}$, since $(\mathsf{i},\mathsf{p})$ is an adjoint pair and $\mathsf{e}$ is exact, it follows that $X\cong \mathsf{pr}(X)$ is an injective object of $\mathcal{B}$. 

(ii) Consider an exact sequence $0\rightarrow X\xrightarrow{a}I\xrightarrow{i}K\rightarrow 0$ in $\mathcal{B}$ with $I$ injective. Applying the functor $\mathsf{F}'$ gives the following exact sequence in $\mathcal{B}$: 
\[
    0\rightarrow \mathsf{F}'(X)\xrightarrow{\mathsf{F}'(a)}\mathsf{F}'(I)\xrightarrow{\mathsf{F}'(i)}\mathsf{F}'(K)\rightarrow \mathbb{R}^1\mathsf{F}'(X)\rightarrow 0,
\]
while applying the functor $\mathsf{r}$ gives the following exact sequence in $\mathcal{A}$: 
\[
0\rightarrow \mathsf{r}(X)\xrightarrow{\mathsf{r}(a)}\mathsf{r}(I)\xrightarrow{\mathsf{r}(i)}\mathsf{r}(K)\rightarrow\mathbb{R}^1\mathsf{r}(X)\rightarrow0.
\]
We have $\mathbb{R}^1\mathsf{F}'(X)\cong \mathsf{coker}(\mathsf{F}'(i))$ and $\mathbb{R}^1(\mathsf{r}(X))\cong \mathsf{coker}(\mathsf{r}(i))$. Moreover, since $\mathsf{e}$ is exact, it follows that $\mathsf{e}(\mathsf{coker}(\mathsf{r}(i)))\cong\mathsf{coker(\mathsf{er}(i))}$. Using the natural isomorphism $\mathsf{er}\simeq \mathsf{F}'\oplus\mathsf{Id}_{\mathcal{B}}$, we infer that 
\[
\mathsf{e}\mathbb{R}^1\mathsf{r}(X)\cong \mathsf{e}(\mathsf{coker}(\mathsf{r}(i)))\cong \mathsf{coker}(\mathsf{F}'(i))\cong \mathbb{R}^1\mathsf{F}'(X). 
\]
In the same way it follows that $\mathbb{R}^i\mathsf{F}'(X)\cong \mathsf{e}\mathbb{R}^i\mathsf{r}(X)$ for all $i\geq 2$. 

(iii) By (ii) and since $\mathsf{e}$ is faithful, we have $\mathbb{R}^i\mathsf{r}(X)=0$ for all $i\geq 1$. Consider an injective resolution of $X$: 
\[
I^{\bullet}: 0\rightarrow X\rightarrow I_0\rightarrow\cdots\rightarrow I_n\rightarrow \cdots
\]
Applying the functor $\mathsf{r}$ to the latter gives an exact sequence 
\[
\mathsf{r}(I^{\bullet}): 0\rightarrow \mathsf{r}(X)\rightarrow\mathsf{r}(I_0)\rightarrow\cdots\rightarrow \mathsf{r}(I_n)\rightarrow\cdots
\]
and since each $\mathsf{r}(I_k)$ is injective, the above is an injective resolution of $\mathsf{r}(X)$. By applying $\mathsf{Hom}_{\mathcal{B}}(\mathsf{e}(Y),-)$ to $I^{\bullet}$ and $\mathsf{Hom}_{\mathcal{A}}(Y,-)$ to $\mathsf{r}(I^{\bullet})$ and using the adjunction $(\mathsf{e},\mathsf{r})$, we derive the following commutative diagram:
\[
\adjustbox{scale=0.87}{\begin{tikzcd}
\cdots \arrow[r] & {\mathsf{Hom}_{\mathcal{A}}(Y,\mathsf{r}(X))} \arrow[r] \arrow[d, "\cong"] & {\mathsf{Hom}_{\mathcal{A}}(Y,\mathsf{r}(I_0))} \arrow[r] \arrow[d, "\cong"] & \cdots \arrow[r] & {\mathsf{Hom}_{\mathcal{A}}(Y,\mathsf{r}(I_n))} \arrow[r] \arrow[d, "\cong"] & \cdots \\
\cdots \arrow[r] & {\mathsf{Hom}_{\mathcal{B}}(\mathsf{e}(Y),X)} \arrow[r]                    & {\mathsf{Hom}_{\mathcal{B}}(\mathsf{e}(Y),I_0)} \arrow[r]                    & \cdots \arrow[r] & {\mathsf{Hom}_{\mathcal{B}}(\mathsf{e}(Y),I_n)} \arrow[r]                    & \cdots
\end{tikzcd}}
\]
From the above we infer that $\mathsf{Ext}_{\mathcal{A}}^i(Y,\mathsf{r}(X))\cong \mathsf{Ext}^i_{\mathcal{B}}(\mathsf{e}(Y),X)$ for all $i\geq 1$. 
\end{proof}

The following is the dual of Lemma \ref{projectives_in_cleft}. 

\begin{lem} \label{injectives_in_cleft}
    Let $(\mathcal{B},\mathcal{A},\mathsf{i},\mathsf{e},\mathsf{l})$ be a cleft coextension of abelian categories. An object of $\mathcal{A}$ is injective if and only if it is a direct summand of $\mathsf{r}(I)$ for some $I\in\Inj\mathcal{B}$.  
\end{lem}
\begin{proof}
    If $X$ is a direct summand of $\mathsf{r}(I)$, then it is injective since $\mathsf{r}(I)$ is injective. On the other hand, for any object $X$ of $\mathcal{A}$, there is an injective object $I\in\mathcal{B}$ and a monomorphism $\mathsf{e}(X)\rightarrowtail I$. Consequently, since $\mathsf{r}$ is left exact, there is a monomorphism $\mathsf{re}(X)\rightarrowtail \mathsf{r}(I)$. Further, there is a monomorphism $X\rightarrowtail \mathsf{re}(X)$, so composing the latter gives a monomorphism $X\rightarrowtail \mathsf{r}(I)$. If $X$ is assumed to be injective, then it is a direct summand of $\mathsf{r}(I)$. 
\end{proof}

\section{Perfect and coperfect endofunctors}

In this section we define the notion of a perfect endofunctor and study its basic properties. The notion of a left perfect endofunctor was recently considered in \cite{graded_injective_generation}. In the generality that we work with, the conditions that we consider were treated by Beligiannis in \cite[Section 7]{beligiannis2} and similar conditions have been studied, for instance, by Fossum-Griffith-Reiten \cite{reiten}, Löfwall \cite{lofwall}, Minamoto-Yamaura \cite{minamoto_yamaura} and Palmer-Roos \cite{palmer_roos}. The term perfect is motivated by the connection with the concept of a perfect bimodule (see Definition \ref{perfect_bimodule}, Lemma \ref{tensor_is_perfect}) in the sense of \cite{perfect}.

\subsection{Perfect endofunctors} We begin with the definition of a perfect endofunctor. Throughout, $\mathcal{B}$ denotes an abelian category (with enough projectives).  
\begin{defn} 
\label{perfect_functor}
Let $\mathsf{F}\colon\mathcal{B}\rightarrow \mathcal{B}$ be a right exact endofunctor. We consider the following condition on $\mathsf{F}$: 
\begin{align*} \label{R}{\tag{$\textbf{R}$}}
    \mathbb{L}_i\mathsf{F}^j(\mathsf{F}(P))=0 \text{ for all } i,j\geq 1 \text{ and every } P\in\Proj\mathcal{B}.
\end{align*}
We say that $\mathsf{F}$ is \emph{perfect} if \textnormal{(\ref{R})} holds and in addition the following are satisfied: \\ 
(i) there is $n\geq 0$ such that $\mathbb{L}_p\mathsf{F}^q=0$ for all $p,q\geq 1$ with $p+q\geq n+1$. \\ 
(ii) $\pd{_{\mathcal{B}}\mathsf{F}(P)}<\infty$ for every $P\in\Proj\mathcal{B}$. 
\end{defn}
 
The following lemma gives an equivalent characterization of condition (\ref{R}). 

\begin{lem} \textnormal{(\!\!\cite[Lemma 7.4 (a)]{beligiannis2})} \label{equivalent_to_condition_R}
    A right exact endofunctor $\mathsf{F}\colon\mathcal{B}\rightarrow \mathcal{B}$ satisfies condition \textnormal{(\ref{R})} if and only if the following condition is satisfied:
    \begin{align*}
        \mathbb{L}_i\mathsf{F}(\mathsf{F}^j(P))=0 \text{ for all } i,j\geq 1 \text{ and every }P\in\Proj\mathcal{B}.
    \end{align*}
\end{lem}
\begin{proof}
    Assume that condition (\ref{R}) is satisfied. Consider a projective resolution of $\mathsf{F}(P)$, for some $P\in\Proj \mathcal{B}$
    \[
    \cdots\rightarrow P_k\rightarrow\cdots\rightarrow P_0\rightarrow\mathsf{F}(P)\rightarrow 0.
    \]
    Applying the functor $\mathsf{F}$ gives the following complex:
    \[
    \cdots\rightarrow \mathsf{F}(P_k)\rightarrow\cdots\rightarrow\mathsf{F}(P_0)\rightarrow \mathsf{F}^2(P)\rightarrow 0,    \]
    whose homology computes $\mathbb{L}_i\mathsf{F}(\mathsf{F}(P))$ for $i\geq 1$, which is 0 by assumption. Therefore we may apply $\mathsf{F}$ to the above, which gives the following complex: 
    \[
    \cdots\rightarrow \mathsf{F}^2(P_k)\rightarrow\cdots\rightarrow\mathsf{F}^2(P_0)\rightarrow \mathsf{F}^3(P)\rightarrow 0,
    \]
    whose homology computes $\mathbb{L}_i\mathsf{F}(\mathsf{F}^2(P))$ for $i\geq 1$, which again is $0$ by assumption. However, the latter resolution is the same as applying $\mathsf{F}^2$ to the projective resolution of $\mathsf{F}(P)$ and this computes $\mathbb{L}_i\mathsf{F}^2(\mathsf{F}(P))$ for $i\geq 1$, i.e.\ $\mathbb{L}_i\mathsf{F}^2(\mathsf{F}(P))\cong 0$ for $i\geq 1$. It follows similarly that $\mathbb{L}_i\mathsf{F}^j(\mathsf{F}(P))\cong 0$ for $j\geq 3$ and all $i\geq 1$. Similarly one can prove the inverse implication. 
\end{proof}

A proof of the following technical lemma can be found in \cite[Lemma 6.5]{graded_injective_generation}; compare with \cite[Lemma 4.2]{perfect}.

\begin{lem} \label{F_projective}
    Consider a right exact endofunctor $\mathsf{F}\colon\mathcal{B}\rightarrow\mathcal{B}$ that satisfies condition \textnormal{(\ref{R})} and let $X$ be an object of $\mathcal{B}$. The following are equivalent: 
    \begin{itemize}
        \item[(i)] $\mathbb{L}_i\mathsf{F}(\mathsf{F}^j(X))=0$ for all $i\geq 1$ and $j\geq 0$.
        \item[(ii)] $\mathbb{L}_i\mathsf{F}^s(\mathsf{F}^j(X))=0$ for all $i,s\geq 1$ and $j\geq 0$. 
        \item[(iii)] $\mathbb{L}_i\mathsf{F}^j(X)=0$ for all $i,j\geq 1$. \end{itemize}
\end{lem}

\begin{defn}
    Consider a right exact endofunctor $\mathsf{F}\colon\mathcal{B}\rightarrow\mathcal{B}$ that satisfies condition (\ref{R}). An object $X$ of $\mathcal{B}$ is called $\mathsf{F}$-\emph{projective} if it satisfies any of the equivalent conditions of Lemma \ref{F_projective}.
\end{defn}

\begin{cor} \label{powers_of_F_proj}
    Consider a right exact endofunctor $\mathsf{F}\colon\mathcal{B}\rightarrow \mathcal{B}$ that satisfies condition \textnormal{(\ref{R})}. If $X\in\mathcal{B}$ is $\mathsf{F}$-projective, then $\mathsf{F}^j(X)$ is $\mathsf{F}$-projective for every $j\geq 1$.  
\end{cor}

\begin{lem} \label{eP_are_proj}
    Let $(\mathcal{B},\mathcal{A},\mathsf{i},\mathsf{e},\mathsf{l})$ be a cleft extension of abelian categories. If the functor $\mathsf{F}$ satisfies condition \textnormal{(\ref{R})}, then $\mathsf{e}(P)$ is $\mathsf{F}$-projective for every $P\in\Proj\mathcal{A}$. 
\end{lem}
\begin{proof}
    Since every projective object in $\mathcal{A}$ is a direct summand of $\mathsf{l}(Q)$ for some projective object $Q$ of $\mathcal{B}$, it is enough to show that $\mathbb{L}_i\mathsf{F}^j(\mathsf{el}(Q))=0$ for every projective object $Q$ of $\mathcal{B}$. But $\mathsf{el}(Q)\cong Q\oplus \mathsf{F}(Q)$, from which the claim follows. 
\end{proof}

\begin{lem} \label{basic_properties_of_perfect_endofunctor_on_cleft}
    Let $(\mathcal{B},\mathcal{A},\mathsf{i},\mathsf{e},\mathsf{l})$ be a cleft extension of abelian categories. The following hold:
    \begin{itemize}
        \item[(i)] If $\mathbb{L}_i\mathsf{F}=0$ for $i>\!\! >0$, then $\mathbb{L}_i\mathsf{l}=0$ for $i>\!\!>0$. 
        \item[(ii)] If $\pd{_{\mathcal{B}}\mathsf{F}(Q)}<\infty$ for every $Q\in\Proj\mathcal{B}$, then $\pd{_{\mathcal{B}}\mathsf{e}(P)}<\infty$ for all $P\in\Proj\mathcal{A}$. 
    \end{itemize}
    In particular, the above apply when $\mathsf{F}$ is perfect.
\end{lem}
\begin{proof}
    (i) By Lemma \ref{basic_homological_properties_of_cleft}(ii) we have $\mathbb{L}_i\mathsf{F}\cong \mathsf{e}\mathbb{L}_i\mathsf{l}$ so by the assumption of $\mathsf{F}$, it follows that $\mathsf{e}\mathbb{L}_i\mathsf{l}=0$ for $i>\!\!>0$. The result follows by the fact that $\mathsf{e}$ is faithful.

    (ii) Since every projective object in $\mathcal{A}$ is a direct summand of $\mathsf{l}(Q)$ for projective object $Q$ of $\mathcal{B}$, it is enough to show that $\pd{_{\mathcal{B}} \mathsf{el}(Q)}<\infty$. But $\mathsf{el}(Q)\cong Q\oplus \mathsf{F}(Q)$ and by assumption we have $\pd{_{\mathcal{B}}\mathsf{F}(Q)}<\infty$. 
\end{proof}

\subsection{Coperfect endofunctors} Here we define the dual notion to a perfect endofunctor. Throughout, $\mathcal{B}$ denotes an abelian category (with enough injectives).

\begin{defn}  \label{coperfect_functor} 
Let $\mathsf{F}'\colon\mathcal{B}\rightarrow\mathcal{B}$ be a left exact endofunctor. We consider the following condition on $\mathsf{F}'$:
\begin{align*} \label{DR}{\tag{\textbf{R'}}}
    \mathbb{R}^i\mathsf{F}'^j(\mathsf{F}'(I))=0 \text{ for all } i,j\geq 1 \text{ and every } I\in\Inj\mathcal{B}.
\end{align*}
We say that $\mathsf{F}'$ is \emph{coperfect} if (\ref{DR}) holds and in addition the following are satisfied:  
(i) there is $n'\geq 0$ such that $\mathbb{R}^p\mathsf{F}'^q=0$ for all $p,q\geq 1$ with $p+q\geq n'+1$.  \\ 
(ii) $\id{_{\mathcal{B}}\mathsf{F}'(I)}<\infty$ for all $I\in\Inj\mathcal{B}$. 
\end{defn}

In the following lemma we provide an equivalent characterization of condition (\ref{DR}), which is dual to Lemma \ref{equivalent_to_condition_R}.

\begin{lem} \label{equivalent_to_condition_DR}
    A left exact endofunctor $\mathsf{F}'\colon\mathcal{B}\rightarrow\mathcal{B}$ satisfies \textnormal{(\ref{DR})} if and only if the following condition is satisfied: 
    \begin{align*}
        \mathbb{R}^i\mathsf{F}'(\mathsf{F}'^j(I))=0 \text{ for all } i,j\geq 1 \text{ and every } I\in\Inj\mathcal{B}.
    \end{align*}
\end{lem}

The following is the dual of Lemma \ref{F_projective}.

\begin{lem} \label{F'_injective}
    Consider a left exact endofunctor $\mathsf{F}'\colon\mathcal{B}\rightarrow\mathcal{B}$ that satisfies condition \textnormal{(\ref{DR})} and let $X$ be an object of $\mathcal{B}$. The following are equivalent: 
    \begin{itemize}
        \item[(i)] $\mathbb{R}^i\mathsf{F}'(\mathsf{F}'^j(X))=0$ for all $i\geq 1$ and $j\geq 0$. 
        \item[(ii)] $\mathbb{R}^i\mathsf{F}'^s(\mathsf{F}'^j(X))=0$ for all $i,s\geq 1$ and $j\geq 0$. 
        \item[(iii)] $\mathbb{R}^i\mathsf{F}'^j(X)=0$ for all $i,j\geq 1$.
    \end{itemize}
\end{lem}

\begin{defn}
    Let $\mathsf{F}'\colon\mathcal{B}\rightarrow\mathcal{B}$ be a left exact endofunctor that satisfies condition (\ref{DR}). An object $X$ of $\mathcal{B}$ is called $\mathsf{F}'$-\emph{injective} if it satisfies any of the equivalent conditions of Lemma \ref{F'_injective}.
\end{defn}

\begin{cor}
    Consider a left exact endofunctor $\mathsf{F}'\colon\mathcal{B}\rightarrow \mathcal{B}$ that satisfies condition \textnormal{(\ref{DR})}. If $X\in\mathcal{B}$ is $\mathsf{F}'$-injective, then $\mathsf{F}'^j(X)$ is $\mathsf{F}'$-injective for every $j\geq 1$. 
\end{cor}

The following is the dual of Lemma \ref{eP_are_proj}. 

\begin{lem} \label{eI_are_inj}
    Let $(\mathcal{B},\mathcal{A},\mathsf{i},\mathsf{e},\mathsf{r})$ be a cleft coextension of abelian categories. If the functor $\mathsf{F}'$ satisfies condition \textnormal{(\ref{DR})}, then $\mathsf{e}(I)$ is $\mathsf{F}'$-injective for every $I\in\Inj\mathcal{A}$.
\end{lem}

The following is the dual to Lemma \ref{basic_properties_of_perfect_endofunctor_on_cleft}. 

\begin{lem} \label{basic_properties_of_coperfect_endofunctor_on_cleft}
     Let $(\mathcal{B},\mathcal{A},\mathsf{i},\mathsf{e},\mathsf{r})$ be a cleft coextension of abelian categories. The following hold:
    \begin{itemize}
        \item[(i)] If $\mathbb{R}^i\mathsf{F}'=0$ for $i>\!\! >0$, then $\mathbb{R}^i\mathsf{r}=0$ for $i>\!\!>0$. 
        \item[(ii)] If $\id{_{\mathcal{B}}\mathsf{F}'(J)}<\infty$ for every $J\in\Inj\mathcal{B}$, then $\id{_{\mathcal{B}}\mathsf{e}(I)}<\infty$ for all $I\in\Inj\mathcal{A}$. 
    \end{itemize}
    In particular, the above apply when $\mathsf{F}'$ is coperfect. 
\end{lem}

\subsection{Perfect endofunctors for module categories} We will now provide a concrete description of perfect endofunctors for module categories. For this, we consider the notion of a perfect bimodule, which was defined in \cite{perfect} over two-sided Noetherian rings. We slightly adjust the definition, in order to work over any ring. For simplicity given a $\Gamma$-bimodule $M$, we will be writing $M^{\otimes j}$ instead of $M^{\otimes_{\Gamma}j}$.

\begin{defn} \label{perfect_bimodule}
    Let $\Gamma$ be a ring and $M$ a $\Gamma$-bimodule. We say that $M$ is \emph{perfect} if the following conditions are satisfied: 
    \begin{itemize}
\item[(i)] $\fd{_{\Gamma}M}<\infty$, 

\item[(ii)] $\pd{M_{\Gamma}}<\infty$, and
 
\item[(iii)] $\Tor_i^{\Gamma}(M,M^{\otimes j})=0$ for all $i,j\geq 1$.
\end{itemize}
\end{defn}


In \cite[Definition 4.4]{perfect}, the conditions are the same, except for (i) which requires $\pd{_{\Gamma}M}<\infty$. If, however, $\Gamma$ is two-sided Noetherian and $M$ is finitely generated on both sides (which is the setting of \cite{perfect}), then $\pd {_{\Gamma}M}=\fd{_{\Gamma}M}$ and so the two notions agree. We collect some elementary properties of perfect bimodules in the following lemma. 

\begin{lem} \textnormal{(\!\!\cite[Corollary 4.3, Lemma 4.5]{perfect})} \label{easy_properties_of_perfect_bimodules}
    Let $\Gamma$ be a ring and $M$ a perfect $\Gamma$-bimodule. The following hold: 
    \begin{itemize}
        \item[(i)] $\mathsf{Tor}_i^{\Gamma}(M^{\otimes s},M^{\otimes j})=0$ for all $i,j,s\geq 1$. 
        \item[(ii)] $\fd{_{\Gamma}M^{\otimes j}}<\infty$ for all $j\geq 1$. 
        \item[(iii)] $\pd{M^{\otimes j}_{\Gamma}}<\infty$ for all $j\geq 1$. 
    \end{itemize}
In particular, the $\Gamma$-bimodule $M^{\otimes i}$ is perfect for every $i\geq 2$. 
\end{lem}
\begin{proof}
    (i) This is a consequence of Lemma \ref{equivalent_to_condition_R} for $\mathsf{F}=-\otimes_{\Gamma}M\colon\Mod\Gamma\rightarrow\Mod\Gamma$.

    (ii) Assume that $\fd {_{\Gamma}M}=n$ and consider a flat resolution of $M$ of length $n$:
    \[
     0\rightarrow F_n\rightarrow\cdots \rightarrow F_0\rightarrow M\rightarrow 0.
    \]
    Since $\mathsf{Tor}_i^{\Gamma}(M,M)=0$ for all $i\geq 1$, applying $M\otimes_{\Gamma}-$ gives an exact sequence 
    \[
     0\rightarrow M\otimes_{\Gamma} F_n\rightarrow \cdots\rightarrow M\otimes_{\Gamma} F_0\rightarrow M\otimes_{\Gamma} M\rightarrow 0
    \]
    and $\fd{_{\Gamma} M\otimes_{\Gamma} F_k}\leq \fd{_{\Gamma}M}$ for every $k$. Therefore $\fd{_{\Gamma}M\otimes_{\Gamma} M}\leq 2\fd {_{\Gamma}M}$. Inductively we derive that $\fd {_{\Gamma}M^{\otimes j}}\leq j\fd {_{\Gamma}M}$ for every $j\geq 1$. 

    (iii) This is proved similarly to (ii). 
\end{proof}

\begin{lem} \label{tensor_is_perfect}
    Let $\Gamma$ be a ring and $M$ be a $\Gamma$-bimodule. The following are equivalent: 
    \begin{itemize}
        \item[(i)] The functor $-\otimes_{\Gamma}M\colon\Mod\Gamma\rightarrow \Mod\Gamma$ is perfect and nilpotent. 
        \item[(ii)] The bimodule $M$ is perfect and nilpotent. 
    \end{itemize}
\end{lem}
\begin{proof}
First of all, we notice that condition (\ref{R}) for the functor $\mathsf{F}=-\otimes_{\Gamma}M$ is equivalent to $\mathsf{Tor}^{\Gamma}_i(M,M^{\otimes j})=0$ for $i,j\geq 1$. 

(i) $\Longrightarrow$ (ii): Assume that $M$ is perfect and nilpotent. By Lemma \ref{easy_properties_of_perfect_bimodules}(ii), we have that $\fd{_{\Gamma}M^{\otimes j}}<\infty$ for all $j\geq 1$. Since $M$ is nilpotent, it follows that
\[
\mathsf{max}\{\fd{_{\Gamma}M^{\otimes j}}, j\geq 1\}=n_{M}<\infty.
\]
Therefore, for $i> n_{M}$ we have $\mathsf{Tor}^{\Gamma}_i(-,M^{\otimes j})=0$ for all $j\geq 1$, so the second condition of the Definition \ref{perfect_functor} is satisfied. Lastly, since $\pd{M_{\Gamma}}<\infty$, it follows that $-\otimes_{\Gamma}M$ maps projective modules to modules of finite projective dimension. 

(ii) $\Longrightarrow$ (i): Assume that the functor $-\otimes_{\Gamma}M$ is perfect. Then, there is $n\geq 0$ such that $\mathsf{Tor}^{\Gamma}_p(-,M^{\otimes q})=0$ for all $p,q\geq 1$ with $p+q\geq n+1$. In particular, for $q=1$ we derive that $\mathsf{Tor}^{\Gamma}_i(-,M)=0$ for all $i\geq n$, i.e.\ $\fd{_{\Gamma}M}<\infty$. Moreover, since $-\otimes_{\Gamma}M$ maps projective modules to modules of finite projective dimension, it follows that $\pd{M_{\Gamma}}<\infty$. \end{proof}

One may now expect the definition of a coperfect bimodule. Such a definition would capture the coperfectness of the functor $\mathsf{Hom}_{\Gamma}(M,-)\colon\Mod\Gamma\rightarrow \Mod\Gamma$. Below we explain that if $M$ is perfect and nilpotent, the latter already holds. In preparation for this we present the following auxiliary lemma.

\begin{lem} \label{ext_computation}
    Let $\Gamma$ be a ring and $M$ a $\Gamma$-bimodule satisfying $\mathsf{Tor}_i^{\Gamma}(M,M^{\otimes j})=0$ for all $i,j\geq 1$. Then, for every (right) injective $\Gamma$-module $I$, we have 
    \[
    \mathsf{Ext}_{\Gamma}^i(M^{\otimes j},\mathsf{Hom}_{\Gamma}(M,I))=0
    \]
    for all $i,j\geq 1$.
\end{lem}
\begin{proof}
    By \cite[Chapter VI, Proposition 5.1]{cartan_eilenberg} we have 
    \[
    \mathsf{Ext}_{\Gamma}^i(M^{\otimes j},\mathsf{Hom}_{\Gamma}(M,I))\cong \mathsf{Hom}_{\Gamma}(\mathsf{Tor}_i^{\Gamma}(M^{\otimes j},M),I)=0,
    \]
    where the last equality holds by Lemma \ref{easy_properties_of_perfect_bimodules}(i) and the assumption.
\end{proof}

The following shows that the notion of a coperfect bimodule is not necessary and it is also important for later use. We do not know whether a categorical version holds (i.e. for general abelian categories). 

\begin{lem} \label{perfect_implies_coperfect}
Let $\Gamma$ be a ring and $M$ a $\Gamma$-bimodule. If $-\otimes_{\Gamma}M\colon\Mod\Gamma\rightarrow\Mod \Gamma$ is perfect and nilpotent, then the functor $\mathsf{Hom}_{\Gamma}(M,-)\colon\Mod\Gamma\rightarrow\Mod\Gamma$ is coperfect and nilpotent. 
\end{lem}
\begin{proof}
Condition (\ref{DR}) is translated to $\mathsf{Ext}^i_{\Gamma}(M^{\otimes j},\mathsf{Hom}_{\Gamma}(M,I))=0$ for all $i,j\geq 1$ and every injective $\Gamma$-module $I$, which holds by Lemma \ref{ext_computation}. By Lemma \ref{easy_properties_of_perfect_bimodules}, for every $j\geq 1$ we have $\pd{M^{\otimes j}_{\Gamma}}<\infty$. Therefore, since $M$ is nilpotent, there is an integer $n_{M}'$ such that $\mathsf{sup}\{\pd M^{\otimes j}_{\Gamma}, j\geq 1\}=n_{M}'<\infty$. We conclude that for $i> n_{M}'$ we have $\mathsf{Ext}_{\Gamma}^i(M^{\otimes j},-)=0$ for all $j\geq 1$ and therefore the second condition of Definition \ref{coperfect_functor} is satisfied. By \cite[Chapter VI, Proposition 5.1]{cartan_eilenberg}, for every $X\in\Mod\Gamma$, $I\in\Inj\Gamma$ and $i\geq 1$ we have 
\[
\mathsf{Ext}^i_{\Gamma}(X,\mathsf{Hom}_{\Gamma}(M,I))\cong \mathsf{Hom}_{\Gamma}(\mathsf{Tor}_i^{\Gamma}(X,M),I),
\]
which shows that $\id\ \!{\mathsf{Hom}_{\Gamma}(M,I)_{\Gamma}}\leq \fd {_{\Gamma}M}<\infty$. In particular, the last condition of Definition \ref{coperfect_functor} is satisfied. Lastly, $\mathsf{Hom}_{\Gamma}(M,-)^n\simeq \mathsf{Hom}_{\Gamma}(M^{\otimes n},-)$ and since $M$ is nilpotent, it follows that $\mathsf{Hom}_{\Gamma}(M,-)$ is nilpotent.
\end{proof}

\section{Gorenstein abelian categories}

\subsection{Comparing Gorenstein categories in cleft extensions} Let $\mathcal{A}$ be an abelian category with enough projective and injective objects. Consider the following invariants of $\mathcal{A}$:
\[
\mathsf{spli}\mathcal{A}=\mathsf{sup}\{\pd{_{\mathcal{A}} X}\ |\ X\in\Inj\mathcal{A}\} \  \text{ and }  \  \mathsf{silp}\mathcal{A}=\mathsf{sup}\{\id{_{\mathcal{A}} X}\ |\ X\in\Proj\mathcal{A}\}.
\] 
We recall the following definition which is due to Beligiannis-Reiten \cite{gorenstein}.

\begin{defn} (\!\!\cite{gorenstein})
    We say that $\mathcal{A}$ is \emph{Gorenstein} if $\mathsf{spli}\mathcal{A}<\infty$ and $\mathsf{silp}\mathcal{A}<\infty$. 
\end{defn}

The above concept generalises the notion of an Iwanaga-Gorenstein ring, see Subsection 4.3. We recall a fundamental property of the invariants $\mathsf{spli}\mathcal{A}$ and $\mathsf{silp}\mathcal{A}$. 

\begin{lem} \label{projective_and_injective_dimensions}
    Let $X$ be an object of $\mathcal{A}$. The following hold: 
    \begin{itemize}
        \item[(i)] If $\pd{_{\mathcal{A}}X}<\infty$, then $\id{_{\mathcal{A}}X} \leq \mathsf{silp}\mathcal{A}$. 
        \item[(ii)] If $\id{_{\mathcal{A}} X}<\infty$, then $\pd{_{\mathcal{A}} X}\leq \mathsf{spli}\mathcal{A}$.
    \end{itemize}
\end{lem}
\begin{proof}
    Both claims are proved similarly, so we only prove (i). By definition the claim holds if $\pd{_{\mathcal{A}}X}=0$. We proceed by induction: if $\pd{_{\mathcal{A}}X}=n$, then there is a short exact sequence $0\rightarrow K\rightarrow P\rightarrow X\rightarrow 0$ in $\mathcal{A}$ with $P$ a projective object and $\pd{_{\mathcal{A}} K}=n-1$. Then $\id {_{\mathcal{A}}P}\leq\mathsf{silp}\mathcal{A}$ and $\id{_{\mathcal{A}}K}\leq \mathsf{silp}\mathcal{A}$ by the induction hypothesis, thus also $\id{_{\mathcal{A}}X}\leq \mathsf{silp}\mathcal{A}$. 
\end{proof}

We will now compare Gorensteinness for abelian categories in a cleft extension that is also a coextension such that the associated endofunctors satisfy perfectness assumptions. We begin by comparing $\mathsf{silp}$ in a cleft extension. 

\begin{prop} \label{one_way}
    Let $(\mathcal{B},\mathcal{A},\mathsf{i},\mathsf{e},\mathsf{l})$ be a cleft extension of abelian categories. If $\mathsf{F}$ is perfect and nilpotent, then
    \[
    \mathsf{silp}\mathcal{B}-n+1\leq \mathsf{silp}\mathcal{A}\leq \mathsf{silp}\mathcal{B}+n+s,
    \]
    where $n$ is an integer as in Definition \textnormal{\ref{perfect_functor}} and $s$ is so that $\mathsf{F}^s=0$. In particular, $\mathsf{silp}\mathcal{B}<\infty$ if and only if $\mathsf{silp}\mathcal{A}<\infty$. 
\end{prop}
\begin{proof} Let us show in (i) the right-hand side of the inequality and in (ii) the left-hand side of the inequality. 

(i) By Lemma \ref{projectives_in_cleft}, every projective object in $\mathcal{A}$ is a direct summand of $\mathsf{l}(P)$ for some projective object $P$ of $\mathcal{B}$. Therefore, it is enough to show that $\id{_{\mathcal{A}}\mathsf{l}(P)}\leq \mathsf{silp}\mathcal{B}+n+s$ for every $P\in\Proj\mathcal{B}$. Let $X$ be an object of $\mathcal{A}$ and consider an exact sequence:
    \begin{align} \label{exact_sequence_1}
        0\rightarrow X'\rightarrow P_{n-2}\rightarrow\cdots\rightarrow P_0\rightarrow X\rightarrow 0,
    \end{align}
    where $n$ is such that $\mathbb{L}_p\mathsf{F}^q=0$ for all $p,q\geq 1$ with $p+q\geq n+1$. If $n=1$, we set $X'=X$. Applying $\mathsf{e}$ to (\ref{exact_sequence_1}) gives the following exact sequence: 
    \[
    0\rightarrow \mathsf{e}(X')\rightarrow \mathsf{e}(P_{n-2})\rightarrow \cdots\rightarrow \mathsf{e}(P_0)\rightarrow \mathsf{e}(X)\rightarrow 0.
    \]
    By Lemma \ref{eP_are_proj}, we know that the objects $\mathsf{e}(P_{n-2}),\dots,\mathsf{e}(P_0)$ are $\mathsf{F}$-projective and therefore by dimension shift we obtain the following isomorphism for every $i,j\geq 1$: $\mathbb{L}_i\mathsf{F}^j(\mathsf{e}(X'))\cong\mathbb{L}_{i+n-1}\mathsf{F}^j(\mathsf{e}(X))$. By the choice of $n$, the latter is $0$, meaning that $\mathsf{e}(X')$ is $\mathsf{F}$-projective. Therefore, by Lemma \ref{F_projective}, it follows that $\mathsf{F}^j\mathsf{e}(X')$ is $\mathsf{F}$-projective for every $j\geq 1$. We thus have the following isomorphisms for $k\geq 1$: 
    \begin{align*}
        \mathsf{Ext}^k_{\mathcal{A}}(\mathsf{lF}^j\mathsf{e}(X'),\mathsf{l}(P))& \cong \mathsf{Ext}_{\mathcal{B}}^k(\mathsf{F}^j\mathsf{e}(X'),\mathsf{el}(P))& \text{ Lemma \ref{basic_homological_properties_of_cleft}} \\ 
                           & \cong \mathsf{Ext}_{\mathcal{B}}^k(\mathsf{F}^j\mathsf{e}(X'),P)\oplus \mathsf{Ext}_{\mathcal{B}}^k(\mathsf{F}^j\mathsf{e}(X'),\mathsf{F}(P)) & \mathsf{el}\simeq \mathsf{Id}_{\mathcal{B}}\oplus \mathsf{F}
    \end{align*}
    We have that $\id{_{\mathcal{B}}P}\leq \mathsf{silp}\mathcal{B}$. Moreover, by assumption, the object $\mathsf{F}(P)$ has finite projective dimension and so by Lemma \ref{projective_and_injective_dimensions}, we have that $\id{_{\mathcal{B}}\mathsf{F}(P)}\leq \mathsf{silp}\mathcal{B}$. From the above, we conclude that  
    \begin{equation}
        \mathsf{Ext}^k_{\mathcal{A}}(\mathsf{lF}^j\mathsf{e}(X'),\mathsf{l}(P))\cong 0 \text{ for every }k\geq \mathsf{silp}\mathcal{B}+1 \text{ and every } j\geq 0.
    \end{equation}
    Consider now the following sequence of short exact sequences in $\mathcal{A}$: 
    \[
    0\rightarrow \mathsf{G}(X')\rightarrow \mathsf{le}(X')\rightarrow X'\rightarrow 0, \ \  0\rightarrow \mathsf{G}^2(X')\rightarrow \mathsf{lFe}(X')\rightarrow \mathsf{G}(X')\rightarrow 0, \ \  \dots
    \]
    Since $\mathsf{F}$ is nilpotent, there is $s$ such that $\mathsf{F}^s=0$ and so for every $k\geq \mathsf{silp}\mathcal{B}+s$, using (4.2), we obtain the following isomorphisms:
    \begin{align*} 
          \mathsf{Ext}_{\mathcal{A}}^k(X',\mathsf{l}(P))&\cong\mathsf{Ext}^{k-1}_{\mathcal{A}}(\mathsf{G}(X'),\mathsf{l}(P)) \\ 
                                 &\cong \mathsf{Ext}_{\mathcal{A}}^{k-2}(\mathsf{G}^2(X'),\mathsf{l}(P)) \\
                                 & \ \ \vdots  \\
                                 &\cong \mathsf{Ext}_{\mathcal{A}}^{k-s+1}(\mathsf{G}^{s-1}(X'),\mathsf{l}(P)) \\ 
                                 &\cong \mathsf{Ext}_{\mathcal{A}}^{k-s+1}(\mathsf{lF}^{s-1}\mathsf{e}(X'),\mathsf{l}(P))\cong 0.
    \end{align*}
    By returning to (\ref{exact_sequence_1}) we conclude that $\mathsf{Ext}_{\mathcal{A}}^k(X,\mathsf{l}(P))\cong 0$ for every $k\geq \mathsf{silp}\mathcal{B}+n+s-1$ and therefore $\id{_{\mathcal{A}}\mathsf{l}(P)}\leq \mathsf{silp}\mathcal{B}+n+s$.

    (ii) Let $P$ be a projective object of $\mathcal{B}$ and $X$ any object of $\mathcal{B}$. Consider an exact sequence as below 
    \begin{align}
        0\rightarrow X'\rightarrow P_{n-2}\rightarrow \cdots\rightarrow P_0\rightarrow X\rightarrow 0,
    \end{align}
    where $n$ is such that $\mathbb{L}_p\mathsf{F}^q=0$ for $p+q\geq n+1$. If $n=1$, we set $X'=X$. Then, by the choice of $n$, it follows that $X'$ is $\mathsf{F}$-projective, so we infer that 
    \begin{align*}
        \mathsf{Ext}^k_{\mathcal{A}}(\mathsf{l}(X'),\mathsf{l}(P))& \cong \mathsf{Ext}^k_{\mathcal{B}}(X',\mathsf{el}(P)) & \text{ Lemma \ref{basic_homological_properties_of_cleft}} \\ 
        & \cong \mathsf{Ext}^k_{\mathcal{B}}(X',P)\oplus\mathsf{Ext}^k_{\mathcal{B}}(X',\mathsf{F}(P)) & \mathsf{el}\simeq \mathsf{Id}_{\mathcal{B}}\oplus \mathsf{F}
    \end{align*}
    Since $\mathsf{l}(P)$ is a projective object of $\mathcal{A}$, it follows that $\id{_{\mathcal{A}}\mathsf{l}(P)}\leq \mathsf{silp}\mathcal{A}$. In particular, it follows that $\mathsf{Ext}^i_{\mathcal{A}}(\mathsf{l}(X'),\mathsf{l}(P))\cong 0$ for $k\geq \mathsf{silp}\mathcal{A}+1$, so $\mathsf{Ext}_{\mathcal{B}}^k(X',P)\cong 0$ for $k\geq \mathsf{silp}\mathcal{A}+1$. By returning to (4.3), we obtain that $\mathsf{Ext}_{\mathcal{A}}^k(X,P)\cong 0$ for all $k\geq \mathsf{silp}\mathcal{A}+n$. This implies that $\mathsf{silp}\mathcal{B}\leq \mathsf{silp}\mathcal{A}+n-1$. 
\end{proof}

We spell out, with proof, the dual statement involving $\mathsf{spli}$ and cleft coextensions.

\begin{prop} \label{the_other_way}
    Let $(\mathcal{B},\mathcal{A},\mathsf{i},\mathsf{e},\mathsf{r})$ be a cleft coextension of abelian categories. If $\mathsf{F}'$ is coperfect and nilpotent, then
    \[
    \mathsf{spli}\mathcal{B}-n'+1\leq \mathsf{spli}\mathcal{A}\leq \mathsf{spli}\mathcal{B}+n'+s,
    \]
    where $n'$ is an integer as in Definition \textnormal{\ref{coperfect_functor}} and $s$ is so that $\mathsf{F}'^s=0$. In particular, $\mathsf{spli}\mathcal{B}<\infty$ if and only if $\mathsf{spli}\mathcal{A}<\infty$. 
\end{prop} 
\begin{proof} Let us show in (i) the right-hand side of the inequality and in (ii) the left-hand side of the inequality. 

(i) Every injective object in $\mathcal{A}$ is a direct summand of $\mathsf{r}(I)$ for some injective object $I$ of $\mathcal{B}$. Therefore, it is enough to show that $\pd{_{\mathcal{A}}\mathsf{r}(I)}\leq \mathsf{spli}\mathcal{B}+n'+s$ for every $I\in\Inj\mathcal{B}$. Let $X$ be an object of $\mathcal{A}$ and consider a resolution as follows:
\begin{align} \label{second_ses}
    0\rightarrow X\rightarrow I_0\rightarrow \cdots\rightarrow I_{n'-2}\rightarrow X'\rightarrow 0,
\end{align}
where $n'$ is such that $\mathbb{R}^p\mathsf{F}'^q=0$ for all $p,q\geq 1$ with $p+q\geq n'+1$. If $n'=1$ we set $X'=X$. Applying $\mathsf{e}$ to (\ref{second_ses}) gives the following exact sequence in $\mathcal{B}$:
\[
0\rightarrow\mathsf{e}(X)\rightarrow\mathsf{e}(I_0)\rightarrow\cdots\rightarrow \mathsf{e}(I_{n'-2})\rightarrow\mathsf{e}(X')\rightarrow 0.
\]
By Lemma \ref{eI_are_inj}, we know that the objects $\mathsf{e}(I_0),\dots,\mathsf{e}(I_{n'-2})$ are $\mathsf{F}'$-injective, so by dimension shift we have the following isomorphism for every $i\geq 1$: $\mathbb{R}^i\mathsf{F}'^j(\mathsf{e}(X'))\cong\mathbb{R}^{i+n'-1}\mathsf{F}'^j(\mathsf{e}(X))\cong 0$. By the choice of $n'$ the latter is $0$ and we conclude that $\mathsf{e}(X')$ is $\mathsf{F}'$-injective. Therefore, by Lemma \ref{F'_injective}, it follows that $\mathsf{F}'^j\mathsf{e}(X')$ is $\mathsf{F}'$-injective for every $j\geq 0$. We thus have the following isomorphisms for $k\geq 1$: 
\begin{align*}
    \mathsf{Ext}_{\mathcal{A}}^k(\mathsf{r}(I),\mathsf{r}\mathsf{F}'^j\mathsf{e}(X'))& \cong \mathsf{Ext}^k_{\mathcal{B}}(\mathsf{er}(I),\mathsf{F}'^j\mathsf{e}(X')) & \text{Lemma \ref{basic_homological_properties_of_cleft}} \\ 
    & \cong \mathsf{Ext}^k_{\mathcal{B}}(I,\mathsf{F}'^j\mathsf{e}(X'))\oplus \mathsf{Ext}^k_{\mathcal{B}}(\mathsf{F}'(I),\mathsf{F}'^j\mathsf{e}(X')) & \mathsf{er}\simeq \mathsf{Id}_{\mathcal{B}}\oplus \mathsf{F}'
\end{align*}
Since $I$ is injective, it follows that $\pd{_{\mathcal{B}} I}\leq \mathsf{spli}\mathcal{B}$. Moreover, by assumption, the object $\mathsf{F}'(I)$ has finite injective dimension and so by Lemma \ref{projective_and_injective_dimensions}, we have that $\pd{_{\mathcal{B}}\mathsf{F}'(I)}\leq \mathsf{spli}\mathcal{B}$. From the above, we conclude that 
\begin{equation}
    \mathsf{Ext}_{\mathcal{A}}^k(\mathsf{r}(I),\mathsf{r}\mathsf{F}'^j\mathsf{e}(X'))\cong 0 \text{ for every } k\geq \mathsf{spli}\mathcal{B}+1 \text{ and every } j\geq 0. 
\end{equation}
Consider now the following sequence of short exact sequences in $\mathcal{A}$: 
\[
0\rightarrow X'\rightarrow \mathsf{re}(X')\rightarrow \mathsf{G'}(X')\rightarrow 0, \  0\rightarrow \mathsf{G}'(X')\rightarrow \mathsf{rF'e}(X')\rightarrow \mathsf{G}'^2(X')\rightarrow 0, \ \dots
\]
Since $\mathsf{F}'$ is nilpotent, there is $s$ such that $\mathsf{F}'^s=0$ and so for every $k\geq \mathsf{spli}\mathcal{B}+s$, using (4.5), we obtain the following isomorphisms: 
\begin{align*}
 \mathsf{Ext}_{\mathcal{A}}^k(\mathsf{r}(I),X')&\cong\mathsf{Ext}^{k-1}_{\mathcal{A}}(\mathsf{r}(I),\mathsf{G}'(X')) \\ 
                                 &\cong \mathsf{Ext}_{\mathcal{A}}^{k-2}(\mathsf{r}(I),\mathsf{G}'^2(X')) \\
                                 & \ \ \vdots \\
                                 &\cong \mathsf{Ext}_{\mathcal{A}}^{k-s+1}(\mathsf{r}(I),\mathsf{G}'^{s-1}(X')) \\ 
                                 &\cong \mathsf{Ext}_{\mathcal{A}}^{k-s+1}(\mathsf{r}(I),\mathsf{rF'}^{s-1}\mathsf{e}(X'))\cong 0.
\end{align*}
By returning to (\ref{second_ses}) we conclude that $\mathsf{Ext}^k_{\mathcal{A}}(\mathsf{r}(I),X)\cong 0$ for every $k\geq \mathsf{spli}\mathcal{B}+s+n'-1$ and therefore $\id{_{\mathcal{A}}\mathsf{r}(I)}\leq \mathsf{spli}\mathcal{B}+s+n'$.

(ii) Let $I$ be an injective object in $\mathcal{B}$ and $X$ any object of $\mathcal{B}$. Consider an exact sequence as below
\begin{align}
    0\rightarrow X\rightarrow I_0\rightarrow\cdots\rightarrow I_{n'-2}\rightarrow X'\rightarrow 0,
\end{align}
where $n'$ is such that $\mathbb{R}^p\mathsf{F}'^q=0$ for all $p,q\geq 1$ with $p+q\geq n'+1$. If $n'=1$, we set $X'=X$. By the choice of $n'$, it follows that $X'$ is $\mathsf{F}'$-injective, so we infer that
\begin{align*}
    \mathsf{Ext}_{\mathcal{A}}^k(\mathsf{r}(I),\mathsf{r}(X')) & \cong \mathsf{Ext}_{\mathcal{B}}^k(\mathsf{er}(I),X') & \text{Lemma \ref{basic_homological_properties_of_cocleft}} \\ 
    & \cong \mathsf{Ext}_{\mathcal{B}}^k(I,X')\oplus \mathsf{Ext}_{\mathcal{B}}^k(\mathsf{F}'(I),X') & \mathsf{er}\simeq \mathsf{Id}_{\mathcal{B}}\oplus \mathsf{F}'
\end{align*}
Since $\mathsf{r}(I)$ is an injective object of $\mathcal{A}$, it follows that $\pd_{\mathcal{A}}\mathsf{r}(I)\leq \mathsf{spli}\mathcal{A}$. In particular, it follows that $\mathsf{Ext}_{\mathcal{A}}^k(\mathsf{r}(I),\mathsf{r}(X'))=0$ for $k\geq \mathsf{spli}\mathcal{A}+1$, so $\mathsf{Ext}_{\mathcal{B}}^k(I,X')=0$ for $k\geq \mathsf{spli}\mathcal{A}+1$. By returning to (4.6), we obtain that $\mathsf{Ext}^k_{\mathcal{B}}(I,X)=0$ for all $k\geq \mathsf{spli}\mathcal{A}+n'$. This implies that $\mathsf{spli}\mathcal{B}\leq \mathsf{spli}\mathcal{A}+n'-1$. 
\end{proof}

We now want to combine Proposition \ref{one_way} and Proposition \ref{the_other_way} into one single statement involving Gorenstein categories and cleft extensions. For this, we need situations where cleft extensions are also cleft coextensions. This is characterized in the following fact. 

\begin{fact}  (\!\!\cite[Remark 2.7]{beligiannis},\cite[Remark 2.5]{arrow}) \label{fact}
    A cleft extension of abelian categories $(\mathcal{B},\mathcal{A},\mathsf{i},\mathsf{e},\mathsf{l})$ is the upper part of a cleft coextension $(\mathcal{B},\mathcal{A},\mathsf{i},\mathsf{e},\mathsf{r})$ precisely when the endofunctor $\mathsf{F}$ of $\mathcal{B}$ admits a right adjoint $\mathsf{F}'$. The situation is summarised in the diagram below: 
    \[
    \begin{tikzcd}
\mathcal{B} \arrow[rr, "\mathsf{i}"] &  & \mathcal{A} \arrow[rr, "\mathsf{e}"] \arrow[ll, "\mathsf{q}"', bend right] \arrow[ll, "\mathsf{p}", dashed, bend left] &  & \mathcal{B} \arrow[ll, "\mathsf{l}"', bend right] \arrow["\mathsf{F}"', loop, distance=2em, in=125, out=55] \arrow["\mathsf{F}'"', dashed, loop, distance=2em, in=305, out=235] \arrow[ll, "\mathsf{r}", dashed, bend left]
\end{tikzcd}
    \]
    In this case, $\mathsf{F}'$ is the endofunctor of $\mathcal{B}$ induced by the cleft coextension.
\end{fact}

\begin{thm} \label{main_thm}
    Let $(\mathcal{B},\mathcal{A},\mathsf{i},\mathsf{e},\mathsf{l})$ be a cleft extension of abelian categories. If $\mathsf{F}$ is perfect, nilpotent and admits a right adjoint that is coperfect, then $\mathcal{A}$ is Gorenstein if and only if $\mathcal{B}$ is Gorenstein.  
\end{thm}
\begin{proof}
    By Fact \ref{fact}, since $\mathsf{F}$ admits a right adjoint $\mathsf{F}'$, the cleft extension $(\mathcal{B},\mathcal{A},\mathsf{i},\mathsf{e},\mathsf{l})$ is the upper part of a cleft coextension. Moreover, since $(\mathsf{F},\mathsf{F}')$ is an adjoint pair of functors, $\mathsf{F}$ is nilpotent if and only if $\mathsf{F}'$ is nilpotent. Then, by Proposition \ref{one_way}, we have that $\mathsf{silp}\mathcal{B}<\infty$ if and only if $\mathsf{silp}\mathcal{A}<\infty$ and by Proposition \ref{the_other_way}, we have that $\mathsf{spli}\mathcal{B}<\infty$ if and only if $\mathsf{spli}\mathcal{A}<\infty$. We conclude that $\mathcal{A}$ is Gorenstein if and only if $\mathcal{B}$ is Gorenstein. 
\end{proof}

\subsection{Eventually homological isomorphisms}

In \cite{arrow2}, Erdmann-Psaroudakis-Solberg studied Gorenstein categories appearing in a cleft extension with the aim to study Gorensteinness for the arrow removal operation \cite{arrow}. We explain their approach and prove similar results, using the theory of perfect endofunctors. First, recall a concept of \cite[Section 3]{PSS}.

\begin{defn}
    Let $\mathsf{e}\colon\mathcal{A}\rightarrow \mathcal{B}$ be an exact functor of abelian categories. We say that $\mathsf{e}$ is an \emph{eventually homological isomorphism} if there is an integer $n$ such that for all $k>n$, there is an isomorphism 
    \[
    \mathsf{Ext}_{\mathcal{A}}^k(X,Y)\cong \mathsf{Ext}_{\mathcal{B}}^k(\mathsf{e}(X),\mathsf{e}(Y)).
    \]
\end{defn}

The approach of \cite{arrow2} relies on the following theorem of \cite{PSS}. 

\begin{thm} \label{psarou14_Gorenstein} \textnormal{(\!\!\cite[Theorem 4.3]{PSS})}
    Let $\mathcal{A}$ and $\mathcal{B}$ be abelian categories with enough projectives and enough injectives. Assume that there is a functor $\mathsf{e}\colon \mathcal{A}\rightarrow \mathcal{B}$ which is an essentially surjective eventually homological isomorphism. Then $\mathcal{A}$ is Gorenstein if and only if $\mathcal{B}$ is Gorenstein. 
\end{thm}

Consider the following condition that the functors $\mathsf{e}$ and $\mathsf{l}$ in a cleft extension of abelian categories $(\mathcal{B},\mathcal{A},\mathsf{i},\mathsf{e},\mathsf{l})$ might satisfy:
\begin{align*} \tag{\textbf{E}}
    \mathsf{l} \text{ is exact and } \mathsf{e} \text{ sends projectives to projectives. } 
\end{align*}
It is shown in \cite[Theorem 3.2]{arrow2} that if a cleft extension satisfies $(\textbf{E})$ and additionally $\mathsf{sup}\{\pd{_{\mathcal{A}}\mathsf{G}(X)} \ | \ X\in\mathcal{A}\}<\infty$, then the functor $\mathsf{e}$ is an eventually homological isomorphism and, in particular, it follows by Theorem \ref{psarou14_Gorenstein} that $\mathcal{A}$ is Gorenstein if and only if $\mathcal{B}$ is Gorenstein. Notice that if $(\textbf{E})$ is satisfied, then by the equivalence of functors $\mathsf{el}\simeq \mathsf{Id}_{\mathcal{B}}\oplus \mathsf{F}$, it follows that $\mathsf{F}$ is exact and sends projectives to projectives. Therefore, if condition $(\textbf{E})$ is satisfied, then the functor $\mathsf{F}$ is perfect in a trivial way. In the following theorem, besides the extra technical assumption $\mathsf{sup}\{\pd{_{\mathcal{B}}\mathsf{F}(P)}\ | \ P\in\Proj\mathcal{B}\}<\infty$ (which holds for module categories - assuming that $\mathsf{F}$ is perfect), we extend in part (a) the aforementioned result of  \cite{arrow2}, essentially by replacing condition $(\textbf{E})$ by the assumption that the functor $\mathsf{F}$ is perfect. Part (b) is inspired by \cite[Theorem 2.10]{qin2}. 

\begin{thm} \label{thm2}
    Let $(\mathcal{B},\mathcal{A},\mathsf{i},\mathsf{e},\mathsf{l})$ be a cleft extension of abelian categories such that $\mathsf{F}$ is perfect. Assume either \textnormal{(a)} or \textnormal{(b)}: 
    \begin{itemize}
    \item[(a)]
        \begin{itemize}
        \item[(i)] $\mathsf{sup}\{\pd{_{\mathcal{A}}\mathsf{G}(X)} \ | \ X\in\mathcal{A}\}=n_{\mathcal{A}}<\infty$, and 
        \item[(ii)] $\mathsf{sup}\{\pd{_{\mathcal{B}}\mathsf{F}(P)} \ |\ P\in\Proj\mathcal{B}\}=n_{\mathsf{F}}<\infty$.
        \end{itemize}
    \item[(b)] 
    \begin{itemize}
        \item[(i)] $\mathsf{F}$ is nilpotent, and 
        \item[(ii)] $\mathsf{sup}\{\pd{_{\mathcal{B}}\mathsf{F}(X)} \ | \ X \ \mathsf{F}\textnormal{-projective}\}=n_{\mathsf{F}}<\infty$. 
    \end{itemize}
    \end{itemize}
    Then the functor $\mathsf{e}$ is an eventually homological isomorphism. 
\end{thm}
\begin{proof}
    (a) Let $X,Y$ be objects in $\mathcal{A}$ and consider an exact sequence as below
    \[
    0\rightarrow X'\rightarrow P_{n-2}\rightarrow \cdots\rightarrow P_0\rightarrow X\rightarrow 0,
    \]
    where $n$ is such that $\mathbb{L}_p\mathsf{F}^q=0$ for all $p,q\geq 1$ with $p+q\geq n+1$. If $n=1$, we set $X'=X$. By the choice of $n$, it follows that $\mathsf{e}(X')$ is $\mathsf{F}$-projective. Therefore, by Lemma \ref{basic_homological_properties_of_cleft}, we have $\mathsf{Ext}^i_{\mathcal{A}}(\mathsf{le}(X'),Y)\cong \mathsf{Ext}_{\mathcal{B}}^i(\mathsf{e}(X'),\mathsf{e}(Y))$ for all $i\geq 1$. Consider the following short exact sequence in $\mathcal{A}$: 
    \[
    0\rightarrow \mathsf{G}(X')\rightarrow\mathsf{le}(X')\rightarrow X'\rightarrow 0.
    \]
    By the above and condition (ii), it follows that $\mathsf{Ext}^i_{\mathcal{A}}(\mathsf{le}(X'),Y)\cong \mathsf{Ext}^i_{\mathcal{A}}(X',Y)$ for all $i\geq n_{\mathcal{A}}+2$. Therefore, $\mathsf{Ext}^i_{\mathcal{A}}(X',Y)\cong \mathsf{Ext}_{\mathcal{B}}^i(\mathsf{e}(X'),\mathsf{e}(Y))$ for $i\geq n_{\mathcal{A}}+2$. On the other hand, $\mathsf{Ext}^{i+n-1}_{\mathcal{A}}(X,Y)\cong \mathsf{Ext}^{i}_{\mathcal{A}}(X',Y)$ for all $i\geq 1$. Moreover, if we apply $\mathsf{e}$ to the first exact sequence, we get the following exact sequence:
    \[
    0\rightarrow \mathsf{e}(X')\rightarrow \mathsf{e}(P_{n-2})\rightarrow \cdots\rightarrow \mathsf{e}(P_0)\rightarrow \mathsf{e}(X)\rightarrow 0,
    \]
    from which we derive that $\mathsf{Ext}^{i+n-1}_{\mathcal{B}}(\mathsf{e}(X),\mathsf{e}(Y))\cong \mathsf{Ext}_{\mathcal{B}}^{i}(\mathsf{e}(X'),\mathsf{e}(Y))$ for $i\geq n_{\mathsf{F}}+1$ (notice that each $\mathsf{e}(P_k)$ has projective dimension at most $n_{\mathsf{F}}$). We conclude that $\mathsf{Ext}_{\mathcal{A}}^{i+n-1}(X,Y)\cong \mathsf{Ext}_{\mathcal{A}}^{i+n-1}(\mathsf{e}(X),\mathsf{e}(Y))$ for $i\geq \mathsf{max}\{n_{\mathsf{F}}+1,n_{\mathcal{A}}+2\}$. 

    (b) Let $X,Y$ be objects of $\mathcal{A}$ and consider $X'$ as in (a). Again, by Lemma \ref{basic_homological_properties_of_cleft}, $\mathsf{Ext}^i_{\mathcal{A}}(\mathsf{le}(X'),Y)\cong \mathsf{Ext}_{\mathcal{B}}^i(\mathsf{e}(X'),\mathsf{e}(Y))$ for all $i\geq 1$. Consider the following short exact sequences in $\mathcal{A}$: 
    \[
    0\rightarrow \mathsf{G}(X')\rightarrow \mathsf{le}(X')\rightarrow X'\rightarrow 0, \ 0\rightarrow \mathsf{G}^2(X')\rightarrow \mathsf{lFe}(X')\rightarrow \mathsf{G}(X')\rightarrow 0, \ \dots
    \]
    Since $\mathsf{e}(X')$ is $\mathsf{F}$-projective, so is $\mathsf{F}^j\mathsf{e}(X')$ for every $j\geq 1$. Consequently, for $j\geq 1$, the projective dimensions of $\mathsf{lF^je}(X')$ are bounded by $n_{\mathsf{F}}$. It follows by the above short exact sequences - using the fact that $\mathsf{F}$ is nilpotent - that $\mathsf{Ext}_{\mathcal{A}}^i(\mathsf{le}(X'),Y)\cong \mathsf{Ext}_{\mathcal{A}}^i(X',Y)$ for all $i\geq n_{\mathsf{F}}+s$ where $s$ is such that $\mathsf{F}^s=0$. Combining the above with $\mathsf{Ext}_{\mathcal{A}}^{i+n-1}(X,Y)\cong \mathsf{Ext}_{\mathcal{A}}^{i}(X',Y)$ for all $i\geq 1$ and $\mathsf{Ext}_{\mathcal{B}}^{i+n-1}(\mathsf{e}(X),\mathsf{e}(Y))\cong \mathsf{Ext}_{\mathcal{B}}^{i}(\mathsf{e}(X'),\mathsf{e}(Y))$ for $i\geq n_{\mathsf{F}}+1$, as in (a), yields the result. 
\end{proof}

We now prove that in a cleft extension, under the assumptions that we usually impose, the functor $\mathsf{e}$ preserves and reflects objects with finite projective dimension. This is of central importance in our theory. 

\begin{prop} \label{preserves_and_reflects}
    Let $(\mathcal{B},\mathcal{A},\mathsf{i},\mathsf{e},\mathsf{l})$ be a cleft extension of abelian categories and assume that the functor $\mathsf{F}$ is perfect and nilpotent. Then, for every object $X\in\mathcal{A}$, the following are equivalent:
    \begin{itemize}
        \item[(i)] $\pd{_{\mathcal{A}}X}<\infty$. 
        \item[(ii)] $\pd{_{\mathcal{B}}\mathsf{e}(X)}<\infty$.
    \end{itemize}
\end{prop}
\begin{proof}
(ii) $\Longrightarrow$ (i): Let $X$ be an object of $\mathcal{A}$ such that $\pd_{\mathcal{B}}\mathsf{e}(X)<\infty$. Consider an exact sequence as below 
 \[
 0\rightarrow X'\rightarrow P_{n-2}\rightarrow \cdots\rightarrow P_1\rightarrow P_0\rightarrow X\rightarrow 0,
 \]
 where $n$ is such that $\mathbb{L}_p\mathsf{F}^q=0$ for all $p,q\geq 1$ with $p+q\geq n+1$. If $n=1$, then we set $X'=X$. Applying the functor $\mathsf{e}$ to the above shows that $\mathsf{e}(X')$ is $\mathsf{F}$-projective (see the proof of Proposition \ref{one_way}) and that $\pd{_{\mathcal{B}} \mathsf{e}(X')}<\infty$. Consequently, by Lemma \ref{F_projective}, it follows that $\mathsf{F}^j\mathsf{e}(X')$ is $\mathsf{F}$-projective for every $j\geq 0$ and also $\pd{_{\mathcal{B}}\mathsf{F}^j(X')}<\infty$ for all $j\geq 0$, using an inductive argument. Consequently, it follows by Lemma \ref{basic_homological_properties_of_cleft} that $\pd{_{\mathcal{A}}\mathsf{l}\mathsf{F}^j\mathsf{e}(X')}<\infty$ for all $j\geq 0$. Consider the short exact sequences: 
 \[
 0\rightarrow \mathsf{G}(X')\rightarrow \mathsf{le}(X')\rightarrow X'\rightarrow 0, \ 0\rightarrow \mathsf{G}^2(X')\rightarrow \mathsf{lFe}(X')\rightarrow \mathsf{G}(X')\rightarrow 0, \ \dots
 \]
 in $\mathcal{A}$. By the above and the fact that $\mathsf{F}$ is nilpotent, we derive that $\pd{_{\mathcal{A}}X'}<\infty$. By returning to the exact sequence of $X$ that we began with, we conclude that $\pd{_{\mathcal{A}}X}<\infty$.
 
(i) $\Longrightarrow$ (ii): Let $X$ be an object of $\mathcal{A}$ with finite projective dimension. Applying the functor $\mathsf{e}$ to a projective resolution of $X$ gives a resolution of $\mathsf{e}(X)$ with terms having finite projective dimension, see Lemma \ref{basic_properties_of_perfect_endofunctor_on_cleft}. 
\end{proof}

We can also obtain inequalities involving the above projective dimensions. This relies on an extra (reasonable) assumption on the functor $\mathsf{F}$, which we explain in the following remark. This is also important for later on, as we will need the functor $\mathsf{e}$ to preserve and reflect objects with finite projective dimension ``uniformly''.

\begin{rem} \label{preserves_and_reflects_bounds}
    Let $(\mathcal{B},\mathcal{A},\mathsf{i},\mathsf{e},\mathsf{l})$ be a cleft extension of abelian categories. If the functor $\mathsf{F}$ is perfect, nilpotent and $\mathsf{sup}\{\pd{_{\mathcal{B}} \mathsf{F}(P)} \ | \ P\in\Proj\mathcal{B}\}=n_{\mathsf{F}}<\infty$, then  
    \[
    \pd{_{\mathcal{B}}\mathsf{e}(X)}-n_{\mathsf{F}}\leq \pd{_{\mathcal{A}} X} \leq \pd{_{\mathcal{B}} \mathsf{e}(X)}+(m+1)n_{\mathsf{F}},
    \]
    where $\mathsf{F}^m=0$. This follows by a careful analysis of the bounds obtained at each step of the proof of Proposition \ref{preserves_and_reflects}.
\end{rem}

There is a dual result involving cleft coextensions and injective dimensions. 

\begin{prop} \label{preserves_and_reflects_inj_dim}
    Let $(\mathcal{B},\mathcal{A},\mathsf{i},\mathsf{e},\mathsf{r})$ be a cleft coextension of abelian categories and assume that the functor $\mathsf{F}'$ is coperfect and nilpotent. Then, for every object $X\in\mathcal{A}$, the following are equivalent: 
    \begin{itemize}
        \item[(i)] $\id{_{\mathcal{A}}X}<\infty$. 
        \item[(ii)] $\id{_{\mathcal{B}}\mathsf{e}(X)}<\infty$. 
    \end{itemize}
\end{prop}

\begin{rem} \label{bounds_for_injective_dimensions} 
    Let $(\mathcal{B},\mathcal{A},\mathsf{i},\mathsf{e},\mathsf{r})$ be a cleft coextension of abelian categories. If the functor $\mathsf{F}'$ is coperfect, nilpotent and $\mathsf{sup}\{\id{_{\mathcal{B}}\mathsf{F}'(I)}\ | \ I\in\Inj\mathcal{B}\}=n_{\mathsf{F}'}<\infty$, then 
    \[
    \id{_{\mathcal{B}} \mathsf{e}(X)}-n_{\mathsf{F}'}\leq \id{_{\mathcal{A}} X} \leq \id{_{\mathcal{B}} \mathsf{e}(X)}+(m+1)n_{\mathsf{F}'},
    \]
    where ${\mathsf{F}'}^m=0$. 
\end{rem}

Using the above results on the functor $\mathsf{e}$, we can make the following remark. 

\begin{rem} \label{equivalence_for_G} 
Consider a cleft extension $(\mathcal{B},\mathcal{A},\mathsf{i},\mathsf{e},\mathsf{l})$ of abelian categories. Assume that $\mathsf{F}$ is perfect, nilpotent and $\mathsf{sup}\{\pd{_{\mathcal{B}}\mathsf{F}(P)} \ | \ P\in\Proj\mathcal{B}\}<\infty$. Then,
\[
\mathsf{sup}\{\pd{_{\mathcal{A}}\mathsf{G}(X)} \ | \ X\in\mathcal{A} \}<\infty \iff \mathsf{sup}\{\pd{_{\mathcal{B}}\mathsf{F}(X)} \ | \ X\in\mathcal{B} \}<\infty.
\]
Indeed, by Remark \ref{preserves_and_reflects_bounds}, the condition $\mathsf{sup}\{\pd{_{\mathcal{A}}\mathsf{G}(X)} \ | \ X\in\mathcal{A}\}<\infty$ is equivalent to $\mathsf{sup}\{\pd{_{\mathcal{B}}\mathsf{eG}(X)}  \ | \ X\in\mathcal{A}\}<\infty$. However, by the natural isomorphism $\mathsf{eG}\simeq \mathsf{Fe}$, the latter happens if and only if $\mathsf{sup}\{\pd{_{\mathcal{B}}\mathsf{Fe}(X)} \ | \ X\in\mathcal{A}\}<\infty$. But $\mathsf{e}$ is essentially surjective and therefore the last is equivalent to $\mathsf{sup}\{\pd{_{\mathcal{B}}\mathsf{F}(X)} \ | \ X\in\mathcal{B}\}<\infty$. 
\end{rem}

The equivalent conditions of the above remark play an important role in the study of homological properties of cleft extensions and have appeared in \cite{arrow, arrow2, monomial_arrow}. They are pivotal for this paper, too. Using the above we can prove the following, which is of independent interest and a partial converse to \cite[Theorem 3.2]{arrow2}. 

\begin{prop}
    Let $(\mathcal{B},\mathcal{A},\mathsf{i},\mathsf{e},\mathsf{l})$ be a cleft extension of abelian categories which satisfies condition $\textnormal{(\textbf{E})}$. If $\mathsf{e}$ is an eventually homological isomorphism and $\mathsf{F}$ is nilpotent, then $\mathsf{sup}\{\pd{_{\mathcal{A}}\mathsf{G}(X)}\ | \ X\in\mathcal{A}\}<\infty$. 
\end{prop}
\begin{proof}
    Let $X$ be any object of $\mathcal{B}$ and $Y$ an object of $\mathcal{A}$. By the isomorphism $\mathsf{el}(X)\cong X\oplus \mathsf{F}(X)$, it follows that for every $k\geq 1$:
   \[
    \mathsf{Ext}_{\mathcal{B}}^k(\mathsf{el}(X),\mathsf{e}(Y))\cong \mathsf{Ext}_{\mathcal{B}}^k(X,\mathsf{e}(Y))\oplus\mathsf{Ext}_{\mathcal{B}}^k(\mathsf{F}(X),\mathsf{e}(Y)).
    \] 
 By assumption, there is $n$ such that $\mathsf{Ext}_{\mathcal{B}}^k(\mathsf{el}(X),\mathsf{e}(Y))\cong \mathsf{Ext}_{\mathcal{A}}^k(\mathsf{l}(X),Y)$ for all $k>n$. From Lemma \ref{basic_homological_properties_of_cleft}, we infer that $\mathsf{Ext}_{\mathcal{A}}^k(\mathsf{l}(X),Y)\cong \mathsf{Ext}_{\mathcal{B}}(X,\mathsf{e}(Y))$ for all $k\geq 1$, since $\mathsf{l}$ is exact. We thus conclude that $\mathsf{Ext}_{\mathcal{B}}^k(\mathsf{F}(X),\mathsf{e}(Y))\cong 0$ for $k>n$. But since $\mathsf{e}$ is essentially surjective, we derive that $\pd{_{\mathcal{B}}\mathsf{F}(X)}\leq n$ for every object $X$ of $\mathcal{B}$ and by Remark \ref{equivalence_for_G} we are done.
\end{proof}

\subsection{Gorenstein rings and perfect bimodules} In this subsection we apply Theorem \ref{main_thm} to cleft extensions of module categories. The key observation, which will be used for the rest of the paper, is that for module categories cleft extensions and cleft coextensions exist simultaneously. Let us begin with a brief recollection on Gorenstein rings.  

\begin{defn} (\!\!\cite[Definition 2.5]{gorenstein}) \label{gorenstein_ring_definition}
    A ring $\Lambda$ is called right (resp. left) Gorenstein if $\Mod \Lambda$ (resp. $\Lambda\lMod$) is a Gorenstein category. We say that $\Lambda$ is \emph{Iwanaga-Gorenstein} if it is left and right Gorenstein. 
\end{defn}

For a two-sided Noetherian ring, being left or right Gorenstein suffices for the ring to be both (i.e.\ to be Iwanaga-Gorenstein). Indeed, by \cite{iwanaga}, if $\Lambda$ is a two-sided Noetherian ring, then $\id{\Lambda_{\Lambda}}=\mathsf{sup}\{\fd{{_{\Lambda}}I}, I\in\Lambda\lInj\}$. One can then prove the following, see for instance \cite[Corollary 6.11]{beligiannis3}. 

\begin{prop} \label{two_sided_noetherian_gorenstein}
    For a two-sided Noetherian ring $\Lambda$, the following are equivalent: 
    \begin{itemize}
        \item[(i)] $\Lambda$ is left Gorenstein. 
        \item[(ii)] $\Lambda$ is right Gorenstein. 
        \item[(iii)] $\id{{_{\Lambda}\Lambda}}<\infty$ and $\id{\Lambda_{\Lambda}}<\infty$. 
    \end{itemize}
\end{prop}

Note that by the above, Definition \ref{gorenstein_ring_definition} agrees with the classical definition of a Gorenstein Artin algebra, due to  Auslander-Reiten \cite{auslander_reiten}, which requires $\id{{_{\Lambda}}\Lambda}<\infty$ and $\id{\Lambda_{\Lambda}}<\infty$.

We now turn our attention to cleft extensions of module categories, beginning with the following fact. 

\begin{fact} \label{fact_2}
     Let $(\mathcal{B},\mathcal{A},\mathsf{i},\mathsf{e},\mathsf{l})$ be a cleft extension of abelian categories. We know from \cite[Proposition 3.4]{beligiannis} that $\mathcal{B}\simeq \Mod \Gamma$ for some ring $\Gamma$ if and only if the functor $\mathsf{e}$ preserves coproducts and $\mathcal{A}\simeq \Mod \Lambda$ for some ring $\Lambda$. In particular, if $(\Mod \Gamma,\Mod \Lambda,\mathsf{i},\mathsf{e},\mathsf{l})$ is a cleft extension of module categories, then the functor $\mathsf{e}$ preserves coproducts. 
\end{fact}

\begin{prop} \label{cleft_of_modules_is_cocleft}
    A cleft extension $(\Mod\Gamma,\Mod \Lambda,\mathsf{i},\mathsf{e},\mathsf{l})$ of module categories is the upper part of a cleft coextension. Moreover, if $\mathsf{F}$ is perfect and nilpotent, then $\mathsf{F}'$ is coperfect and nilpotent. 
\end{prop}
\begin{proof}
    By Fact \ref{fact_2} and the equivalence $\mathsf{el}\simeq \mathsf{Id}_{\Mod \Gamma}\oplus \mathsf{F}$, it follows that the functor $\mathsf{F}$ preserves coproducts and so by Watt's theorem, there is an isomorphism $\mathsf{F}\simeq -\otimes_{\Gamma}M$ for some $\Gamma$-bimodule $M$. Therefore, $\mathsf{F}$ admits a right adjoint $\mathsf{F}'\simeq \mathsf{Hom}_{\Gamma}(M,-)$ and thus, by Fact \ref{fact}, the given cleft extension is the upper part of a cleft coextension. By Lemma \ref{tensor_is_perfect}(i), the functor $\mathsf{F}$ is perfect and nilpotent if and only if $M$ is a perfect and nilpotent $\Gamma$-bimodule. In this case, by Lemma \ref{perfect_implies_coperfect}(ii), the functor $\mathsf{F}'\simeq \mathsf{Hom}_{\Gamma}(M,-)$ is coperfect. 
\end{proof}

We can now translate nicely Theorem \ref{main_thm} in the context of module categories.

\begin{cor} \label{main_thm_2}
    Let $(\Mod \Gamma,\Mod \Lambda,\mathsf{i},\mathsf{e},\mathsf{l})$ be a cleft extension of module categories. If $\mathsf{F}$ is perfect and nilpotent, then $\Lambda$ is right Gorenstein if and only if $\Gamma$ is right Gorenstein. 
\end{cor}
\begin{proof}
     By Proposition \ref{cleft_of_modules_is_cocleft}, the given cleft extension is also a cleft coextension and the endofunctors $\mathsf{F}$ and $\mathsf{F}'$ are perfect and coperfect respectively. Since $\mathsf{F}$ is assumed to be nilpotent, the result follows from Theorem \ref{main_thm}.
\end{proof}

Combining the above with Proposition \ref{two_sided_noetherian_gorenstein} gives the following.  
 
\begin{cor} \label{cor_for_gorenstein_rings}
    Let $(\Mod \Gamma,\Mod \Lambda,\mathsf{i},\mathsf{e},\mathsf{l})$ be a cleft extension of module categories, where $\Gamma$ and $\Lambda$ are two-sided Noetherian rings. If $\mathsf{F}$ is perfect and nilpotent, then $\Lambda$ is Iwanaga-Gorenstein if and only if $\Gamma$ is Iwanaga-Gorenstein. 
\end{cor}

\section{Cleft extensions of module categories} 

Some examples of classes of cleft extensions of rings, together with inclusions, are summarised below. For an extensive list of examples, we refer to \cite{beligiannis}. 
\[
\adjustbox{scale=0.96}{
\begin{tabular}{ c c c c c } 
  & & \\[-0.5em]
 \big\{\text{One-point extensions}\big\} & $\subseteq$ & \big\{\text{Triangular matrix rings}\big\} \\ 
  \rotatebox[origin=c]{270}{$\subseteq$} & & \rotatebox[origin=c]{270}{$\subseteq$} \\
 \big\{\text{\!\! 0-0 Morita context rings}\big\} &  & \big\{\text{Tensor rings}\big\} \\ 
\rotatebox[origin=c]{270}{$\subseteq$} & & \rotatebox[origin=c]{270}{$\subseteq$}  \\
 \big\{\text{Trivial extension rings}\big\} 
  & $\subseteq$ & \big\{\text{Positively graded rings}\big\} & $\subseteq$ & \big\{\text{$\theta$-extensions}\big\} \\ 
\end{tabular}}
\]

\subsection{Triangular matrix rings and Morita context rings}
\begin{exmp} \label{perfect_for_triangular} (\textbf{Triangular matrix rings})
Let $A$ and $B$ be two rings and $N$ an $A$-$B$-bimodule. Consider the triangular matrix ring
\[
\Lambda=\begin{pmatrix}
    A & N \\ 
    0 & B
\end{pmatrix}
\]
By \cite{reiten}, the modules over $\Lambda$ can be viewed as triples $(X,Y,f)$ where $X$ is an $A$-module, $Y$ is a $B$-module and $f\in\mathsf{Hom}_B(X\otimes_AN,Y)$. Then, a morphism $(X,Y,f)\rightarrow (X',Y',f')$ is given by a pair $(a,b)$ where $a\colon X\rightarrow X'$ is an $A$-module homomorphism and $b\colon Y\rightarrow Y'$ is a $B$-module homomorphism such that $b\circ f=f'\circ (a\otimes N)$. There is a cleft extension of module categories: 
\[
\begin{tikzcd}
\Mod A\!\times \!B \arrow[rr, "\mathsf{i}"] &  & \Mod \Lambda \arrow[rr, "\mathsf{e}"] \arrow[ll, "\mathsf{q}"', bend right] &  & \Mod A\!\times\! B \arrow[ll, "\mathsf{l}"', bend right] \arrow["\mathsf{F}"', loop, distance=2em, in=125, out=55]
\end{tikzcd}
\]
where the functors are given as follows:
\begin{itemize}
    \item[(i)] The functor $\mathsf{i}$ is defined by $\mathsf{i}(X,Y)=(X,Y,0)$ on objects and given a morphism $(a,b)\colon(X,Y)\rightarrow (X',Y')$ then $\mathsf{i}(a,b)=(a,b)$.
    \item[(ii)] The functor $\mathsf{e}$ is defined by $\mathsf{e}(X,Y,f)=(X,Y)$ on objects and given a morphism $(a,b)\colon(X,Y,f)\rightarrow (X',Y',f')$ then $\mathsf{e}(a,b)=(a,b)$. 
    \item[(iii)] The functor $\mathsf{l}$ is defined by $\mathsf{l}(X,Y)=(X,(X\otimes_A N)\oplus Y,1)$ on objects and given a morphism $(a,b)\colon (X,Y)\rightarrow (X',Y')$ then $\mathsf{l}(a,b)=(a,(a\otimes N)\oplus b)$.
    \item[(iv)] The functor $\mathsf{q}$ is defined by $\mathsf{q}(X,Y,f)=(X,\mathsf{coker}f)$ on objects and given a morphism $(a,b)\colon(X,Y,f)\rightarrow (X',Y',f')$ then $\mathsf{q}(a,b)=(a,b')$ where $b'$ is the unique morphism $\mathsf{coker}f\rightarrow \mathsf{coker}f'$ such that $b'\circ\mathsf{coker}f=\mathsf{coker}f'\circ b$. 
    \item[(v)] The functor $\mathsf{F}$ is defined by $\mathsf{F}(X,Y)=(0,X\otimes_A N)$ on objects and given a morphism $(a,b)\colon (X,Y)\rightarrow (X',Y')$ then $\mathsf{F}(a,b)=(0,a\otimes N)$. 
\end{itemize}
We will now find necessary and sufficient conditions for the functor $\mathsf{F}$ to be perfect. First of all, we notice that $\mathsf{F}^2=0$ thus, in particular, $\mathsf{F}$ is nilpotent. Consider an $A\!\times\! B$-module $(X,Y)$ and a projective resolution of it, as below
\[
\cdots\rightarrow (P_n,Q_n)\rightarrow\cdots\rightarrow (P_0,Q_0)\rightarrow (X,Y)\rightarrow 0.
\]
Applying $\mathsf{F}$ gives the following complex: 
\[
\cdots\rightarrow (0,P_n\otimes_AN)\rightarrow \cdots\rightarrow (0,P_0\otimes_AN)\rightarrow 0,
\]
from which we derive that $\mathbb{L}_i\mathsf{F}(X,Y)\cong (0,\mathsf{Tor}_i^A(X,N))$ for all $i\geq 1$. Therefore, for any $A\!\times\! B$-module $(X,Y)$, we have $\mathbb{L}_i\mathsf{F}(\mathsf{F}(X,Y))=0$ and in particular the functor $\mathsf{F}$ satisfies condition (\ref{R}). Moreover, the left derived functor $\mathbb{L}_i\mathsf{F}$ vanishes for $i>\!\!>0$ if and only if $\mathsf{Tor}_i^A(X,N)=0$ for $i>\!\!>0$, which is equivalent to $\fd{_AN}<\infty$. Lastly, $\mathsf{F}$ maps projective modules to modules of finite projective dimension if and only if the functor $-\otimes_AN\colon\Mod B\rightarrow \Mod B$ maps projective modules to modules of finite projective dimension, which happens precisely when $\pd{N_B}<\infty$. Summing up, the functor $\mathsf{F}$ is nilpotent and it is perfect if and only if $\fd{_AN}<\infty$ and $\pd{N_B}<\infty$ (see also \cite[Example 4.10 (3)]{perfect}).
\end{exmp}

By the above computations, together with Corollary \ref{cor_for_gorenstein_rings}, we recover \cite[Theorem 3.3]{chen} (see also \cite[Theorem 2.2]{xiong_zhang} and \cite[Section 2]{zhang}). 

\begin{cor} \label{gorenstein_triangular}
    Let $\Lambda=\big(\begin{smallmatrix}   A & N\\   0 & B \end{smallmatrix}\big)$ be a triangular matrix ring, where $A$ and $B$ are two-sided Noetherian and $M$ is finitely generated on both sides, in which case $\Lambda$ is two-sided Noetherian. Assume that the following are satisfied: 
    \begin{itemize}
        \item[(i)] $\pd{_{A}N}<\infty$. 
        \item[(ii)] $\pd{N_{B}}<\infty$. 
    \end{itemize}
    Then $\Lambda$ is Iwanaga-Gorenstein if and only if $A$ and $B$ are Iwanaga-Gorenstein.
\end{cor}

In the following example we provide sufficient conditions so that the functor $\mathsf{e}$ of Example \ref{perfect_for_triangular} is an eventually homological isomorphism. Results of this nature were first explored in \cite{OPS}, using recollements (see in particular \cite[Section 8]{PSS})

\begin{exmp}  \label{triangular_eventual}
Let $\Lambda=\big(\begin{smallmatrix}   A & N\\   0 & B \end{smallmatrix}\big)$ be a triangular matrix ring and assume that $\pd{_AN}<\infty$ and $\gd B<\infty$. Further, consider the following functors
\[
\Mod \Lambda \xrightarrow{ \ \mathsf{e} \ } \Mod A\!\times\! B\xrightarrow{\ \mathsf{pr} \ } \Mod A
\]
\[
\hspace*{-0.8cm}{(X,Y,f)} \ \longmapsto  \ \ (X,Y) \ \ \longmapsto  \  X
\]
Recall from Example \ref{perfect_for_triangular} that $\Mod \Lambda$ is a cleft extension of $\Mod A\!\times\! B$ and the functor $\mathsf{G}$ is given by $(X,Y,f)\mapsto (0,X\otimes_A N,0)$. Consequently, since $\gd B<\infty$, it follows that $\mathsf{sup}\{\pd{_{\mathcal{A}}\mathsf{G}(X)} \ | \ X\in\Mod\Lambda\}<\infty$. Further, under the given assumptions, the functor $\mathsf{F}$ is perfect and for every $(P,Q)\in \Proj A\!\times\! B$ we have 
\[
\pd {\mathsf{F}(P,Q)_{A\times B}}=\pd{(0,P\otimes_A N)_{A\times B}}\leq \pd N_B\leq  \gd B,
\]
i.e.\ $\mathsf{sup}\{\pd{\mathsf{F}(P,Q)_{A\times B}} \ | \ (P,Q)\in\Proj A\!\times\! B\}<\infty$. Thus, we may apply Theorem \ref{thm2} to conclude that $\mathsf{e}$ is an eventually homological isomorphism. Further, since $\gd B<\infty$, it also follows that the projection functor $\mathsf{pr}$ is an eventually homological isomorphism. All in all, the composition $\mathsf{pr}\circ\mathsf{e}\colon\Mod \Lambda\rightarrow \Mod A$ is an eventually homological isomorphism. 
\end{exmp}

\begin{exmp} \label{morita_context_rings} \textbf{(Morita context rings)}
    Let $A$, $B$ be rings, $N$ an $A$-$B$-bimodule and $M$ a $B$-$A$-bimodule. Consider the Morita context ring 
    \[
    \Lambda=\begin{pmatrix}
        A & N \\ 
        M & B
    \end{pmatrix}
    \]
    with zero bimodule homomorphisms, i.e.\ multiplication is given by 
\[
\begin{pmatrix}
  a & n\\
  m & b
\end{pmatrix} \begin{pmatrix}
  a' & n'\\
  m' & b'
\end{pmatrix}=\begin{pmatrix}
  aa' & an'+nb'\\
  ma'+bm' & bb'
\end{pmatrix}.
\]
Recall that $\Mod \Lambda$ is equivalent to a category with objects tuples $(X,Y,f,g)$ where $X\in\Mod A$, $Y\in\Mod B$, $f\in\mathsf{Hom}_A(Y\otimes_BM,X)$ and $g\in\mathsf{Hom}_B(X\otimes_AN,Y)$. For a detailed exposition of this we refer to \cite{moritagreen, morita}. There is a cleft extension of module categories
\[
\begin{tikzcd}
\Mod A\!\times\! B \arrow[rr, "\mathsf{i}"'] &  & \Mod \Lambda \arrow[rr, "\mathsf{e}"'] \arrow[ll, "\mathsf{q}"', bend right] &  & \Mod A\!\times\! B \arrow[ll, "\mathsf{l}"', bend right] \arrow["\mathsf{F}"', loop, distance=2em, in=125, out=55]
\end{tikzcd}
\]
where the functors are given as follows (for simplicity we only explain what happens on objects): 
\begin{itemize}
    \item[(i)]  $\mathsf{i}(X,Y)=(X,Y,0,0)$.
    \item[(ii)] $\mathsf{e}(X,Y,f,g)=(X,Y)$. 
    \item[(iii)] $\mathsf{l}(X,Y)=(X,X\otimes_A N,0,1)\oplus (Y\otimes_B M,Y,1,0)$. 
    \item[(iv)] $\mathsf{q}(X,Y,f,g)=(X,\mathsf{coker}g)\oplus (Y,\mathsf{coker}f)$. 
    \item[(v)] $\mathsf{F}(X,Y)=(Y\otimes_B M,X\otimes_A N)$. 
\end{itemize}
Given an $A\!\times\! B$-module $(X,Y)$, consider a projective resolution of it as below 
\[
\cdots\rightarrow (P_n,Q_n)\rightarrow\cdots\rightarrow (P_0,Q_0)\rightarrow (X,Y)\rightarrow 0.
\]
Applying $\mathsf{F}$ to the latter gives the following complex:
\[
\cdots\rightarrow (Q_n\otimes_B M,P_n\otimes_A N)\rightarrow\cdots\rightarrow (Q_0\otimes_B M,P_0\otimes_A N)\rightarrow 0,
\]
from which we derive that $\mathbb{L}_i\mathsf{F}(X,Y)\cong (\mathsf{Tor}^B_i(Y,M),\mathsf{Tor}^A_i(X,N))$ for $i\geq 1$. Let us only work in the simple case when $N\otimes_B M=M\otimes_A N=0$. Then, $\mathsf{F}^2=0$, and in particular $\mathsf{F}$ is nilpotent. Moreover, if $\mathsf{Tor}_i^B(N,M)\cong \mathsf{Tor}_i^A(M,N)\cong 0$ for all $i\geq 1$, then the functor $\mathsf{F}$ satisfies condition (\ref{R}). Additionally, if $\pd {M_A}<\infty$ and $\pd{N_B}<\infty$, then $\mathsf{F}$ maps projective modules to modules of finite projective dimension. Lastly, if $\fd{{_A}N}<\infty$ and $\fd{{_B}M}<\infty$, then we infer that $\mathbb{L}_i\mathsf{F}=0$ for $i\geq \mathsf{max}\{\fd{_AN}, \fd{_BM}\}+1$.
\end{exmp}

By the above computations and using Corollary \ref{cor_for_gorenstein_rings}, we infer the following. 
We recover \cite[Corollary 4.15]{monomorphism_cat} as a special case.

\begin{cor} \label{gorenstein_morita_context}
    Let $\Lambda=\big(\begin{smallmatrix}   A & N\\   M & B \end{smallmatrix}\big)$ be a Morita context ring as in Example \ref{morita_context_rings}, where $A$ and $B$ are two-sided Noetherian and $N$, $M$ are finitely generated on both sides, in which case $\Lambda$ is two-sided Noetherian. Assume that the following are satisfied: 
    \begin{itemize}
        \item[(i)] $N\otimes_BM=M\otimes_AN=0$. 
        \item[(ii)] $\mathsf{Tor}_i^B(N,M)\cong \mathsf{Tor}_i^A(M,N)\cong 0$ for all $i\geq 1$. 
        \item[(iii)] $\pd{_{A}}N<\infty$, $\pd{{_B}M}<\infty$, $\pd{M_A}<\infty$ and $\pd{N_B}<\infty$. 
    \end{itemize}
    Then $\Lambda$ is Iwanaga-Gorenstein if and only if $A$ and $B$ are Iwanaga-Gorenstein. 
\end{cor}

\subsection{$\theta$-extensions} 
We apply Corollary \ref{main_thm_2} to the class of $\theta$-extensions \cite{marmaridis} and provide examples. We also explain in Remark \ref{canonical_form} that cleft extensions of module categories admit a ``canonical form'', given by those arising from $\theta$-extensions.

\begin{defn} 
    Let $\Gamma$ be a ring, $M$ a $\Gamma$-$\Gamma$-bimodule and $\theta\colon M\otimes_{\Gamma}M\rightarrow M$ an associative $\Gamma$-bimodule homomorphism. The \emph{$\theta$-extension of $\Gamma$ by $M$}, denoted by $\Gamma\!\ltimes\!_{\theta}M$, is defined to be the ring with underlying group $\Gamma\oplus M$ and multiplication given as follows:
    \[
    (\gamma,m)\cdot (\gamma',m'):=(\gamma\gamma',\gamma m'+m\gamma'+\theta(m\otimes m')),
    \]
    for $\gamma,\gamma'\in\Gamma$ and $m,m'\in M$. 
\end{defn}

Given a $\theta$-extension $\Gamma\!\ltimes\!_{\theta}M$, consider the ring homomorphisms $f\colon\Gamma\!\ltimes\!_{\theta}M\rightarrow \Gamma$ and $g\colon\Gamma\rightarrow \Gamma\!\ltimes\!_{\theta}M$ given by $(\gamma,m)\mapsto \gamma$ and $\gamma\mapsto (\gamma,0)$ respectively. Denote by $\textbf{Z}$ the restriction functor $\Mod \Gamma\rightarrow \Mod \Gamma\!\ltimes\!_{\theta}M$ induced by $f$ and denote by $\textbf{U}$ the restriction functor $\Mod\Gamma\!\ltimes\!_{\theta}M$ induced by $g$. Consider the induced cleft extension of module categories
\[
\begin{tikzcd}
\Mod \Gamma \arrow[rr, "\textbf{Z}"] &  & \Mod \Gamma\!\ltimes\!_{\theta}M \arrow[rr, "\textbf{U}"] \arrow[ll, "\textbf{C}=-\otimes_{\Gamma\ltimes_{\theta}M}\Gamma"', bend right] &  & \Mod \Gamma \arrow[ll, "\textbf{T}=-\otimes_{\Gamma}\Gamma\ltimes_{\theta} M"', bend right]
\end{tikzcd}
\]
By Watt's theorem, the induced endofunctor $\mathsf{F}$ on $\Mod \Gamma$ is naturally isomorphic to $-\otimes_{\Gamma}M$. It is clear by our results (see Corollary \ref{main_thm_2}) that if $M$ is perfect and nilpotent, then $\Gamma$ is right Gorenstein if and only if $\Gamma\!\ltimes\!_{\theta}M$ is right Gorenstein. In order to spell-out the result when restricted to two-sided Noetherian rings, we begin with the following characterization of two-sided Noetherian $\theta$-extensions.

\begin{fact} \label{noetherian_theta_extensions}
    Let $\Lambda=\Gamma\ltimes_{\theta}M$ be $\theta$-extension of a ring $\Gamma$. Then, $\Lambda$ is Noetherian if and only if $\Gamma$ is Noetherian and $M$ is finitely generated as a right $\Gamma$-module. Moreover, $\Lambda$ is two-sided Noetherian if and only if $\Gamma$ is two-sided Noetherian and $M$ is finitely generated on both sides. See \cite[Proposition 7.5(iv)]{beligiannis} for a proof. 
\end{fact}

\begin{prop} \label{main_thm_3}
    Let $\Gamma$ be a two-sided Noetherian ring, $M$ a $\Gamma$-bimodule that is finitely generated on both sides and $\theta\colon M\otimes_{\Gamma}M\rightarrow M$ an associative $\Gamma$-bimodule homomorphism. If $M$ is perfect and nilpotent, then $\Gamma\!\ltimes\!_{\theta}M$ is Iwanaga-Gorenstein if and only if $\Gamma$ is Iwanaga-Gorenstein.
\end{prop}
\begin{proof}
    There is a cleft extension $(\Mod \Gamma,\Mod \Gamma\ltimes_{\theta}M,\textbf{Z},\textbf{U},\textbf{T})$ of module categories with $\mathsf{F}\simeq -\otimes_{\Gamma}M$ perfect and nilpotent. Since $M$ is assumed to be nilpotent, it follows that $\theta$ is nilpotent, meaning that $\Gamma\ltimes_{\theta}M$ is two-sided Noetherian; see Fact \ref{noetherian_theta_extensions}. The result follows from Corollary \ref{cor_for_gorenstein_rings}.
\end{proof}

In the following corollary we recover \cite[Theorem 4.6]{perfect}.

\begin{cor} \label{gorenstein_tensor_rings}
    Let $\Gamma$ be a two-sided Noetherian ring and $M$ a $\Gamma$-bimodule that is nilpotent and finitely generated on both sides. Assume that $M$ is perfect. Then the tensor ring $T_{\Gamma}(M)$ is Iwanaga-Gorenstein if and only if $\Gamma$ is Iwanaga-Gorenstein. 
\end{cor}
\begin{proof}
There is an isomorphism $T_R(M)\cong R\!\ltimes\!_{\theta}M'$ where $M'=M\oplus M^{\otimes 2}\oplus \cdots$ and $\theta$ is induced by $M^{\otimes k}\otimes_{\Gamma}M^{\otimes l}\rightarrow M^{\otimes k+l}$. Further, it follows by Lemma \ref{easy_properties_of_perfect_bimodules} that $M'$ is perfect (since the same holds for $M$) and, evidently, it is also nilpotent. The result follows from Proposition \ref{main_thm_3}. 
\end{proof}

We also obtain a result on trivial extensions. 

\begin{cor} \label{gorenstein_trivial_extensions}
    Let $\Gamma$ be a two-sided Noetherian ring and $M$ a $\Gamma$-bimodule that is finitely generated on both sides. Assume that $M$ is perfect and nilpotent. Then the trivial extension $\Gamma\ltimes M$ is Iwanaga-Gorenstein if and only if $\Gamma$ is Iwanaga-Gorenstein.
\end{cor}
\begin{proof}
    Once we observe that $\Gamma\!\ltimes\! M\cong \Gamma\!\ltimes\!_{\theta}M$ with $\theta=0$, the result follows from Proposition \ref{main_thm_3}. 
\end{proof}

In the extreme case when $M$ is projective as a $\Gamma$-bimodule and $M\otimes_{\Gamma} M=0$ the above recovers \cite[Corollary 4.2]{arrow2}.

As explained, every $\theta$-extension gives rise to a cleft extension of module categories. There is more than meets the eye: every cleft extension of module categories is isomorphic (in an appropriate sense) to a cleft extension induced from a $\theta$-extension. In the following remark we explain this precisely. 

\begin{rem} \label{canonical_form} (\!\!\cite{beligiannis,beligiannis2}, see also \cite[Proposition 6.9]{graded_injective_generation}) Let $\Gamma$ be a ring. The relations between the various constructions considered in the paper are summarised below: 
\[
\begin{tikzcd}
\{\text{cleft extensions of $\Gamma$}\} \arrow[d, shift right]                                         \\
\{\text{$\theta$-extensions of $\Gamma$}\} \arrow[d, shift right] \arrow[u, shift right]               \\
\{\text{cleft extensions $(\Mod\Gamma,\Mod\Lambda,\mathsf{i},\mathsf{e},\mathsf{l})$}\} \arrow[u, "\text{up to Morita equivalence}"', shift right]
\end{tikzcd}
\] 
The relation between cleft extensions of $\Gamma$ and $\theta$-extensions of $\Gamma$ is easily deduced, while the fact that a $\theta$-extension gives rise to a cleft extension has already been explained. The fact that a cleft extension $(\Mod\Gamma,\Mod\Lambda,\mathsf{i},\mathsf{e},\mathsf{l})$ gives rise to a $\theta$-extension \emph{up to Morita equivalence} is more subtle. In this case, there is a $\Gamma$-bimodule $M$, an associative $\Gamma$-bimodule homomorphism $\theta\colon M\otimes_{\Gamma}M\rightarrow M$ and an equivalence $ \Mod \Lambda\xrightarrow{\sim} \Mod \Gamma\!\ltimes\!_{\theta}M$ making the following diagram commutative:
    \[
\begin{tikzcd}
\Mod \Gamma \arrow[rr, "\mathsf{i}"]                                 &  & \Mod \Lambda \arrow[rr, "\mathsf{e}"] \arrow[dd, "\simeq"] \arrow[ll, "\mathsf{q}"', bend right] &  & \Mod \Gamma \arrow[ll, "\mathsf{l}"', bend right, shift right]                    \\
                                                                     &  &                                                                                                  &  &                                                                                   \\
\Mod \Gamma \arrow[rr, "\textbf{Z}"] \arrow[uu, no head, equal] &  & \Mod \Gamma\!\ltimes\!_{\theta}M \arrow[rr, "\textbf{U}"] \arrow[ll, "\textbf{C}"', bend right]      &  & \Mod \Gamma \arrow[ll, "\textbf{T}"', bend right] \arrow[uu, no head, equal]
\end{tikzcd}
    \]
    where the bottom diagram is the cleft extension associated to $\Gamma\ltimes_{\theta}M$. This provides a \emph{canonical form} for cleft extensions of module categories. 
\end{rem}

The following is important for later use. 

\begin{lem} \label{restriction_of_cleft}
    Let $(\Mod\Gamma,\Mod\Lambda,\mathsf{i},\mathsf{e},\mathsf{l})$ be a cleft extension of module categories of Noetherian rings. If $\mathsf{F}$ is nilpotent, then the given cleft extension restricts to a cleft extension $(\smod\Gamma,\smod\Lambda,\mathsf{i},\mathsf{e},\mathsf{l})$ between the respective subcategories of finitely generated modules.
\end{lem}
\begin{proof}
    By Remark \ref{canonical_form}, it is enough to prove the claim for the cleft extension associated to a Noetherian $\theta$-extension $\Gamma\!\ltimes\!_{\theta}M$. Then, the functor $\mathsf{F}\simeq -\otimes_{\Gamma}M$ is nilpotent if and only if $M$ is nilpotent, in which case follows that $\theta$ is nilpotent. In particular, it follows from Fact \ref{noetherian_theta_extensions} that $M$ is finitely generated on both sides and so the functors of the cleft extension $(\Mod\Gamma,\Mod\Gamma\!\ltimes\!_{\theta}M,\textbf{Z},\textbf{U},\textbf{T})$ restrict to functors between $\smod\Gamma$ and $\smod\Gamma\!\ltimes\!_{\theta}M$. 
\end{proof}

\section{Singularity categories}

The aim of this section is to study the singularity category of abelian categories that appear in a cleft extension diagram. Our main objective is to obtain the middle part of the commutative diagram of Theorem A, see Proposition \ref{cleft_of_singularity}. Under extra assumptions, we obtain an equivalence of singularity categories, see Theorem \ref{equivalence_singularity}. 

\subsection{Cleft extensions and singularity categories}
We denote by $\mathcal{A}$ an abelian category with enough projective objects. The \emph{singularity category} of $\mathcal{A}$, introduced in \cite{buchweitz}, is the Verdier quotient 
\[
\mathsf{D}_{\mathsf{sg}}(\mathcal{A}):=\mathsf{D}^{\mathsf{b}}(\mathcal{A})/\mathsf{K}^{\mathsf{b}}(\Proj\mathcal{A}).
\]

We recall two well-known lemmata about adjoint pairs of triangle functors and Verdier quotients. For the first we refer the reader to \cite[Lemma 1.2]{orlov} and for the second to \cite[Propositions 1.5, 1.6]{bondal_kapranov}. 

\begin{lem} \label{orlov}
    Let $\mathsf{G}\colon\mathcal{T}\rightarrow \mathcal{T}'$ and $\mathsf{F}\colon \mathcal{T}'\rightarrow\mathcal{T}$ be triangle functors such that $(\mathsf{F},\mathsf{G})$ is an adjoint pair. Moreover, let $\mathcal{S}$ and $\mathcal{S}'$ be triangulated subcategories of $\mathcal{T}$ and $\mathcal{T}'$ respectively such that $\mathsf{G}(\mathcal{S})\subseteq \mathcal{S}'$ and $\mathsf{F}(\mathcal{S}')\subseteq \mathcal{S}$. Then the latter functors induce triangle functors
    \[
    \overline{\mathsf{G}}\colon \mathcal{T}/\mathcal{S}\rightarrow \mathcal{T}'/\mathcal{S}' \text{,  } \overline{\mathsf{F}}\colon\mathcal{T}'/\mathcal{S}'\rightarrow\mathcal{T}/\mathcal{S}
    \]
    and $(\overline{\mathsf{F}},\overline{\mathsf{G}})$ is an adjoint pair. 
\end{lem}

\begin{lem}  \label{equivalence_of_triangulated_cats}
    Let $\mathsf{G}\colon\mathcal{T}\rightarrow \mathcal{T}'$ and $\mathsf{F}\colon \mathcal{T}'\rightarrow\mathcal{T}$ be triangle functors such that $(\mathsf{F},\mathsf{G})$ is an adjoint pair. The following hold: 
    \begin{itemize}
        \item[(i)] If $\mathsf{F}$ is fully faithful, then $\mathcal{T}/\mathsf{ker}\mathsf{G}\simeq \mathcal{T}'$. 
        \item[(ii)] If $\mathsf{G}$ is fully faithful, then $\mathcal{T}'/\mathsf{kerF}\simeq \mathcal{T}$. 
    \end{itemize}
\end{lem}

We will now show that a cleft extension diagram satisfying certain homological conditions gives rise to a similar diagram of singularity categories. We will afterwards show that the homological conditions identified, are satisfied under the usual perfectness conditions of this paper.

\begin{prop} \label{cleft_of_singularity}
        Let $(\mathcal{B},\mathcal{A},\mathsf{i},\mathsf{e},\mathsf{l})$ be a cleft extension of abelian categories such that the following conditions are satisfied: 
    \begin{itemize}
        \item[(i)] $\pd{_{\mathcal{A}}\mathsf{i}(P)}<\infty$ for every $P\in\Proj\mathcal{B}$. 
        \item[(ii)] $\pd{_{\mathcal{B}}\mathsf{e}(P)}<\infty$ for every $P\in\Proj\mathcal{A}$. 
        \item[(iii)] $\mathbb{L}_n\mathsf{l}=0$ for $n$ large enough. 
        \item[(iv)] $\mathbb{L}_n\mathsf{q}=0$ for $n$ large enough. 
    \end{itemize}
    Then we get a diagram of singularity categories and triangle functors as below 
    \[
    \begin{tikzcd}
\mathsf{D}_{\mathsf{sg}}(\mathcal{B}) \arrow[rr, "\mathsf{i}"] &  & \mathsf{D}_{\mathsf{sg}}(\mathcal{A}) \arrow[rr, "\mathsf{e}"] \arrow[ll, "\mathbb{L}_{\mathsf{sg}}\mathsf{q}"', bend right] &  & \mathsf{D}_{\mathsf{sg}}(\mathcal{B}) \arrow[ll, "\mathbb{L}_{\mathsf{sg}}\mathsf{l}"', bend right]
\end{tikzcd}
    \]
    such that $(\mathbb{L}_{\mathsf{sg}}\mathsf{l},\mathsf{e})$, $(\mathbb{L}_{\mathsf{sg}}\mathsf{q},\mathsf{i})$ are adjoint pairs, $\mathsf{ei}\simeq \mathsf{Id}_{\mathsf{D}_{\mathsf{sg}}(\mathcal{B})}$ and $\mathbb{L}_{\mathsf{sg}}\mathsf{q}\mathbb{L}_{\mathsf{sg}}\mathsf{l}\simeq\mathsf{Id}_{\mathsf{D}_{\mathsf{sg}}(\mathcal{B})}$, where $\mathbb{L}_{\mathsf{sg}}\mathsf{l}$ denotes the functor induced by $\mathbb{L}\mathsf{l}$ and $\mathbb{L}_{\mathsf{sg}}\mathsf{q}$ the functor induced by $\mathbb{L}\mathsf{q}$.
\end{prop}
\begin{proof}
    First, we get the following diagram of derived categories:
    \[
    \begin{tikzcd}
\mathsf{D}^{-}(\mathcal{B}) \arrow[rr, "\mathsf{i}"] &  & \mathsf{D}^{-}(\mathcal{A}) \arrow[rr, "\mathsf{e}"] \arrow[ll, "\mathbb{L}\mathsf{q}"', bend right] &  & \mathsf{D}^{-}(\mathcal{B}) \arrow[ll, "\mathbb{L}\mathsf{l}"', bend right]
\end{tikzcd}
    \]
    with $(\mathbb{L}\mathsf{q},\mathsf{i})$ and $(\mathbb{L}\mathsf{l},\mathsf{e})$ being adjoint pairs of triangle functors. Since $\mathbb{L}\mathsf{i}=\mathbb{R}\mathsf{i}$ and $\mathbb{L}\mathsf{e}=\mathbb{R}\mathsf{e}$ are induced by the exact functors $\mathsf{i}$ and $\mathsf{e}$ with $\mathsf{ei}\simeq\mathsf{Id}_{\mathcal{B}}$, we have $\mathsf{ei}\simeq\mathsf{Id}_{\mathsf{D}^{-}(B)}$ on the level of derived categories. By the assumptions (iii) and (iv), we derive that the above functors restrict to:  
    \[
    \begin{tikzcd}
\mathsf{D}^{\mathsf{b}}(\mathcal{B}) \arrow[rr, "\mathsf{i}"] &  & \mathsf{D}^{\mathsf{b}}(\mathcal{A}) \arrow[rr, "\mathsf{e}"] \arrow[ll, "\mathbb{L}\mathsf{q}"', bend right] &  & \mathsf{D}^{\mathsf{b}}(\mathcal{B}) \arrow[ll, "\mathbb{L}\mathsf{l}"', bend right]
\end{tikzcd}
    \]
    and we have $\mathsf{ei}\simeq \mathsf{Id}_{\mathsf{D}^{\mathsf{b}}(\mathcal{B})}$. By assumption (i), it follows that $\mathsf{i}$ restricts to a functor $\mathsf{K}^{\mathsf{b}}(\Proj\mathcal{B})\rightarrow\mathsf{K}^{\mathsf{b}}(\Proj\mathcal{A})$. By (ii), the same is true for $\mathsf{e}$ and it is evidently true for $\mathbb{L}\mathsf{q}$ and $\mathbb{L}\mathsf{l}$. Therefore, by Lemma \ref{orlov}, we derive the following diagram:
    \[
    \begin{tikzcd}
\mathsf{D}_{\mathsf{sg}}(\mathcal{B}) \arrow[rr, "\mathsf{i}"] &  & \mathsf{D}_{\mathsf{sg}}(\mathcal{A}) \arrow[rr, "\mathsf{e}"] \arrow[ll, "\mathbb{L}_{\mathsf{sg}}\mathsf{q}"', bend right] &  & \mathsf{D}_{\mathsf{sg}}(\mathcal{B}) \arrow[ll, "\mathbb{L}_{\mathsf{sg}}\mathsf{l}"', bend right]
\end{tikzcd}
    \]
with $(\mathbb{L}_{\mathsf{sg}}\mathsf{q},\mathsf{i})$ and $(\mathbb{L}_{\mathsf{sg}}\mathsf{q},\mathsf{e})$ being adjoint pairs of triangle functors. Then $\mathsf{ei}\simeq \mathsf{Id}_{\mathsf{D}_{\mathsf{sg}}(\mathcal{B})}$ and since $(\mathbb{L}_{\mathsf{sg}}\mathsf{q}\mathbb{L}_{\mathsf{sg}}\mathsf{l},\mathsf{ei})$ is an adjoint pair, we have $\mathbb{L}_{\mathsf{sg}}\mathsf{q}\mathbb{L}_{\mathsf{sg}}\mathsf{l}\simeq \mathsf{Id}_{\mathsf{D}_{\mathsf{sg}}(\mathcal{B})}$. 
\end{proof} 

Let us now verify that the conditions of Proposition \ref{cleft_of_singularity} are satisfied whenever the endofunctor $\mathsf{F}$ associated to a cleft extension is perfect and nilpotent. The following Proposition is essential; see \cite[Proposition 6.6]{graded_injective_generation} for a proof, given in the context of module categories, which works verbatim in this more general setting.

\begin{prop} \label{derived_of_q_vanishes}
    Let $(\mathcal{B},\mathcal{A},\mathsf{i},\mathsf{e},\mathsf{l})$ be a cleft extension of abelian categories. If $\mathsf{F}$ is perfect and nilpotent, then $\mathbb{L}_i\mathsf{q}=0$ for $i$ large enough. 
\end{prop}

\begin{lem} \label{i_maps_proj}
    Let $(\mathcal{B},\mathcal{A},\mathsf{i},\mathsf{e},\mathsf{l})$ be a cleft extension of abelian categories. If $\mathsf{F}$ is perfect and nilpotent, then the functor $\mathsf{i}$ maps projective objects to objects with finite projective dimension.
\end{lem}
\begin{proof}
    Let $P$ be a projective object of $\mathcal{B}$. We know from Proposition \ref{preserves_and_reflects} that $\pd{_{\mathcal{B}}\mathsf{i}(P)}<\infty$ if and only if $\pd{_{\mathcal{B}}\mathsf{ei}(P)}<\infty$. Since $\mathsf{ei}(P)=P$, the latter is true. 
\end{proof}

\begin{cor}  \label{sunnefiasmenh_kuriakh}
    Let $(\mathcal{B},\mathcal{A},\mathsf{i},\mathsf{e},\mathsf{l})$ be a cleft extension of abelian categories such that $\mathsf{F}$ is perfect and nilpotent. Then the conditions of Proposition \ref{cleft_of_singularity} are satisfied. 
\end{cor}
\begin{proof}
    This follows from Lemma \ref{basic_properties_of_perfect_endofunctor_on_cleft}, Proposition \ref{derived_of_q_vanishes} and Lemma \ref{i_maps_proj}. 
\end{proof}

By the latter, we derive the middle part of the diagram of the introduction. 

\begin{cor} \label{middle_part}
    Let $(\Mod \Gamma,\Mod \Lambda,\mathsf{i},\mathsf{e},\mathsf{l})$ be a cleft extension of module categories of Noetherian rings. If $\mathsf{F}$ is perfect and nilpotent, then there exists a diagram 
    \[
\begin{tikzcd}
\mathsf{D}_{\mathsf{sg}}(\Gamma) \arrow[rr, "\mathsf{i}"] &  & \mathsf{D}_{\mathsf{sg}}(\Lambda) \arrow[rr, "\mathsf{e}"] \arrow[ll, "\mathbb{L}_{\mathsf{sg}}\mathsf{q}"', bend right] &  & \mathsf{D}_{\mathsf{sg}}(\Gamma) \arrow[ll, "\mathbb{L}_{\mathsf{sg}}\mathsf{l}"', bend right]
\end{tikzcd}
    \]
    of triangle functors such that $(\mathbb{L}_{\mathsf{sg}}\mathsf{q},\mathsf{i})$, $(\mathbb{L}_{\mathsf{sg}}\mathsf{q},\mathsf{i})$ are adjoint pairs, $\mathsf{ei}\simeq \mathsf{Id}_{\mathsf{D}_{\mathsf{sg}}(\Gamma)}$ and $\mathbb{L}_{\mathsf{sg}}\mathsf{q}\mathbb{L}_{\mathsf{sg}}\mathsf{l}\simeq\mathsf{Id}_{\mathsf{D}_{\mathsf{sg}}(\Gamma)}$. 
\end{cor}
\begin{proof}
    By Lemma \ref{restriction_of_cleft}, the cleft extension $(\Mod\Gamma,\Mod\Lambda,\mathsf{i},\mathsf{e},\mathsf{l})$ restricts to a cleft extension $(\smod\Gamma,\smod\Lambda,\mathsf{i},\mathsf{e},\mathsf{l})$ and the result follows by Corollary \ref{sunnefiasmenh_kuriakh}. 
\end{proof}

\subsection{An equivalence of singularity categories} 
We provide sufficient conditions for a cleft extension to induce an equivalence of singularity categories. In preparation for this, we recall two well-known lemmata.

Given an abelian category $\mathcal{A}$ and an object $X$ of $\mathcal{A}$, we denote by $X[n]$ the complex of $\mathsf{C}(\mathcal{A})$ with $X$ in degree $-n$ and 0 elsewhere.

\begin{lem} \label{perfect_objects}
    Let $\mathcal{A}$ be an abelian category with enough projectives. Then for every object $X$ of $\mathcal{A}$, the following are equivalent: 
    \begin{itemize}
        \item[(i)] $\pd{_{\mathcal{A}}X}<\infty$. 
        \item[(ii)] $X[0]\in\mathsf{K}^{\mathsf{b}}(\Proj\mathcal{A})$ as an object of $\mathsf{D}(\mathcal{A})$. 
    \end{itemize}
\end{lem}
\begin{proof}
    (i)$\implies$(ii): Assume that $\pd{_{\mathcal{A}} X}<\infty$ and let $P^{\bullet}$ be a projective resolution of $X$ of finite length. Then, $P^{\bullet}\cong X[0]$ in $\mathsf{D}(A)$ and $P^{\bullet}\in\mathsf{K}^{\mathsf{b}}(\Proj\mathcal{A})$.

    (ii)$\implies$(i): If $X[0]\in\mathsf{K}^{\mathsf{b}}(\Proj\mathcal{A})$, then for any $Y\in\mathcal{A}$ we have 
    \[
    \mathsf{Ext}_{\mathcal{A}}^i(X,Y)\cong \mathsf{Hom}_{\mathsf{D}(A)}(X[0],Y[i])\cong \mathsf{Hom}_{\mathsf{K}(\mathcal{A})}(X[0],Y[i]),
    \]
    and since $X[0]$ has finite length, the latter is $0$ for $i>d$ (for some $d$), meaning that $\pd{_{\mathcal{A}}X}<\infty$. 
\end{proof}

\begin{rem}
   An easy consequence of the above is that $\mathsf{D}_{\mathsf{sg}}(\mathcal{A})=0$ if and only if $\pd{_{\mathcal{A}} X}<\infty$ for every $X\in\mathcal{A}$. If $\mathcal{A}=\smod \Lambda$ is the category of finitely generated modules over a Noetherian ring, then the latter translates to $\gd \Lambda<\infty$ in case $\Lambda$ is a finite dimensional algebra or a local commutative ring (by the fact that there are finitely many simple modules and by the Auslander-Buchsbaum theorem respectively). However, $\mathsf{D}_{\mathsf{sg}}(\Lambda)=0$ is not in general equivalent to $\gd \Lambda<\infty$, see for instance \cite[Remark 8]{krause_fin_dim}. If we replace $\smod\Lambda$ by $\Mod \Lambda$, then indeed $\mathsf{D}_{\mathsf{sg}}(\Mod \Lambda)=0$ if and only if $\gd\Lambda<\infty$, for any ring $\Lambda$. 
\end{rem}

\begin{lem} \label{object_in_singularity}
    For any object $X\in\mathsf{D}_{\mathsf{sg}}(\mathcal{A})$, where $\mathcal{A}$ is an abelian category, there is an object $Y\in\mathcal{A}$ such that $X\cong Y[n]$ in $\mathsf{D}_{\mathsf{sg}}(\mathcal{A})$ for some $n\in\mathbb{Z}$. 
\end{lem}
\begin{proof} 
    This is well-known, see for instance \cite[Lemma 2.1]{chen2}, where this is proved for $\mathcal{A}=\smod\Lambda$ for a Noetherian ring $\Lambda$. The exact same proof works for any abelian category with enough projectives. 
\end{proof}

\begin{thm} \label{equivalence_singularity}
    Let $(\mathcal{B},\mathcal{A},\mathsf{i},\mathsf{e},\mathsf{l})$ be a cleft extension of abelian categories. Assume that $\mathsf{F}$ satisfies the following: 
    \begin{itemize}
        \item[(a)] $\pd{_{\mathcal{B}}\mathsf{F}(P)}<\infty$ for every $P\in\Proj\mathcal{B}$. 
        \item[(b)] $\mathbb{L}_n\mathsf{F}=0$ for $n$ large enough. 
        \item[(c)] The functor $\mathsf{e}$ reflects objects with finite projective dimension.
    \end{itemize} 
    The following are equivalent: 
    \begin{itemize}
        \item[(i)] The functor $\mathsf{e}\colon \mathsf{D}_{\mathsf{sg}}(\mathcal{A})\rightarrow \mathsf{D}_{\mathsf{sg}}(\mathcal{B})$ is an equivalence. 
        \item[(ii)] $\mathbb{L}_{\mathsf{sg}}\mathsf{F}\simeq 0$ in $\mathsf{D}_{\mathsf{sg}}(\mathcal{B})$. 
    \end{itemize}
\end{thm}
\begin{proof}
The assumptions (a) and (b) ensure that $\mathsf{e}$, $\mathbb{L}\mathsf{l}$ and $\mathbb{L}\mathsf{F}$ give rise to functors on the level of singularity categories which we denote by $\mathsf{e}, \mathbb{L}_{\mathsf{sg}}\mathsf{l}$ and $\mathbb{L}_{\mathsf{sg}}\mathsf{F}$ respectively.  

(i) $\Longrightarrow$ (ii): We have $\mathsf{el}\simeq \mathsf{Id}_{\mathcal{B}}\oplus\mathsf{F}$ and therefore $\mathsf{e}\mathbb{L}_{\mathsf{sg}}\mathsf{l}\simeq \mathsf{Id}_{\mathsf{D}_{\mathsf{sg}}(\mathcal{B})}\oplus\mathbb{L}_{\mathsf{sg}}\mathsf{F}$. If the functor $\mathsf{e}\colon\mathsf{D}_{\mathsf{sg}}(\mathcal{A})\rightarrow\mathsf{D}_{\mathsf{sg}}(\mathcal{B})$ is an equivalence, then $\mathbb{L}_{\mathsf{sg}}\mathsf{F}\simeq 0$.

(ii) $\Longrightarrow$ (i): We have $\mathsf{el}\simeq \mathsf{Id}_{\mathcal{B}}\oplus\mathsf{F}$ and therefore $\mathsf{e}\mathbb{L}_{\mathsf{sg}}\mathsf{l}\simeq \mathsf{Id}_{\mathsf{D}_{\mathsf{sg}}(\mathcal{B})}\oplus\mathbb{L}_{\mathsf{sg}}\mathsf{F}$. It follows that $\mathsf{e}\mathbb{L}_{\mathsf{sg}}\mathsf{l}\simeq \mathsf{Id}_{\mathsf{D}_{\mathsf{sg}}(\mathcal{B})}$. Therefore $\mathbb{L}_{\mathsf{sg}}\mathsf{l}$ is fully faithful and so, by Lemma \ref{equivalence_of_triangulated_cats}, there is an equivalence 
\[
\mathsf{D}_{\mathsf{sg}}(\mathcal{A})/\mathsf{kere}\simeq \mathsf{D}_{\mathsf{sg}}(\mathcal{B}).
\]
We claim that $\mathsf{kere}$ is trivial. Indeed, let $X\in\mathsf{D}_{\mathsf{sg}}(\mathcal{A})$ be such that $\mathsf{e}(X)=0$. By Lemma \ref{object_in_singularity}, there is $Y\in\mathcal{A}$ and $n\in\mathbb{Z}$ such that $X\cong Y[n]$ in $\mathsf{D}_{\mathsf{sg}}(\mathcal{A})$. Therefore $\mathsf{e}(Y)[n]=\mathsf{e}(Y[n])=0$, so $\mathsf{e}(Y[0])=\mathsf{e}(Y)[0]=0$ meaning that $\mathsf{e}(Y[0])\in\mathsf{K}^{\mathsf{b}}(\Proj\mathcal{B})$ as an object of $\mathsf{D}^{\mathsf{b}}(\mathcal{B})$. Thus, by Lemma \ref{perfect_objects}, it follows that $\pd{_{\mathcal{B}}\mathsf{e}(Y)}<\infty$. By the assumption on the functor $\mathsf{e}$, we get that $\pd{_{\mathcal{A}}Y}<\infty$, meaning that $Y[0]=0$ in $\mathsf{D}_{\mathsf{sg}}(\mathcal{A})$, thus also $X\cong Y[n]=0$ in $\mathsf{D}_{\mathsf{sg}}(\mathcal{A})$. 
\end{proof}

Evidently, $\mathbb{L}_{\mathsf{sg}}\mathsf{F}\simeq 0$ in $\mathsf{D}_{\mathsf{sg}}(\mathcal{B})$ (provided that it exists) if and only if the functor $\mathbb{L}\mathsf{F}\colon\mathsf{D}(\mathcal{B})\rightarrow \mathsf{D}(\mathcal{B})$ maps $\mathsf{D}^{\mathsf{b}}(\mathcal{B})$ to $\mathsf{K}^{\mathsf{b}}(\Proj\mathcal{B})$. In the following lemma we present an instance of this happening. 

\begin{lem} \label{left_derived_of_F_vanishes}
    Let $\mathsf{F}\colon\mathcal{B}\rightarrow\mathcal{B}$ be a right exact functor of an abelian category with enough projectives, which satisfies the following: 
    \begin{itemize}
        \item[(i)] $\pd{_{\mathcal{B}}\mathsf{F}(X)}<\infty$ for every $X\in\mathcal{B}$. 
        \item[(ii)] $\mathbb{L}_n\mathsf{F}=0$ for $n$ large enough. 
    \end{itemize}
    Then, $\mathbb{L}_{\mathsf{sg}}\mathsf{F}=0$ in $\mathsf{D}_{\mathsf{sg}}(\mathcal{B})$. 
\end{lem}
\begin{proof}
    The functor $\mathbb{L}\mathsf{F}\colon \mathsf{D}^{\mathsf{b}}(\mathcal{B})\rightarrow\mathsf{D}^{\mathsf{b}}(\mathcal{B})$ exists by assumption (ii), and by assumption (i), it gives rise to a functor $\mathbb{L}_{\mathsf{sg}}\mathsf{F}\colon\mathsf{D}_{\mathsf{sg}}(\mathcal{B})\rightarrow \mathsf{D}_{\mathsf{sg}}(\mathcal{B})$. 
    For every $X\in\mathsf{D}_{\mathsf{sg}}(\mathcal{B})$ there is $Y\in\mathcal{B}$ and $n\in\mathbb{Z}$ such that $X\cong Y[n]$. But, $\mathbb{L}_{\mathsf{sg}}\mathsf{F}(Y[n])=0$ in $\mathsf{D}_{\mathsf{sg}}(\mathcal{B})$ if and only if $\mathbb{L}_{\mathsf{sg}}\mathsf{F}(Y[0])=0$ in $\mathsf{D}_{\mathsf{sg}}(\mathcal{B})$ and for the latter, it is enough that $\pd{_{\mathcal{B}}\mathbb{L}_i\mathsf{F}(Y)}<\infty$ for every $i\geq 0$. For $i=0$ this holds true by assumption. In order to show that $\pd{_{\mathcal{B}}\mathbb{L}_1\mathsf{F}(Y)}<\infty$, we consider a short exact sequence $0\rightarrow K\rightarrow P\rightarrow Y\rightarrow 0$ with $P\in\Proj\mathcal{B}$. The latter gives rise to the following exact sequence in $\mathcal{B}$:
    \[
    0\rightarrow\mathbb{L}_1\mathsf{F}(Y)\rightarrow \mathsf{F}(K)\rightarrow\mathsf{F}(P)\rightarrow\mathsf{F}(Y)\rightarrow 0.
    \]
The objects $\mathsf{F}(K),\mathsf{F}(P),\mathsf{F}(Y)$ have finite projective dimension by assumption, so we conclude that $\pd{_{\mathcal{B}}\mathbb{L}_1\mathsf{F}(Y)}<\infty$. The claim is proved similarly for $i\geq 2$.
\end{proof}

For module categories Theorem \ref{equivalence_singularity} implies the following. 

\begin{cor} \label{equivalence_singularity_for_mod}
    Let $(\Mod \Gamma,\Mod \Lambda,\mathsf{i},\mathsf{e},\mathsf{l})$ be a cleft extension of module categories of Noetherian rings. Assume that the functor $\mathsf{F}$ is perfect and nilpotent. The following are equivalent: 
    \begin{itemize}
        \item[(i)] The functor $\mathsf{e}\colon\mathsf{D}_{\mathsf{sg}}(\Lambda)\rightarrow \mathsf{D}_{\mathsf{sg}}(\Gamma)$ is an equivalence. 
        \item[(ii)] $\mathbb{L}_{\mathsf{sg}}\mathsf{F}\simeq 0$ in $\mathsf{D}_{\mathsf{sg}}(\Gamma)$.
    \end{itemize}
\end{cor}
\begin{proof}
   Since $\mathsf{F}$ is perfect and nilpotent, it follows by Lemma \ref{basic_properties_of_perfect_endofunctor_on_cleft} and Proposition \ref{preserves_and_reflects} that the assumptions of Theorem \ref{equivalence_singularity} are satisfied. Hence, together with the observation that $(\Mod\Gamma,\Mod\Lambda,\mathsf{i},\mathsf{e},\mathsf{l})$ restricts to $(\smod\Gamma,\smod\Lambda,\mathsf{i},\mathsf{e},\mathsf{l})$, see Lemma \ref{restriction_of_cleft}, the result follows. 
\end{proof}

By the above and using Lemma \ref{left_derived_of_F_vanishes}, we can give the following examples. 

\begin{exmp} \label{equivalence_for_triangular}
    Let $A$ and $B$ be Noetherian rings and $N$ an $A$-$B$-bimodule that is finitely generated on both sides. Consider the triangular matrix ring $\Lambda=\big(\begin{smallmatrix}   A & N\\   0 & B \end{smallmatrix}\big)$, which is also Noetherian. Assume that $\gd B<\infty$ and $\pd{_AN}<\infty$. By Example \ref{perfect_for_triangular}, $\Mod\Lambda$ is a cleft extension of $\Mod A\!\times\! B$ and under the given assumptions, the functor $\mathsf{F}$, given by $\mathsf{F}(X,Y)=(0,X\otimes_A N)$, is perfect and nilpotent. Moreover, it follows that $\pd{\mathsf{F}(X,Y)_{A\times B}}<\infty$ for every $(X,Y)\in\Mod A\!\times\! B$. We infer from Lemma \ref{left_derived_of_F_vanishes} that $\mathbb{L}_{\mathsf{sg}}\mathsf{F}=0$ in $\mathsf{D}_{\mathsf{sg}}(A\!\times\! B)$. Therefore, we may apply Corollary \ref{equivalence_singularity_for_mod} to obtain the following equivalences:
    \[
    \mathsf{D}_{\mathsf{sg}}(\begin{pmatrix}   A & N\\   0 & B \end{pmatrix})\simeq \mathsf{D}_{\mathsf{sg}}(A\!\times\! B)\simeq \mathsf{D}_{\mathsf{sg}}(A).
    \] 
    Compare; \cite[Theorem 4.1]{chen} and \cite[Corollary 8.4]{PSS}. 
\end{exmp}

\begin{exmp} \label{equivalence_for_trivial_extensions}
    Let $\Gamma$ be a finite dimensional algebra over a field $k$ and $M$ a $\Gamma$-bimodule that is finitely generated on both sides. Then, the trivial extension $\Gamma\!\ltimes\! M$ is also finite dimensional over $k$. Assume that $_{\Gamma}M_{\Gamma}$ is nilpotent and projective (as a bimodule). The functor $\mathsf{F}=-\otimes_{\Gamma}M$ is perfect and nilpotent and moreover $\mathsf{ImF}\subseteq \Proj\Gamma$. It follows by Lemma \ref{left_derived_of_F_vanishes} that $\mathbb{L}_{\mathsf{sg}}\mathsf{F}=0$ in $\mathsf{D}_{\mathsf{sg}}(\Gamma)$. By the above, we may apply Corollary \ref{equivalence_singularity_for_mod} to obtain the following equivalence: 
    \[
    \mathsf{D}_{\mathsf{sg}}(\Gamma\!\ltimes\! M)\simeq \mathsf{D}_{\mathsf{sg}}(\Gamma).
    \]
    Compare; \cite[Corollary 4.4]{arrow2}. 
\end{exmp}

\subsection{Singular equivalences of Morita type with level} 
Recently, in \cite{qin2}, ''perfectness'' assumptions (in the sense of Definition \ref{perfect_bimodule}) were used to prove singular equivalences for extensions of algebras and more precisely singular equivalences ``of Morita type with level'' in the sense of Wang \cite{wang}. In this section we recall this notion, translate the results of \cite{qin2} to the context of $\theta$-extensions and compare with our work. 

Given an algebra $\Gamma$ over a commutative ring $k$, we denote by $\Gamma^e$ the enveloping algebra of $\Gamma$, that is the tensor product algebra $\Gamma\otimes_k\Gamma^{\mathsf{op}}$. 

\begin{defn} \label{morita_type}(\!\!\cite[Definition 2.1]{wang})
    Let $k$ be a commutative ring and let $\Gamma$ and $\Lambda$ be two $k$-algebras which are projective as $k$-modules. Let $M$ be a $\Lambda$-$\Gamma$-bimodule and $N$ a $\Gamma$-$\Lambda$-bimodule. We say that $(M,N)$ defines a \emph{singular equivalence of Morita type with level $n$}, for some nonnegative integer $n$, if the following conditions are satisfied: 
    \begin{itemize}
        \item[(i)] $M$ is finitely generated and projective on both sides. 
        \item[(ii)] $N$ is finitely generated and projective on both sides. 
        \item[(iii)] $N\otimes_{\Lambda}M\cong \Omega^n_{\Gamma^e}(\Gamma)$ in $\umod\Gamma^e$ and $M\otimes_{\Gamma}N\cong \Omega^n_{\Lambda^e}(\Lambda)$ in $\umod\Lambda^e$. 
    \end{itemize}
\end{defn}

\begin{rem} \label{morita_type_implies_equivalence} (\!\!\cite[Remark 2.2]{wang})
    If $\Gamma$ and $\Lambda$ are finite dimensional algebras over a field and $(_{\Lambda}M_{\Gamma},_{\Gamma}N_{\Lambda})$ defines a singular equivalence of Morita type with level between them, then the functor $-\otimes_{\Lambda}M\colon\mathsf{D}_{\mathsf{sg}}(\Lambda)\rightarrow \mathsf{D}_{\mathsf{sg}}(\Gamma)$ is a triangle equivalence. A proof can be found in \cite[Proposition 4.2]{dalezios}. 
\end{rem}

From now on, we denote by $\Gamma$ and $\Lambda$ two finite dimensional algebras over a field. Recall that $\Lambda$ is an \emph{extension} of $\Gamma$ if $\Gamma\subseteq \Lambda$ and $\Gamma$ is a subalgebra of $\Lambda$. 

\begin{prop} \label{qin} \textnormal{(\!\!\cite[Theorem 3.7]{qin2})}
    Let $\Gamma\subseteq \Lambda$ be an extension of algebras. Assume that the following are satisfied: 
    \begin{itemize}
        \item[(i)] $\mathsf{Tor}_i^{\Gamma}((\Lambda/\Gamma),(\Lambda/\Gamma)^{\otimes j})=0$ for all $i,j\geq 1$. 
        \item[(ii)] $\pd{_{\Gamma^e}\Lambda/\Gamma}<\infty$. 
        \item[(iii)] $\Lambda/\Gamma$ is a nilpotent $\Gamma$-bimodule. 
    \end{itemize}
    Then there is a singular equivalence of Morita type with level between $\Gamma$ and $\Lambda$. 
\end{prop}

Consider a $\theta$-extension $\Lambda=\Gamma\!\ltimes\!_{\theta}M$. Then, the algebra $\Lambda$ is an extension of $\Gamma$ and the $\Gamma$-bimodule $\Lambda/\Gamma$ is isomorphic to $M$. Therefore, Proposition \ref{qin} applies (see also \cite[Proposition 6.3]{qin2}).

\begin{cor} \label{equivalence_of_morita_type}
    Let $\Gamma$ be a finite dimensional algebra over a field and consider a $\theta$-extension $\Lambda=\Gamma\ltimes_{\theta}M$ where $M$ is a finitely generated (on both sides) $\Gamma$-bimodule. Assume that the following are satisfied: 
    \begin{itemize}
        \item[(i)] $\mathsf{Tor}_i^{\Gamma}(M,M^{\otimes j})=0$ for all $i,j\geq 1$. 
        \item[(ii)] $\pd{_{\Gamma^e}M}<\infty$. 
        \item[(iii)] $M$ is nilpotent. 
    \end{itemize}
    Then there is a singular equivalence of Morita type with level between $\Gamma$ and $\Lambda$. 
\end{cor}

In order to compare the above with the singular equivalence of Corollary \ref{equivalence_singularity_for_mod}, we present the following lemma, see \cite[Lemma 3.5]{dalezios} or \cite[Proposition 3.7]{OPS}. 

\begin{lem}  \label{apo_to_apeiro}
    Let $\Gamma$ be a finite dimensional algebra over a field and $M$ a $\Gamma$-bimodule that is finitely generated on both sides. If $\pd{_{\Gamma^e}M}<\infty$, then 
    \begin{itemize}
        \item[(i)] $X\otimes_{\Gamma}^{\mathbb{L}}M\in\mathsf{K}^{\mathsf{b}}(\Proj\Gamma)$ for every $X\in\mathsf{D}^{\mathsf{b}}(\Mod\Gamma)$.
        \item[(ii)] $X\otimes_{\Gamma}^{\mathbb{L}}M\in\mathsf{K}^{\mathsf{b}}(\proj\Gamma)$ for every $X\in\mathsf{D}^{\mathsf{b}}(\smod\Gamma)$. 
    \end{itemize} 
\end{lem}

Notice that if the bimodule $M$ satisfies the conditions of Corollary \ref{equivalence_of_morita_type}, then it is perfect. Moreover, by Lemma \ref{apo_to_apeiro}, the assumption $\pd{_{\Gamma^e}M}<\infty$ implies that the induced functor $-\otimes^{\mathbb{L}}_{\Gamma}M\colon \mathsf{D}_{\mathsf{sg}}(\Gamma)\rightarrow\mathsf{D}_{\mathsf{sg}}(\Gamma)$ is 0. That is, the singular equivalence of Corollary \ref{equivalence_of_morita_type} is a special case of Corollary \ref{equivalence_singularity_for_mod}.

\begin{rem}
    There are singular equivalences in our context that are not obtained by Corollary \ref{equivalence_of_morita_type}, for instance Example \ref{equivalence_for_triangular}. However, in view of \cite[Theorem]{dalezios}, \cite[Theorem 1.1]{qin1} and Remark \ref{canonical_form}, given any equivalence $\mathsf{D}_{\mathsf{sg}}(\Lambda)\simeq \mathsf{D}_{\mathsf{sg}}(\Gamma)$ obtained by Corollary \ref{equivalence_singularity_for_mod}, we may infer that there is a singular equivalence of Morita type with level between $\Lambda$ and $\Gamma$, provided that $\Lambda/\mathsf{rad}(\Lambda)$ and $\Gamma/\mathsf{rad}(\Lambda)$ are separable over $k$ (this happens, for instance, whenever $k$ is algebraically closed). When $\Gamma/\mathsf{rad}(\Gamma)$ is separable over $k$, then $\pd{_{\Gamma^e}}\Gamma<\infty$ if and only if $\gd \Gamma<\infty$ (see for instance \cite{han} and in particular \cite[Remark 3]{han} for a thorough discussion and references on this). 
\end{rem}

Singular equivalences of Morita type with level deserve their own attention: several properties are invariant under such an equivalence, for instance Hochschild homology (see \cite[Proposition 3.7]{wang}) and the $\mathsf{Fg}$ condition for Hochschild cohomology (see \cite[Theorem 7.3]{fg_condition}) in the sense of Snashall-Solberg \cite{snashall_solberg}. More equivalences of this type are proved in \cite{chen_liu_wang,dalezios,qin1}.

\section{Gorenstein projective modules} We compare the Gorenstein projective objects in a cleft extension. We explain the existence of the upper part of the diagram of Theorem A of the Introduction and obtain equivalences of stable categories. We end this section with applications regarding CM-free rings and algebras of finite Cohen-Macaulay type.

\subsection{Gorenstein homological algebra} \label{gorenstein_homological_algebra}

We denote by $\mathcal{A}$ an abelian category with enough projective objects. An acyclic complex $P^{\bullet}$, of projective objects of $\mathcal{A}$, is called \emph{totally acyclic} if $\mathsf{Hom}_{\mathcal{A}}(P^{\bullet},Q)$ is acyclic for every $Q\in\Proj\mathcal{A}$. The following was introduced in \cite{enochs_jenda}. 

\begin{defn}
    An object $X$ of $\mathcal{A}$ is called \emph{Gorenstein projective} if there is a totally acyclic complex $P^{\bullet}$ over $\Proj\mathcal{A}$ with $X=\mathsf{B}^0(P^{\bullet})\coloneqq\mathsf{Im}(P^{-1}\rightarrow P^0)$.
\end{defn}

We denote by $\GProj\mathcal{A}$ the category of Gorenstein projective objects of $\mathcal{A}$. This is an additive and extension closed subcategory of $\mathcal{A}$. In particular, it is an exact category. Furthermore, it is Frobenius with the projective-injective objects being the projective objects of $\mathcal{A}$. Consequently, by a theorem of Happel \cite{happel}, the stable category $\uGProj\mathcal{A}$ is triangulated. The latter admits a description via homotopy categories (see e.g. \cite[Proposition 7.2]{krause}), which we recall below.

\begin{prop}
    Given a totally acyclic complex $P^{\bullet}$ over $\Proj\mathcal{A}$, the assignment $P^{\bullet}\mapsto\mathsf{B}^0(P^{\bullet})$ induces an equivalence of triangulated categories 
    \[
    \mathsf{K}_{\mathsf{tac}}(\Proj\mathcal{A})\simeq \uGProj\mathcal{A},
    \]
    where $\mathsf{K}_{\mathsf{tac}}(\Proj\mathcal{A})$ denotes the full subcategory of $\mathsf{K}(\Proj\mathcal{A})$ that consists of the totally acyclic complexes.
\end{prop}

Denote by $\mathsf{K}^{-,\mathsf{b}}(\Proj\mathcal{A})$ the full subcategory of $\mathsf{K}^{-}(\Proj\mathcal{A})$ that consists of the complexes that are acyclic almost everywhere. Upon identification of $\uGProj\mathcal{A}$ with $\mathsf{K}_{\mathsf{tac}}(\Proj\mathcal{A})$ and of $\mathsf{D}^{\mathsf{b}}(\mathcal{A})$ with $\mathsf{K}^{-,\mathsf{b}}(\Proj\mathcal{A})$, we obtain a triangle functor 
\[
i\colon\uGProj\mathcal{A}\rightarrow\mathsf{D}_{\mathsf{sg}}(\mathcal{A}),
\]
where a totally acyclic complex $P^{\bullet}$ over $\Proj\mathcal{A}$ is mapped to the object of $\mathsf{D}_{\mathsf{sg}}(\mathcal{A})$ represented by the truncation $\tau_0P^{\bullet}$. The functor $i$ is fully faithful (see for instance \cite[Proposition 4.9]{relative_sing_cats}). 

For a ring $\Lambda$, we denote by $\GProj\Lambda$ the Gorenstein projective objects of $\Mod \Lambda$, which are simply called the \emph{Gorenstein projective modules} of $\Lambda$. Further, we denote by $\Gproj\Lambda$ the Gorenstein projective objects of $\smod\Lambda$, which are often called the \emph{Cohen-Macaulay} $\Lambda$-modules when $\Lambda$ is an Iwanaga-Gorenstein algebra. We have 
\[
\Gproj\Lambda=\GProj\Lambda\cap\smod\Lambda,
\]
provided that $\Lambda$ is right Noetherian, see \cite[Lemma 3.4]{chen3}. In this case, the triangle functor $i\colon \uGProj \Lambda\rightarrow \mathsf{D}_{\mathsf{sg}}(\Mod \Lambda)$ restricts to $i\colon\uGproj\Lambda\rightarrow \mathsf{D}_{\mathsf{sg}}(\Lambda)$. The following is a fundamental theorem of Buchweitz \cite[Theorem 4.4.1]{buchweitz} and it was obtaind independently by Happel \cite{happel2}.

\begin{thm} \label{buchweitz}
    Let $\Lambda$ be a Noetherian ring. If $\Lambda$ is right Gorenstein, then
    \[
    i\colon\uGproj\Lambda\rightarrow \mathsf{D}_{\mathsf{sg}}(\Lambda)
    \]
    is a triangle equivalence.  
\end{thm}

Beligiannis \cite{beligiannis3} has proved a version of the above for abelian categories, see \cite[Theorem 4.16]{beligiannis3}. For instance, it follows that the analogous of Theorem \ref{buchweitz} holds for $\Mod \Lambda$ and in this case, in fact, the functor $i\colon \uGProj\Lambda\rightarrow \mathsf{D}_{\mathsf{sg}}(\Mod\Lambda)$ is a triangle equivalence if \emph{and only if} $\Lambda$ is right Gorenstein (see \cite[Theorem 6.9]{beligiannis3}). The converse to Theorem \ref{buchweitz} holds in some cases, for instance for Artin algebras, see \cite{bergh_jorgensen_oppermann}.

\subsection{Gorenstein projective modules in cleft extensions}  \label{stable_cat_of_gorenstein_proj}

The following is a reformulation of \cite[Proposition 3.4]{OPS} stated for abelian categories. The proof is essentially the same, but we recall it for the reader's convenience.

\begin{prop} \label{adjoint_pair_and_gorenstein_projectives}
    Let $\mathcal{A},\mathcal{B}$ be abelian categories with enough projective objects and assume that there is an adjoint pair of functors
    \[
    \begin{tikzcd}
\mathcal{A} \arrow[rr, "\mathsf{e}"] &  & \mathcal{B} \arrow[ll, "\mathsf{l}"', bend right]
\end{tikzcd}
    \]
   where $\mathsf{e}$ is exact. Assume further that the following are satisfied:  
   \begin{itemize}
       \item[(i)] $\mathbb{L}_i\mathsf{l}=0$ for $i>\!\!>0$. 
       \item[(ii)] $\pd{_{\mathcal{B}}\mathsf{e}(P)}\leq d$ for all $P\in\Proj\mathcal{A}$ and some $d$.  
   \end{itemize}
   Then $\mathsf{l}(\GProj\mathcal{B})\subseteq \GProj\mathcal{A}$ and $\mathbb{L}_i\mathsf{l}(X)=0$ for every $X\in\GProj\mathcal{B}$ and $i\geq 1$. 
\end{prop}
\begin{proof}
    Consider $X\in\GProj\mathcal{B}$ and let $P^{\bullet}$ be a totally acyclic complex over $\Proj\mathcal{B}$ with $X=\mathsf{B}^0(P^{\bullet})$. For all $i$, we have: 
    \[
    \mathsf{H}^i(\mathsf{l}(P^{\bullet}))\cong \mathbb{L}_{j-i}\mathsf{l}(\mathsf{B}^{j+1}(P^{\bullet})) \text{ for all } j>i.
    \]
    By the assumption on the left derived functors of $\mathsf{l}$, it follows that $\mathsf{l}(P^{\bullet})$ is acyclic and in particular $\mathsf{l}(X)=\mathsf{B}^0(\mathsf{l}(P^{\bullet}))$ and $\mathbb{L}_i\mathsf{l}(X)=0$ for $i\geq 1$. We are left to show that $\mathsf{l}({P^{\bullet}})$ is coacyclic. Let $Q$ be a projective object of $\mathcal{A}$. Then, by adjunction, $\mathsf{Hom}_{\mathcal{A}}(\mathsf{l}(P^{\bullet}),Q)$ is isomorphic to $\mathsf{Hom}_{\mathcal{B}}(P^{\bullet},\mathsf{e}(Q))$ as a complex of groups. Since $\mathsf{e}(Q)$ has finite projective dimension (i.e.\ belongs in $\mathsf{K}^{\mathsf{b}}(\Proj\mathcal{B})$) and $P^{\bullet}$ is coacyclic, it follows that $\mathsf{Hom}_{\mathcal{B}}(P^{\bullet},\mathsf{e}(Q))$ is acyclic, which completes the proof. 
\end{proof}

\begin{cor} \label{gorenstein_proj_in_B}
    Let $(\mathcal{B},\mathcal{A},\mathsf{i},\mathsf{e},\mathsf{l})$ be a cleft extension of abelian categories such that $\mathsf{F}$ is perfect and nilpotent. Assume that $\pd{_{\mathcal{B}}\mathsf{F}(P)}\leq d$ for every $P\in\Proj\mathcal{B}$ and some $d$. The following hold: 
    \begin{itemize}
        \item[(i)] $\mathsf{l}(\GProj\mathcal{B})\subseteq \GProj\mathcal{A}$. 
        \item[(ii)] $\mathsf{q}(\GProj\mathcal{A})\subseteq \GProj\mathcal{B}$. 
    \end{itemize}
    In particular, $X\in\GProj\mathcal{B}$ if and only if $\mathsf{l}(X)\in\GProj\mathcal{A}$. 
\end{cor}
\begin{proof} (i) By Remark \ref{preserves_and_reflects_bounds}, we may apply Proposition \ref{adjoint_pair_and_gorenstein_projectives} to the adjoint pair $(\mathsf{l},\mathsf{e})$. 

   (ii) By Remark \ref{preserves_and_reflects_bounds}, Proposition \ref{derived_of_q_vanishes} and Lemma \ref{i_maps_proj}, we may apply Proposition \ref{adjoint_pair_and_gorenstein_projectives} to the adjoint pair $(\mathsf{q},\mathsf{i})$. 

   In particular, if $X\in\GProj\mathcal{B}$, then $\mathsf{l}(X)\in\GProj\mathcal{A}$ and if $\mathsf{l}(X)\in\GProj\mathcal{A}$, then $X\cong \mathsf{ql}(X)\in\GProj\mathcal{B}$, which proves the last claim. 
\end{proof}

Of course, from the above - and under the same assumptions - it follows that the functors $\mathsf{l}$ and $\mathsf{q}$ induce functors between the respective stable categories of Gorenstein projective objects. We are interested in instances where the functors $\mathsf{i}$ and $\mathsf{e}$ induce triangle functors on the level of stable categories of Gorenstein projectives. To achieve this, we introduce some notation. 

Recall that the category $\mathsf{C}(\mathcal{A})$ of chain complexes over $\mathcal{A}$ is abelian with enough projective objects, since we assume the same for $\mathcal{A}$. In particular, for every complex $X\in\mathsf{C}(\mathcal{A})$, there is a contractible complex $P_{X}$ that consists of projective objects of $\mathcal{A}$ with an epimorphism $P_X\twoheadrightarrow X$. We denote $\mathsf{ker}(P_{X}\twoheadrightarrow X)$ by $\Omega(X)$. Further, we use $\tilde{\Omega}$ to denote $\Omega\circ[1]$. 

\begin{lem} \label{cone}
    Let $X$ be an acyclic complex over $\mathcal{A}$. Assume that $\pd{_{\mathcal{A}}X^n}\leq d$ for all $n$ and some $d$. Then, for all $i$ there is a morphism $\phi\colon\mathsf{B}^i(\tilde{\Omega}^d(X))\rightarrow \mathsf{B}^i(X)$ such that $\mathsf{cone}(\phi)$ is quasi-isomorphic to a complex of projective objects concentrated in degrees $-d$ to 0.
\end{lem}
\begin{proof}
    This is proved in \cite[Lemma 3.6]{OPS} for $\mathcal{A}=\smod \Lambda$ (and $\Lambda$ Noetherian). The exact same proof works in our setting. 
\end{proof}

The following is a reformulation of \cite[Proposition 3.5]{OPS} for abelian categories, see also \cite[Theorem 5.3]{hu_pan}. Our proof is a slight variation of the one found in \cite{OPS}.  

\begin{prop} \label{sthnk}
    Let $\mathcal{A}$ and $\mathcal{B}$ be abelian categories with enough projective objects, and assume that there is an adjoint pair of functors 
    \[
    \begin{tikzcd}
\mathcal{A} \arrow[rr, "\mathsf{e}"] &  & \mathcal{B} \arrow[ll, "\mathsf{r}", bend left]
\end{tikzcd}
    \]
where $\mathsf{e}$ is exact. Assume further that the following are satisfied: 
\begin{itemize}
    \item[(i)] $\pd{_{\mathcal{B}}\mathsf{e}(P)}\leq d$ for all $P\in\Proj\mathcal{A}$ and some $d$. 
    \item[(ii)] $\mathbb{R}\mathsf{r}(Q)\in\mathsf{K}^{\mathsf{b}}(\Proj\mathcal{A})$ for every $Q\in\Proj\mathcal{B}$. 
\end{itemize}
Then the functor $\mathsf{e}$ induces a triangle functor $\underline{\mathsf{e}}\colon\uGProj\mathcal{A}\rightarrow \uGProj\mathcal{B}$.
\end{prop}
\begin{proof}  
    We will first show that for every $X\in\GProj\mathcal{A}$, the object $\Omega^d\mathsf{e}(X)$ is in $\GProj\mathcal{B}$. Consider a totally acyclic complex $P^{\bullet}$ over $\Proj\mathcal{A}$ with $X\cong \mathsf{B}^d(P^{\bullet})$. We then have the following isomorphisms 
    \[
    \mathsf{e}(X)=\mathsf{e}(\mathsf{B}^d(P^{\bullet}))=\mathsf{B}^d(\mathsf{e}(P^{\bullet}))
    \]
    and consequently, 
    \[
    \Omega^d\mathsf{e}(X)=\Omega^d\mathsf{B}^d(\mathsf{e}(P^{\bullet}))=\mathsf{B}^0(\tilde{\Omega}^d\mathsf{e}(P^{\bullet})).
    \]
    It is clear, by the assumption that $\pd{_{\mathcal{B}}\mathsf{e}(P)}\leq d$ for all $P\in\Proj\mathcal{A}$, that $\tilde{\Omega}^d\mathsf{e}(P^{\bullet})$ is a complex consisting of projective objects in $\mathcal{B}$. We will show that it is in fact totally acyclic. Since $P^{\bullet}$ is acyclic, it follows that $\mathsf{e}(P^{\bullet})$ is acyclic and so is $\tilde{\Omega}^d\mathsf{e}(P^{\bullet})$. We are left to show coacyclicity, i.e.\ that $\mathsf{Hom}_{\mathcal{B}}(\tilde{\Omega}^d\mathsf{e}(P^{\bullet}),Q)$ is acyclic for every $Q\in\Proj\mathcal{B}$. For all $i$ and for $j>\!\!>0$, we have the following isomorphisms: 
    \begin{align*}
        \mathsf{H}^i(\mathsf{Hom}_{\mathcal{B}}(\tilde{\Omega}^d\mathsf{e}(P^{\bullet}),Q)) & \cong \mathsf{Ext}_{\mathcal{B}}^{i+j}(\mathsf{B}^{j+1}(\tilde{\Omega}^d\mathsf{e}(P^{\bullet})),Q) \\ 
                          & \cong \mathsf{Ext}_{\mathcal{B}}^{i+j}(\mathsf{B}^{j+1}(\mathsf{e}(P^{\bullet})),Q) \\ 
                          & \cong \mathsf{Ext}_{\mathcal{B}}^{i+j}(\mathsf{e}(\mathsf{B}^{j+1}(P^{\bullet})),Q) \\ 
                          & \cong \mathsf{Ext}_{\mathcal{A}}^{i+j}(\mathsf{B}^{j+1}(P^{\bullet}),\mathbb{R}\mathsf{r}(Q)).
    \end{align*}
    The second isomorphism follows from Lemma \ref{cone}. Since $P^{\bullet}$ is totally acyclic and $\mathbb{R}\mathsf{r}(Q)\in\mathsf{K}^{\mathsf{b}}(\Proj\mathcal{A})$, the last Ext group vanishes. We have managed to show that $\Omega^d\mathsf{e}(X)$ is Gorenstein projective, i.e.\ $\Omega^d\mathsf{e}\colon\uGProj\mathcal{A}\rightarrow \uGProj\mathcal{B}$ is well-defined. Observing that the following diagram is commutative, 
    \[
    \begin{tikzcd}
\uGProj\mathcal{A} \arrow[d, hook] \arrow[r, "\Omega^d\mathsf{e}"]       & \uGProj\mathcal{B} \arrow[d, hook]    \\
\mathsf{D}_{\mathsf{sg}}(\mathcal{A}) \arrow[r, "{[-d]\circ\mathsf{e}}"] & \mathsf{D}_{\mathsf{sg}}(\mathcal{B})
\end{tikzcd}
    \]
    shows that $\Omega^d\mathsf{e}\colon\uGProj\mathcal{A}\rightarrow \uGProj\mathcal{B}$ is a triangle functor. Composing the latter with $\Omega^{-d}:=(\Omega^{-1})^d$, where $\Omega^{-1}$ denotes the shift functor of $\uGProj\mathcal{B}$,
    \[
\begin{tikzcd}
\uGProj\mathcal{A} \arrow[d, hook] \arrow[r, "\Omega^d\mathsf{e}"] \arrow[rr, "\underline{\mathsf{e}}", bend left] & \uGProj\mathcal{B} \arrow[d, hook] \arrow[r, "\Omega^{-d}"] & \uGProj\mathcal{B} \arrow[d, hook]    \\
\mathsf{D}_{\mathsf{sg}}(\mathcal{A}) \arrow[r, "{[-d]\circ\mathsf{e}}"] \arrow[rr, "\mathsf{e}"', bend right]     & \mathsf{D}_{\mathsf{sg}}(\mathcal{B}) \arrow[r, "{[d]}"]    & \mathsf{D}_{\mathsf{sg}}(\mathcal{B})
\end{tikzcd}
    \]
    yields the desired result. 
\end{proof}

Given an exact functor $\mathsf{e}\colon\mathcal{A}\rightarrow \mathcal{B}$ of abelian categories, we will say that it induces a functor $\uGProj\mathcal{A}\rightarrow \uGProj\mathcal{B}$, which we denote by $\underline{\mathsf{e}}$, if the latter occurs as a restriction of $\mathsf{e}\colon\mathsf{D}_{\mathsf{sg}}(\mathcal{A})\rightarrow \mathsf{D}_{\mathsf{sg}}(\mathcal{B})$ up to equivalence in $\mathsf{D}_{\mathsf{sg}}(\mathcal{B})$. An other source of examples of such a situation arises from the following proposition. 

\begin{prop} \label{sthnk2}
    Let $(\mathcal{B},\mathcal{A},\mathsf{i},\mathsf{e},\mathsf{l})$ be a cleft extension of abelian categories. If the functor $\mathsf{e}$ is an eventually homological isomorphism, then it gives rise to a triangle functor $\underline{\mathsf{e}}\colon\uGProj\mathcal{A}\rightarrow \uGProj\mathcal{B}$.
\end{prop}
\begin{proof}
    First off, if we assume $\mathsf{e}$ to be an eventually homological isomorphism, then $\mathsf{e}(P)\leq d$ for every $P\in\Proj\mathcal{B}$ and some $d$. Indeed, for such $P$ we have 
    \[
    \mathsf{Ext}^i_{\mathsf{A}}(P,X)\cong \mathsf{Ext}^{i}_{\mathcal{B}}(\mathsf{e}(P),\mathsf{e}(X))
    \]
    for all $i>d$, for some $d$, from which follows that $\pd{_{\mathcal{B}}\mathsf{e}(P)}\leq d$. Once we show that $\Omega^d\mathsf{e}(X)\in\GProj\mathcal{B}$ for every $X\in\GProj\mathcal{A}$, then the proof continues as in the proof of Proposition \ref{sthnk}. Let $P^{\bullet}$ be a totally acyclic complex over $\Proj\mathcal{A}$ with $X=\mathsf{B}^d(P^{\bullet})$. As in Proposition \ref{sthnk}, $\Omega^d\mathsf{e}(X)=\mathsf{B}^0(\tilde{\Omega}^d\mathsf{e}(P^{\bullet}))$ and by the observation we began with, $\tilde{\Omega}^d\mathsf{e}(P^{\bullet})$ is a complex over $\Proj\mathcal{B}$. We must show that the latter is totally acyclic. Acyclicity follows as in Proposition \ref{sthnk}. For coacyclicity, given $Q\in\Proj\mathcal{B}$, for all $i$ and $j>\!\!>0$ we have
    \[
    \mathsf{H}^i(\mathsf{Hom}_{\mathcal{B}}(\tilde{\Omega}^d\mathsf{e}(P^{\bullet}),Q))  \cong \mathsf{Ext}_{\mathcal{B}}^{i+j}(\mathsf{e}(\mathsf{B}^{j+1}(P^{\bullet})),Q)   
    \]
    It is enough to show that $\mathsf{Ext}_{\mathcal{B}}^{i+j}(\mathsf{e}(\mathsf{B}^{j+1}(P^{\bullet})),\mathsf{el}(Q))\cong 0$, since $Q$ is a direct summand of $\mathsf{el}(Q)$. However, the functor $\mathsf{e}$ is an eventually homological isomorphism, so we infer
    \[
    \mathsf{Ext}_{\mathcal{B}}^{i+j}(\mathsf{e}(\mathsf{B}^{j+1}(P^{\bullet})),\mathsf{el}(Q))\cong \mathsf{Ext}_{\mathcal{A}}^{i+j}(\mathsf{B}^{j+1}(P^{\bullet}),\mathsf{l}(Q))
    \]
    and the right-hand side vanishes since $\mathsf{l}(Q)\in\Proj\mathcal{A}$ and $P^{\bullet}$ is totally acyclic.
\end{proof}

We can produce examples satisfying the above using Theorem \ref{thm2}. Let us point out that investigating connections between eventually homological isomorphisms and Gorenstein projective objects is not new, see \cite{qin3}.

\begin{cor} \label{upper_part_of_diagram}
    Let $(\mathcal{B},\mathcal{A},\mathsf{i},\mathsf{e},\mathsf{l})$ be a cleft extension of abelian categories that is also a cleft coextension. Assume that the functor $\mathsf{F}$ is perfect and nilpotent such that $\pd{_{\mathcal{B}}\mathsf{F}(P)}\leq d$ for every $P\in\Proj\mathcal{A}$ and some $d$. Then there are triangle functors
    \[
\begin{tikzcd}
\uGProj\mathcal{B} \arrow[rr, "\underline{\mathsf{i}}", dashed] &  & \uGProj\mathcal{A} \arrow[rr, "\underline{\mathsf{e}}", dashed] \arrow[ll, "\mathsf{q}"', bend right] &  & \uGProj\mathcal{B} \arrow[ll, "\mathsf{l}"', bend right]
\end{tikzcd}
    \]
    with the dashed arrows existing provided that $\mathbb{R}\mathsf{r}(\mathsf{K}^{\mathsf{b}}(\Proj\mathcal{B}))\subseteq \mathsf{K}^{\mathsf{b}}(\Proj\mathcal{A})$ and $\mathbb{R}\mathsf{p}(\mathsf{K}^{\mathsf{b}}(\Proj\mathcal{A}))\subseteq \mathsf{K}^{\mathsf{b}}(\Proj\mathcal{B})$, in which case  $(\mathsf{q},\underline{\mathsf{i}})$ and $(\mathsf{l},\underline{\mathsf{e}})$ are adjoint pairs and $\underline{\mathsf{e}}\underline{\mathsf{i}}\simeq \mathsf{Id}_{\uGProj\mathcal{B}}$.
\end{cor}
\begin{proof}
    The first part is a direct consequence of Corollary \ref{gorenstein_proj_in_B}, while the second follows by combining Proposition \ref{sthnk} with Remark \ref{preserves_and_reflects_bounds}. The fact that $(\mathsf{q},\underline{\mathsf{i}})$ and $(\mathsf{l},\underline{\mathsf{i}})$ are adjoint pairs (provided that $\underline{\mathsf{i}}$ and $\underline{\mathsf{e}}$ exist) follows by the fact that $(\mathbb{L}_{\mathsf{sg}}\mathsf{q},\mathsf{i})$ and $(\mathbb{L}_{\mathsf{sg}}\mathsf{l},\mathsf{e})$ are adjoint pairs on the level of singularity categories. Lastly, $\underline{\mathsf{e}}\underline{\mathsf{i}}$ is a right adjoint to $\mathsf{ql}\simeq \mathsf{Id}_{\uGProj\mathcal{B}}$, thus it follows that $\underline{\mathsf{e}}\underline{\mathsf{i}}\simeq \mathsf{Id}_{\uGProj\mathcal{B}}$. 
\end{proof}

Keep the setting as in the above corollary and recall from Proposition \ref{adjoint_pair_and_gorenstein_projectives} that $\mathbb{L}\mathsf{l}(X)=\mathsf{l}(X)$ for every $X\in\GProj\mathcal{B}$ and $\mathbb{L}\mathsf{q}(Y)=\mathsf{q}(Y)$ for every $Y\in\GProj\mathcal{A}$. Consequently, using Lemma \ref{restriction_of_cleft}, we infer the following, which is precisely the upper part of the commutative diagram of Theorem A. 

\begin{cor} \label{cleft_of_stable}
    Let $(\Mod\Gamma,\Mod\Lambda,\mathsf{i},\mathsf{e},\mathsf{l})$ be a cleft extension of module categories of Noetherian rings. Assume that the endofunctor $\mathsf{F}$ is perfect and nilpotent. Then there is a diagram of triangle functors
    \[
\begin{tikzcd}
\uGproj \Gamma \arrow[rr, "\underline{\mathsf{i}}", dashed] &  & \uGproj\Lambda \arrow[rr, "\underline{\mathsf{e}}", dashed] \arrow[ll, "\mathsf{q}"', bend right] &  & \uGproj\Gamma \arrow[ll, "\mathsf{l}"', bend right]
\end{tikzcd}
    \]
    where the dashed functors exist if $\mathbb{R}\mathsf{r}(\Gamma)\in\mathsf{K}^{\mathsf{b}}(\proj\Lambda)$ and $\mathbb{R}\mathsf{p}(\Lambda)\in\mathsf{K}^{\mathsf{b}}(\proj\Gamma)$. 
\end{cor}

Below we provide sufficient conditions, in terms of the endofunctor $\mathsf{F}'$ of a cleft extension $(\Mod\Gamma,\Mod \Lambda,\mathsf{i},\mathsf{e},\mathsf{l})$, so that the conditions implying the existence of the dashed arrows are met.

\begin{prop} \label{projectives_F'_inj}
    Let $(\Mod\Gamma,\Mod\Lambda,\mathsf{i},\mathsf{e},\mathsf{l})$ be a cleft extension of module categories of Noetherian rings. Assume that $\mathsf{F}$ is perfect and nilpotent. The following hold: 
    \begin{itemize}
        \item[(i)] If $\Gamma$ is $\mathsf{F}'$-injective and $\pd \mathsf{F}'(\Gamma)_{\Gamma}<\infty$, then $\mathbb{R}\mathsf{r}(\Gamma)=\mathsf{r}(\Gamma)\in \mathsf{K}^{\mathsf{b}}(\proj\Lambda)$. 
        \item[(ii)] If $\mathsf{e}(\Lambda)$ is $\mathsf{F}'$-injective and $\pd {\mathsf{F}'}^j\mathsf{e}(\Lambda)_{\Gamma}<\infty$ for every $j\geq 0$, then $\mathbb{R}\mathsf{p}(\Lambda)\in \mathsf{K}^{\mathsf{b}}(\proj\Gamma)$. 
    \end{itemize} 
\end{prop}
\begin{proof} By Lemma \ref{restriction_of_cleft}, the given cleft extension restricts to a cleft extension $(\smod\Gamma,\smod\Lambda,\mathsf{i},\mathsf{e},\mathsf{l})$ of small module categories.

(i) Using Lemma \ref{basic_homological_properties_of_cocleft}, we infer that $\mathsf{e}\mathbb{R}^i\mathsf{r}(\Gamma)\cong \mathbb{R}^i\mathsf{F}'(\Gamma)\cong 0$ and so $\mathbb{R}^i\mathsf{r}(\Gamma)\cong 0$ for all $i\geq 1$. We conclude that $\mathbb{R}\mathsf{r}(\Gamma)=\mathsf{r}(\Gamma)$. On the other hand, $\mathsf{er}\simeq \mathsf{Id}_{\smod\Gamma}\oplus \mathsf{F}'$ and so $\mathsf{er}(\Gamma)\simeq \Gamma\oplus \mathsf{F}'(\Gamma)$. Since $\pd \mathsf{F}'(\Gamma)_{\Gamma}<\infty$, it follows that $\pd\ \!\mathsf{er}(\Gamma)_{\Gamma}<\infty$. Since $\mathsf{F}$ is perfect and nilpotent, we know from Proposition \ref{preserves_and_reflects} that $\mathsf{e}$ reflects modules of finite projective dimension, hence $\mathsf{r}(\Gamma)$ has finite projective dimension. 

(ii) Consider the following short exact sequences in $\smod\Lambda$:
    \[
    0\rightarrow \Lambda\rightarrow \mathsf{re}(\Lambda)\rightarrow \mathsf{G}'(\Lambda)\rightarrow 0, \  0\rightarrow \mathsf{G}'(\Lambda)\rightarrow \mathsf{rF}'\mathsf{e}(\Lambda)\rightarrow \mathsf{G}'^2(\Lambda)\rightarrow 0, \  \dots
    \]
    which in turn give triangles in $\mathsf{D}(\Lambda)$:
    \[
    \Lambda\rightarrow \mathsf{re}(\Lambda)\rightarrow \mathsf{G}'(\Lambda)\rightarrow \Lambda[1], \  \mathsf{G}'(\Lambda)\rightarrow \mathsf{rF}'\mathsf{e}(\Lambda)\rightarrow \mathsf{G}'^2(\Lambda)\rightarrow \mathsf{G}'(\Lambda)[1], \  \dots
    \]
    Applying $\mathbb{R}\mathsf{p}$ to the latter, since $\mathsf{F}'$ is nilpotent (and so is $\mathsf{G}'$), shows that in order to prove the claim, it is enough to show that $\mathbb{R}\mathsf{p}({\mathsf{rF}'}^j\mathsf{e}(\Lambda))\in\mathsf{K}^{\mathsf{b}}(\proj\Lambda)$ for every $j\geq 0$. By assumption, $\mathsf{e}(\Lambda)$ is $\mathsf{F}'$-injective and so by Lemma \ref{F'_injective}, we also have that $\mathsf{F}'^j\mathsf{e}(\Lambda)$ is $\mathsf{F}'$-injective for every $j\geq 0$. Lemma \ref{basic_homological_properties_of_cocleft} shows that $\mathbb{R}^i\mathsf{r}(\mathsf{F}'^j\mathsf{e}(\Lambda))=0$ for all $i,j\geq 1$. Therefore, in order to compute $\mathbb{R}^i\mathsf{p}(\mathsf{r}\mathsf{F}'^j\mathsf{e}(\Lambda))$, it is enough to begin with an injective resolution of $\mathsf{F}'^j\mathsf{e}(\Lambda)$ and apply $\mathsf{pr}$. However, $\mathsf{pr}\simeq \mathsf{Id}_{\smod\Gamma}$ and so $\mathbb{R}^i\mathsf{p}(\mathsf{rF}'^j\mathsf{e}(\Lambda))=0$ for all $i,j\geq 0$. Consequently, $\mathbb{R}\mathsf{p}(\mathsf{rF}'^j\mathsf{e}(\Lambda))=\mathsf{F}'^j\mathsf{e}(\Lambda)$, which has finite projective dimension by assumption. 
\end{proof}

\begin{exmp}
    Let $A$ and $B$ be Noetherian rings and $N$ an $A$-$B$-bimodule that is finitely generated on both sides. Consider the triangular matrix ring $\Lambda=\big(\begin{smallmatrix}   A & N\\   0 & B \end{smallmatrix}\big)$, which is also Noetherian. Recall from Example \ref{perfect_for_triangular} that $\Mod\Lambda$ is a cleft extension of $\Mod A\!\times\! B$ and the functor $\mathsf{F}$ is given by $(X,Y)\mapsto (0,X\otimes_AN)$. If $\pd{_{A}N}<\infty$ and $\pd N_{B}<\infty$, then by Example \ref{perfect_for_triangular} the functor $\mathsf{F}$ is perfect and nilpotent. Consequently, by Corollary \ref{cleft_of_stable}, the functors $\mathsf{l}\colon\Mod\ A\!\times\! B\rightarrow \Mod\Lambda$ and $\mathsf{q}\colon\Mod\Lambda\rightarrow \Mod A\!\times\! B$ give rise to triangle functors as below:
        \[
    \begin{tikzcd}
\uGproj A\!\times\!B \arrow[rr, "\underline{\mathsf{i}}", dashed] &  & \uGproj\Lambda \arrow[rr, "\underline{\mathsf{e}}", dashed] \arrow[ll, "\mathsf{q}"', bend right] &  & \uGproj A\!\times\!B \arrow[ll, "\mathsf{l}"', bend right]
\end{tikzcd}
    \]
    We will now use Proposition \ref{projectives_F'_inj} to find sufficient conditions for the existence of the dashed arrows above. By the description of $\mathsf{F}$, it follows that the functor $\mathsf{F}'$ is given by $(X,Y)\mapsto (\mathsf{Hom}_B(N,Y),0)$ and in particular it follows that $\mathbb{R}^i\mathsf{F}'(X,Y)=(\mathsf{Ext}^i_B(N,Y),0)$. Assume that $\gd A<\infty$ and that $N_B$ is projective. Then $\mathbb{R}^i\mathsf{F}'(A,B)=(\mathsf{Ext}_B^i(N,B),0)=0$ for $i\geq 1$, meaning that $(A,B)$ is $\mathsf{F}'$-injective. Further, $\pd\mathsf{F}'(A,B)_{A\times B}=\pd{(\mathsf{Hom}_B(N,B),0)}_{A\times B}\leq \gd A$. Hence, the assumptions of Proposition \ref{projectives_F'_inj}(i) are met. Note that the assumptions of Proposition \ref{projectives_F'_inj}(ii) are also satisfied. Indeed, $\mathbb{R}^i\mathsf{F}'(\mathsf{e}(\Lambda))=(\mathsf{Ext}^i_B(N,N\oplus B),0)= 0$ for $i\geq 1$, meaning that $\mathsf{e}(\Lambda)$ is $\mathsf{F}'$-injective. Further, as before, the projective dimensions of ${\mathsf{F}'}^j(\mathsf{e}(\Lambda))$ are bounded by $\gd A$. In conclusion, if $\gd A<\infty$ and $N_B$ is projective, then the functors $\underline{\mathsf{i}}$ and $\underline{\mathsf{e}}$ exist, in which case $(\mathsf{q},\underline{\mathsf{i}})$ and  $(\mathsf{l},\underline{\mathsf{e}})$ are adjoint pairs and $\underline{\mathsf{e}}\underline{\mathsf{i}}\simeq \mathsf{Id}_{\uGproj A\!\times\!B}$. 
\end{exmp}

\begin{exmp} \label{stable_gorenstein_of_arrow}
    Let $\Lambda=\Gamma\ltimes M$ be a trivial extension of a finite dimensional algebra $\Gamma$ over a field, where $M$ is perfect and nilpotent $\Gamma$-bimodule that is finitely generated on both sides, for instance an arrow removal \cite{arrow}. Then, the functors $\mathsf{l}\colon\Mod\Gamma\rightarrow \Mod\Lambda$ and $\mathsf{q}\colon\Mod\Lambda\rightarrow \Mod\Gamma$ of the usual cleft extension give rise to triangle functors as below: 
    \[
    \begin{tikzcd}
\uGproj\Gamma \arrow[rr, "\underline{\mathsf{i}}", dashed] &  & \uGproj\Gamma\ltimes M \arrow[rr, "\underline{\mathsf{e}}", dashed] \arrow[ll, "\mathsf{q}"', bend right] &  & \uGproj\Gamma \arrow[ll, "\mathsf{l}"', bend right]
\end{tikzcd}
    \]
    We shall now find sufficient conditions for the existence of the dashed arrows. For this, we assume $_{\Gamma}M_{\Gamma}$ to be a projective bimodule. Then, from Theorem \ref{thm2}, the functor $\mathsf{e}$ of the cleft extension $(\Mod\Lambda,\Mod\Gamma,\mathsf{i},\mathsf{e},\mathsf{l})$ is an eventually homological isomorphism (see also \cite[Corollary 3.3]{arrow2}). Consequently, Proposition \ref{sthnk2} yields the existence of the functor $\underline{\mathsf{e}}$. The existence of the functor $\underline{\mathsf{i}}$, which follows from the fact that the functor $\mathsf{e}$ maps projectives to projectives, is left to the reader to verify. In this case, $(\mathsf{l},\underline{\mathsf{e}})$ and $(\mathsf{q},\underline{\mathsf{i}})$ are adjoint pairs and $\underline{\mathsf{e}}\underline{\mathsf{i}}\simeq \mathsf{Id}_{\uGproj\Gamma}$. We will see in Proposition \ref{equivalence_of_stable_for_theta} that the functor $\underline{\mathsf{e}}$ is, in fact, an equivalence. 
\end{exmp}

\subsection{An Equivalence of stable categories of Gorenstein projectives} \label{equivalence of gorenstein projectives section}

In Example \ref{stable_gorenstein_of_arrow}, in fact, the functor $\mathsf{e}$ is an equivalence of stable categories of Gorenstein projective modules, as shown below in greater generality. First, assume the existence of a cleft extension $(\Mod\Gamma,\Mod\Lambda,\mathsf{i},\mathsf{e},\mathsf{l})$ of Noetherian rings, such that $\mathsf{F}$ is perfect and nilpotent and $\mathsf{e}$ induces a triangle functor $\underline{\mathsf{e}}\colon\uGproj\Lambda\rightarrow \uGproj\Gamma$. For a Gorenstein projective $\Gamma$-module $X$, we have $\mathbb{L}_{\mathsf{sg}}\mathsf{l}(X)=\mathsf{l}(X)$. Consequently, the module $\mathsf{el}(X)$ is, up to isomorphism in $\mathsf{D}_{\mathsf{sg}}(\Gamma)$, Gorenstein projective. It follows from the isomorphism $\mathsf{el}(X)\cong X\oplus \mathsf{F}(X)$, that $\mathsf{F}(X)$ is, up to isomorphism in $\mathsf{D}_{\mathsf{sg}}(\Gamma)$, Gorenstein projective. In particular, the functor $\mathbb{L}_{\mathsf{sg}}\mathsf{F}\colon\mathsf{D}_{\mathsf{sg}}(\Gamma)\rightarrow \mathsf{D}_{\mathsf{sg}}(\Gamma)$ restricts (up to isomorphism) to a functor $\underline{\mathsf{F}}\colon\uGproj\Gamma\rightarrow \uGproj\Gamma$. 

\begin{prop} \label{equivalence_of_stable}
    Consider a cleft extension $(\Mod\Gamma,\Mod\Lambda,\mathsf{i},\mathsf{e},\mathsf{l})$ of module categories of Noetherian rings. Assume that $\mathsf{F}$ is perfect and nilpotent and that the functor $\mathsf{e}$ induces a triangle functor $\underline{\mathsf{e}}\colon\uGproj\Lambda\rightarrow \uGproj\Gamma$. Then the following are equivalent: 
    \begin{itemize}
        \item[(i)] The functor $\underline{\mathsf{e}}\colon\uGproj\Lambda\rightarrow \uGproj\Gamma$ is an equivalence. 
        \item[(ii)] $\underline{\mathsf{F}}\simeq 0$ in $\uGproj\Gamma$. 
    \end{itemize}
\end{prop}
\begin{proof}
    First, observe that $\underline{\mathsf{e}}\mathsf{l}\simeq \mathsf{Id}_{\uGproj\Gamma}\oplus \underline{\mathsf{F}}$ in $\uGproj\Gamma$. 
    
    (i) $\Longrightarrow$ (ii): If $\underline{\mathsf{e}}$ is an equivalence, then by the above follows that $\underline{\mathsf{F}}\simeq 0$. 

    (ii) $\Longrightarrow$ (i): If $\underline{\mathsf{F}}\simeq 0$, then $\underline{\mathsf{e}}\mathsf{l}\simeq \mathsf{Id}_{\uGproj\Gamma}$ and in particular $\mathsf{l}$ is fully faithful. Consequently, the functor $\underline{\mathsf{e}}$ induces an equivalence 
    \[
    \uGproj\Lambda/\mathsf{ker\underline{e}}\simeq \uGproj\Gamma.
    \]
    We claim that $\mathsf{ker\underline{e}}$ is trivial. Indeed, if $\underline{\mathsf{e}}(X)\cong 0$ in $\uGproj\Gamma$, then in view of the isomorphism $\mathsf{e}(X)\cong \underline{\mathsf{e}}(X)$ in $\mathsf{D}_{\mathsf{sg}}(\Gamma)$, it follows that $\mathsf{e}(X)$ has finite projective dimension as a module, see Lemma \ref{perfect_objects}. Consequently, it follows by Proposition \ref{preserves_and_reflects} that $\pd X_{\Lambda}<\infty$, i.e.\ $X\cong 0$ in $\mathsf{D}_{\mathsf{sg}}(\Lambda)$, thus also $X\cong 0$ in $\uGproj\Lambda$. 
\end{proof}

\begin{exmp} 
Consider a triangular matrix ring $\Lambda=\big(\begin{smallmatrix}   A & N\\   0 & B \end{smallmatrix}\big)$ with $A$ and $B$ Noetherian and $N$ finitely generated on both sides. We know from Example \ref{triangular_eventual} that if $\pd{_AN}<\infty$ and $\gd B<\infty$, then the functor $\mathsf{e}$ of the cleft extension $(\Mod\Lambda,\Mod A\!\times\!B,\mathsf{i},\mathsf{e},\mathsf{l})$ - see Example \ref{perfect_for_triangular} - is an eventually homological isomorphism. Hence, it follows by Proposition \ref{sthnk2} that there are triangle functors
\[
\begin{tikzcd}
\uGproj \Lambda \arrow[rr, "\underline{\mathsf{e}}"] &  & \uGproj A\!\times\! B \arrow[ll, "\mathsf{l}"', bend right] \arrow["\underline{\mathsf{F}}"', loop, distance=2em, in=125, out=55]
\end{tikzcd}
\]
By Example \ref{equivalence_for_triangular} we have $\mathbb{L}_{\mathsf{sg}}\mathsf{F}\simeq 0$ in $\mathsf{D}_{\mathsf{sg}}(\Gamma)$ and so $\underline{\mathsf{F}}\simeq 0$ in $\uGproj\Gamma$. One may then apply Proposition \ref{equivalence_of_stable} to obtain an equivalence of stable categories of Gorenstein projectives: $\uGproj\Lambda\simeq \uGproj A\!\times \!B\simeq \uGproj A$.
\end{exmp}

Other instances of equivalences are those arising from Lemma \ref{apo_to_apeiro}. This was proved in \cite[Theorem 3.11]{qin2}. Below we explain, using the techniques developed in this paper, how this is a consequence of Proposition \ref{equivalence_of_stable}. 

Recall that given a $\theta$-extension $\Lambda=\Gamma\ltimes_{\theta}M$, the ring $\Lambda$ is a cleft extension of $\Gamma$. In the following proposition, we denote by $\mathsf{e}$ the functor induced by the ring homomorphism $\Gamma\rightarrow \Lambda$. 

\begin{prop} \label{equivalence_of_stable_for_theta}
    Let $\Gamma$ be a finite dimensional algebra over a field and consider a $\theta$-extension $\Lambda=\Gamma\ltimes_{\theta}M$ where $M$ is a finitely generated (on both sides) $\Gamma$-bimodule. Assume that the following are satisfied: 
    \begin{itemize}
        \item[(i)] $\mathsf{Tor}_i^{\Gamma}(M,M^{\otimes j})=0$ for all $i,j\geq 1$. 
        \item[(ii)] $\pd{_{\Gamma^e}M}<\infty$. 
        \item[(iii)] $M$ is nilpotent.
    \end{itemize}
    Then there is a triangle equivalence $\uGproj\Lambda\simeq \uGproj\Gamma$ induced by the functor $\mathsf{e}$.
\end{prop}
\begin{proof}
   Consider the cleft extension $(\Mod\Lambda,\Mod\Gamma,\mathsf{i},\mathsf{e},\mathsf{l})$.
   We know from Lemma \ref{apo_to_apeiro} that $\mathbb{L}\mathsf{F}(\mathsf{D}^{\mathsf{b}}(\Mod\Lambda))\subseteq \mathsf{K}^{\mathsf{b}}(\Proj\Gamma)$. The functor $\mathbb{L}\mathsf{F}$ on $\mathsf{F}$-projective objects is given by the functor $\mathsf{F}$. It follows that the projective dimensions of $\mathsf{F}(X)$ for $X$ $\mathsf{F}$-projective admit a common bound (if we assume otherwise, then we can find a sequence $X_i$ of $\mathsf{F}$-projective modules with $\pd \mathsf{F}(X_i)>\pd \mathsf{F}(X_{i-1})$ and then $\mathbb{L}\mathsf{F}(\oplus X_i)=\oplus\mathsf{F}(X_i)$ will not belong in $\mathsf{K}^{\mathsf{b}}(\Proj\Gamma)$). We then infer from Theorem \ref{thm2}(ii) that the functor $\mathsf{e}\colon\Mod\Lambda\rightarrow \Mod\Gamma$ is an eventually homological isomorphism. Therefore, the same holds for the restricted cleft extension $(\smod\Lambda,\smod\Gamma,\mathsf{i},\mathsf{e},\mathsf{l})$ (see also \cite[Theorem 2.10]{qin2}). Consequently, Proposition \ref{sthnk2} ensures that the functor $\mathsf{e}$ gives rise to a triangle functor $\underline{\mathsf{e}}\colon\uGproj\Lambda\rightarrow \uGproj\Gamma$. Again by Lemma \ref{apo_to_apeiro}, we know that $\mathbb{L}_{\mathsf{sg}}\mathsf{F}\simeq 0$ in $\mathsf{D}_{\mathsf{sg}}(\Gamma)$ and so $\underline{\mathsf{F}}\simeq 0$ in $\uGproj\Gamma$. Using Proposition \ref{equivalence_of_stable}, we conclude that $\underline{\mathsf{e}}\colon \uGproj\Lambda\rightarrow\uGproj\Gamma$ is an equivalence. 
\end{proof}

\subsection{Gorenstein injective modules in cleft extensions} Denote by $\mathcal{A}$ an abelian category with enough injective objects. An acyclic complex $I^{\bullet}$ of injective objects of $\mathcal{A}$ is called \emph{totally acyclic} if $\mathsf{Hom}_{\mathcal{A}}(E,I^{\bullet})$ is acyclic for every $E\in\Inj\mathcal{A}$. We remind the reader of the following definition from \cite{enochs_jenda}. 

\begin{defn}
    An object $X$ of $\mathcal{A}$ is called \emph{Gorenstein injective} if there is a totally acyclic complex $I^{\bullet}$ over $\Inj\mathcal{A}$ with $X=\mathsf{Z}^0(I^{\bullet})\coloneqq\mathsf{ker}(I^0\rightarrow I^1)$.
\end{defn}

We denote by $\GInj\mathcal{A}$ the category of Gorenstein injective objects of $\mathcal{A}$. As with Gorenstein projectives, one can consider the stable category $\oGInj\mathcal{A}$ and there is a triangle equivalence 
\[
\mathsf{K}_{\mathsf{tac}}(\Inj\mathcal{A})\simeq \oGInj\mathcal{A}.
\]
For a ring $\Lambda$ we use $\GInj\Lambda$ to denote the Gorenstein injective objects of $\Mod \Lambda$. When $\Lambda$ is a Noetherian ring with a dualizing complex $D_{\Lambda}$, then from \cite{iyengar_krause} there is a triangle equivalence 
\[
-\otimes_{\Lambda}D_{\Lambda}\colon\mathsf{K}(\Proj\Lambda)\rightarrow \mathsf{K}(\Inj\Lambda)
\]
which restricts to an equivalence 
\begin{equation} \label{stable of injective vs stable of projective}
    \uGProj\Lambda\simeq \oGInj\Lambda.
\end{equation}
For a Noetherian ring $\Lambda$, the category of finitely generated 
$\Lambda$-modules might not have enough injectives. For this reason, in order to consider $\oGinj\Lambda$, we restrict to the case of Artin algebras in which case the equivalence of (\ref{stable of injective vs stable of projective}) restricts to
\[
\uGproj\Lambda\simeq \oGinj\Lambda.
\]
In view of the above, we will not state any result regarding cleft coextensions and stable categories of Gorenstein injectives. Let us rather focus on some structural results on Gorenstein injectives, in the spirit of Section \ref{stable_cat_of_gorenstein_proj}. The following is dual to Proposition \ref{adjoint_pair_and_gorenstein_projectives}.
\begin{prop}
    Let $\mathcal{A}$ and $\mathcal{B}$ be abelian categories with enough injective objects and assume that there is an adjoint pair of functors 
    \[
    \begin{tikzcd}
\mathcal{A} \arrow[rr, "\mathsf{e}"] &  & \mathcal{B} \arrow[ll, "\mathsf{r}", bend left]
\end{tikzcd}
    \]
    where $\mathsf{e}$ is exact. Assume further that the following are satisfied: 
    \begin{itemize}
        \item[(i)] $\mathbb{R}^i\mathsf{r}=0$ for $i>\!\!>0$. 
        \item[(ii)] $\id{_{\mathcal{B}}\mathsf{e}(I)}\leq d$ for all $I\in\Inj\mathcal{A}$ and some $d$. 
    \end{itemize}
    Then $\mathsf{r}(\GInj\mathcal{B})\subseteq \GInj\mathcal{A}$ and $\mathbb{R}^i\mathsf{r}(X)=0$ for every $X\in\GInj\mathcal{B}$ and $i\geq 1$.
\end{prop}

Now (and for later use) we need the following, which is dual to Proposition \ref{derived_of_q_vanishes}. 

\begin{prop} \label{vanishing of right derived}
    Let $(\mathcal{B},\mathcal{A},\mathsf{i},\mathsf{e},\mathsf{r})$ be a cleft coextension of abelian categories. If $\mathsf{F}'$ is coperfect and nilpotent, then $\mathbb{R}^i\mathsf{p}=0$ for $i$ large enough.  
\end{prop}

By the above, together with Proposition \ref{preserves_and_reflects_inj_dim}, Remark \ref{bounds_for_injective_dimensions} and the dual of Lemma \ref{i_maps_proj} (which is left to the reader to spell out), we conclude the following dual to Corollary \ref{gorenstein_proj_in_B}.

\begin{cor}
    Let $(\mathcal{B},\mathcal{A},\mathsf{i},\mathsf{e},\mathsf{r})$ be a cleft coextension of abelian categories such that $\mathsf{F}'$ is coperfect and nilpotent. Assume that $\id{_{\mathcal{B}}\mathsf{F}'(I)}\leq d$ for every $I\in\Inj\mathcal{B}$ and some $d$. The following hold: 
    \begin{itemize}
        \item[(i)] $\mathsf{r}(\GInj\mathcal{B})\subseteq \GInj\mathcal{A}$. 
        \item[(ii)] $\mathsf{p}(\GInj\mathcal{A})\subseteq \GInj\mathcal{B}$.
    \end{itemize}
    In particular, $X\in\GInj\mathcal{B}$ if and only if $\mathsf{r}(X)\in\GInj\mathcal{A}$. 
\end{cor}

\subsection{CM-free rings and algebras of finite Cohen-Macaulay type} 

In this section we apply our techniques to derive results about CM-free rings and algebras of finite Cohen-Macaulay type, beginning with the former. 

Consider a two-sided Noetherian ring $\Lambda$. Recall from \cite{chen4} that $\Lambda$ is said to be \emph{Cohen-Macaulay free} (CM-free for short) if $\Gproj\Lambda=\proj\Lambda$. 

\begin{prop} \label{CM_free_rings}
    Let $(\Mod\Gamma,\Mod\Lambda,\mathsf{i},\mathsf{e},\mathsf{l})$ be a cleft extension of module categories of two-sided Noetherian rings. Assume that $\mathsf{F}$ is perfect and nilpotent. The following hold: 
    \begin{itemize}
        \item[(i)] If $\Lambda$ is CM-free, then $\Gamma$ is CM-free. 
        \item[(ii)] If the functor $\mathsf{e}\colon\Mod\Lambda\rightarrow \Mod\Gamma$ induces a functor $\underline{\mathsf{e}}\colon\uGproj\Lambda\rightarrow \uGproj\Gamma$ and $\Gamma$ is CM-free, then $\Lambda$ is CM-free. 
    \end{itemize}
\end{prop}
\begin{proof}
     (i) Let $X\in \Gproj\Gamma$. Then, $\mathsf{l}(X)\in\Gproj\Lambda=\proj\Lambda$ and so it follows that $X\cong \mathsf{ql}(X)\in\proj\Gamma$. 

     (ii) Let $X\in\Gproj\Lambda$. Consider the following short exact sequences in $\Mod\Lambda$: 
     \[
     0\rightarrow \mathsf{G}(X)\rightarrow \mathsf{le}(X)\rightarrow X\rightarrow 0, \ 0\rightarrow \mathsf{G}^2(X)\rightarrow \mathsf{lFe}(X)\rightarrow \mathsf{G}(X)\rightarrow 0, \ \dots
     \]
     which in turn give triangles in $\mathsf{D}_{\mathsf{sg}}(\Lambda)$ as below 
     \[
     \mathsf{G}(X)\rightarrow \mathsf{le}(X)\rightarrow X\rightarrow \mathsf{G}(X)[1], \ \mathsf{G}^2(X)\rightarrow \mathsf{lFe}(X)\rightarrow \mathsf{G}(X)\rightarrow \mathsf{G}^2(X)[1], \ \dots
     \]
     We have $\underline{\mathsf{e}}(X)\in\uGproj\Gamma$. Since $\underline{\mathsf{e}}(X)\cong \mathsf{e}(X)$ in $\mathsf{D}_{\mathsf{sg}}(\Gamma)$, it follows by assumption that $\mathsf{e}(X)\cong 0$ in $\mathsf{D}_{\mathsf{sg}}(\Gamma)$. Consequently, all the middle terms of the above triangles are 0. Since the functor $\mathsf{F}$ is nilpotent, it follows that $X\cong 0$ in $\mathsf{D}_{\mathsf{sg}}(\Lambda)$, meaning that $X$ has finite projective dimension. However, since $X\in\Gproj\Lambda$, it follows necessarily that $X\in\proj\Lambda$. 
\end{proof}

We now move to the second part of this subsection. Recall from \cite{beligiannis4, beligiannis5}, that an Artin algebra $\Lambda$ is of \emph{finite Cohen-Macaulay type} (finite CM type for short) if $\Gproj\Lambda$ is of finite representation type, i.e.\ there are finitely many isomorphism classes of indecomposable Gorenstein projective modules in $\smod\Lambda$. In what follows, we show that in the presence of a cleft extension $(\Mod\Gamma,\Mod\Lambda,\mathsf{i},\mathsf{e},\mathsf{l})$, under relaxed conditions on $\mathsf{F}$ (see Proposition \ref{finite_cohen_macaulay_type}), if $\Lambda$ is of finite Cohen-Macaulay type, then so is $\Gamma$. For this we prepare the following two lemmata.  

\begin{lem} \label{kernel}
    Let $(\Mod \Gamma,\Mod\Lambda,\mathsf{i},\mathsf{e},\mathsf{l})$ be a cleft extension of module categories. If $\mathsf{F}$ is nilpotent, then $\mathsf{kerq}=0$. 
\end{lem}
\begin{proof}
    By Remark \ref{canonical_form}, it is enough to show that for a $\theta$-extension $\Gamma\!\ltimes\!_{\theta}M$, the functor $-\otimes_{\Gamma\ltimes_{\theta}M}\Gamma$ has trivial kernel provided that $M$ is nilpotent, which follows by \cite[Lemma 5.3]{beligiannis}. 
\end{proof}

\begin{lem} \label{indecomposable}
    Let $(\smod\Gamma,\smod\Lambda,\mathsf{i},\mathsf{l},\mathsf{q})$ be a cleft extension of finitely generated modules over Artin algebras. Assume that $\mathsf{ker}\mathsf{q}=0$. If $X\in\smod\Gamma$ is indecomposable, then $\mathsf{l}(X)\in\smod\Lambda$ is indecomposable. 
\end{lem}
\begin{proof}
    If $\mathsf{l}(X)\cong X_1\oplus X_2$ for some $X_1,X_2\in\smod\Lambda$, then $X\cong \mathsf{q}(X_1)\oplus \mathsf{q}(X_2)$ and since $X$ is assumed to be indecomposable, it follows that $\mathsf{q}(X_1)\cong 0$ or $\mathsf{q}(X_2)\cong 0$. From the assumption $\mathsf{kerq}=0$ follows that one of $X_1,X_2$ must be $0$. 
\end{proof}

\begin{prop} \label{finite_cohen_macaulay_type}
    Let $(\Mod\Gamma,\Mod\Lambda,\mathsf{i},\mathsf{e},\mathsf{l})$ be a cleft extension of module categories of Artin algebras. Assume that the functor $\mathsf{F}$ is nilpotent and satisfies the following: 
    \begin{itemize}
        \item[(i)] $\mathbb{L}_i\mathsf{F}=0$ for $i>\!\!>0$. 
        \item[(ii)] $\pd{\mathsf{F}(P)_{\Gamma}}<\infty$ for every $P\in\Mod\Gamma$.
    \end{itemize}
    If $\Lambda$ is of finite Cohen-Macaulay type, then $\Gamma$ is of finite Cohen-Macaulay type. 
\end{prop}
\begin{proof}
    If $\{P_j,i\in I\}$ is a family of pairwise non-isomorphic indecomposable Gorenstein projective modules in $\smod\Gamma$, then $\{\mathsf{l}(P_i),i\in I\}$ is a family of pairwise non-isomorphic indecomposable Gorenstein projective modules in $\smod\Lambda$. Indeed, if $\mathsf{l}(P_i)\cong \mathsf{l}(P_j)$, then $P_i\cong \mathsf{ql}(P_i)\cong \mathsf{ql}(P_j)\cong P_j$. The fact that each $\mathsf{l}(P_i)$ is indecomposable follows from Lemmata \ref{kernel}  and \ref{indecomposable}. By assumption it follows that $I$ must be finite. \end{proof}

Although the class of algebras of finite CM-type is rather large (for instance it contains all monomial algebras \cite{chen_shen_zhou}), proving an inverse to the above does not seem feasible. We spend the rest of this section discussing aspects towards this direction.

To start with, a problem is that it is not known, in general, whether two algebras with equivalent stable categories of Gorenstein projective modules share the property of being of finite CM type. This requires the extra assumption that the algebras are virtually Gorenstein, in the sense of \cite{beligiannis4}. Recall that an Artin algebra $\Lambda$ is called \emph{virtually Gorenstein} if 
\[
\uGProj\Lambda^{\mathsf{c}}=\uGproj\Lambda.
\]
Note that the inclusion $\uGproj\Lambda\subseteq \uGProj\Lambda^{\mathsf{c}}$ holds for all Artin algebras. Given two algebras $\Gamma$ and $\Lambda$ that are virtually Gorenstein and $\uGproj\Gamma\simeq \uGproj\Lambda$, it follows by \cite[Theorem 4.10]{beligiannis5} that $\Gamma$ is of finite CM type if and only if $\Lambda$ is of finite CM type. Every Iwanaga-Gorenstein algebra is virtually Gorenstein. Consequently, from Corollary \ref{cor_for_gorenstein_rings} and Proposition \ref{equivalence_of_stable}, we infer the following.  

\begin{cor} \label{gorenstein_algebras_of_finite_CM_type}
    Let $(\Mod\Gamma,\Mod\Lambda,\mathsf{i},\mathsf{e},\mathsf{l})$ be a cleft extension of modules categories of Artin algebras. Assume that $\mathsf{F}$ is perfect and nilpotent, $\mathsf{e}$ induces a functor $\underline{\mathsf{e}}\colon\uGproj\Lambda\rightarrow \uGproj\Gamma$ and $\underline{\mathsf{F}}\simeq 0$ in $\uGproj\Gamma$. If $\Gamma$ is Iwanaga-Gorenstein of finite CM type, then $\Lambda$ is Iwanaga-Gorenstein of finite CM type. 
\end{cor}

\begin{exmp}
    Consider a trivial extension $\Lambda=\Gamma\ltimes M$ where $\Gamma$ and $\Lambda$ are Artin algebras. Assume that $\pd{_{\Gamma^e}M}<\infty$ and $M\otimes_{\Gamma}M=0$. If $\Gamma$ is Iwanaga-Gorenstein of finite CM type, then it follows by Proposition \ref{equivalence_of_stable_for_theta} and Corollary \ref{gorenstein_algebras_of_finite_CM_type} that $\Lambda$ is Iwanaga-Gorenstein of finite CM type. 
\end{exmp}

In view of the above, it is an interesting problem to compare virtual Gorensteinness in cleft extensions. Using arguments that we have repeated before, see for instance the proof of Proposition \ref{CM_free_rings}, it is easy to prove the following. 

\begin{cor}
    Consider a cleft extension $(\Mod\Gamma,\Mod\Lambda,\mathsf{i},\mathsf{e},\mathsf{l})$ of Artin algebras. Assume that the following are satisfied:
    \begin{itemize}
        \item[(i)] $\mathsf{F}$ is perfect and nilpotent.
        \item[(ii)] $\mathsf{e}$ induces a compact-preserving triangle functor $\underline{\mathsf{e}}\colon\uGProj\Lambda\rightarrow \uGProj\Gamma$.
    \end{itemize}
    Then $\Lambda$ is virtually Gorenstein if and only if $\Gamma$ is virtually Gorenstein. 
\end{cor}

\section{Big singularity categories}
In this section we work with the \emph{big singularity category} due to Krause \cite{krause}. Our main objective is to obtain the lower part of the diagram of Theorem A of the Introduction. This is essentially a combination of the results we have presented on perfect endofunctors and the corresponding results of \cite{OPS}, adjusted in our context. Lastly, under extra assumptions, we obtain an equivalence of big singularity categories arising from cleft extensions of rings, see Theorem \ref{fox}. 

\subsection{Homotopy category of injectives and big singularity categories} Let $\mathcal{T}$ be a triangulated category with set-indexed coproducts. Given a class $\mathcal{X}$ of objects of $\mathcal{T}$, its right orthogonal is defined as 
\[
\mathcal{X}^{\perp}=\{T\in\mathcal{T}:\mathsf{Hom}_{\mathcal{T}}(\mathcal{X},T[n])=0 \ \forall n\in\mathbb{Z}\}. 
\]
Recall that an object $X\in\mathcal{T}$ is called \emph{compact} if $\mathsf{Hom}_{\mathcal{T}}(X,-)$ commutes with coproducts. We denote by $\mathcal{T}^{\mathsf{c}}$ the subcategory of $\mathcal{T}$ that consists of the compact objects. Further, we say that $\mathcal{T}$ is \emph{compactly generated} if there is a set of compact objects $\mathcal{G}$ such that $\mathcal{G}^{\perp}=0$. In this case, the localizing subcategory generated by $\mathcal{G}$ is $\mathcal{T}$. For instance, it is well-known that $\mathsf{D}(R)$, the unbounded derived category of a ring, is compactly generated and $\mathsf{D}(R)^{\mathsf{c}}=\mathsf{K}^{\mathsf{b}}(\proj R)$.

\begin{krauserecollement}
Consider a Noetherian ring $\Lambda$. Over Noetherian rings, the coproduct of injective modules is injective; this makes $\mathsf{K}(\Inj\Lambda)$, the homotopy category of injective $\Lambda$-modules, a triangulated category with coproducts. 

The \emph{big singularity category} of $\Lambda$ is defined to be $\mathsf{K}_{\mathsf{ac}}(\Inj\Lambda)$; the triangulated subcategory of $\mathsf{K}(\Inj\Lambda)$ that consists of the acyclic complexes. A fundamental result is the existence of \emph{Krause's recollement} \cite{krause}
\[
\begin{tikzcd}
\mathsf{K}_{\mathsf{ac}}(\Inj \Lambda) \arrow[rr, "\mathsf{I}"] &  & \mathsf{K}(\Inj\Lambda) \arrow[rr, "\mathsf{Q}"] \arrow[ll, "\mathsf{L}_{\Lambda}"', bend right] \arrow[ll, "\mathsf{R}_{\Lambda}", bend left] &  & \mathsf{D}(\Lambda) \arrow[ll, "\mathsf{Q}_{\lambda}"', bend right] \arrow[ll, "\mathsf{Q}_{\rho}", bend left]
\end{tikzcd}
\]
where the functors $\mathsf{I}$ and $\mathsf{Q}$ are given by 
\[
\mathsf{K}_{\mathsf{ac}}(\Inj\Lambda)\hookrightarrow\mathsf{K}(\Inj\Lambda) \text{  \ \   and   \ \  } \mathsf{K}(\Inj\Lambda)\hookrightarrow\mathsf{K}(\Mod\Lambda)\xrightarrow{\mathsf{can}}\mathsf{D}(R),
\]
respectively. We outline some key points. 

First of all, $\mathsf{K}(\Inj\Lambda)$ is compactly generated. Indeed, for every $X\in\mathsf{K}(\Inj\Lambda)$ and every $Y\in\smod\Lambda$, by a straightforward computation, we have
\begin{align}  
    \mathsf{Hom}_{\mathsf{K}(\Inj\Lambda)}(\lambda( Y),X[n])\cong\mathsf{Hom}_{\mathsf{K}(\Mod\Lambda)}(Y,X[n])
\end{align} 
for all $n$, where $\lambda(Y)$ denotes an injective resolution of $Y$ (see \cite[Lemma 2.1]{krause}). Further, by \cite[Lemma 2.2]{krause}, if $X\in\mathsf{K}(\Mod\Lambda)$ is nonzero, then there is $Y\in\smod\Lambda$ and $n$ such that $\mathsf{Hom}_{\mathsf{K}(\Mod \Lambda)}(Y,X[n])\neq 0$. By combining the above we conclude that $\{\lambda(Y), \ Y\in\smod\Lambda\}$ is a set of compact generators of $\mathsf{K}(\Inj\Lambda)$ and in particular $\mathsf{K}(\Inj\Lambda)$ is compactly generated. The thick subcategory generated by the latter is identified with $\mathsf{D}^{\mathsf{b}}(\smod\Lambda)$, i.e.\ $\mathsf{K}(\Inj\Lambda)^{\mathsf{c}}=\mathsf{D}^{\mathsf{b}}(\smod\Lambda)$.

By (8.1), it follows that for $X\in\mathsf{K}(\Inj\Lambda)$ we have 
\[
\mathsf{Hom}_{\mathsf{K}(\Inj\Lambda)}(\lambda(\Lambda),X[n])\cong\mathsf{Hom}_{\mathsf{K}(\Mod\Lambda)}(\Lambda,X[n])\cong\mathsf{H}^n(X),
\]
so $\mathsf{K}_{\mathsf{ac}}(\Inj\Lambda)=(\lambda(\Lambda))^{\perp}$. In particular, it follows that $\mathsf{K}_{\mathsf{ac}}(\Inj\Lambda)$ is closed under coproducts. By \cite[Proposition 3.6, Theorem 4.2]{krause}, the sequence 
\[
\mathsf{K}_{\mathsf{ac}}(\Inj\Lambda)\xrightarrow{\mathsf{I}}\mathsf{K}(\Inj\Lambda)\xrightarrow{\mathsf{Q}}\mathsf{D}(\Lambda)
\]
is a localization and a colocalization sequence. Thus by \cite[Theorem 3.1]{neeman5}, \emph{up to direct summands}, we obtain an equivalence 
\[
\mathsf{K}_{\mathsf{ac}}(\Inj\Lambda)^{\mathsf{c}}\simeq\mathsf{K}(\Inj\Lambda)^{\mathsf{c}}/\mathsf{D}(\Lambda)^{\mathsf{c}}\simeq\mathsf{D}_{\mathsf{sg}}(\Lambda).
\]
What this means precisely, is that there is a triangle functor $\mathsf{F}\colon\mathsf{D}_{\mathsf{sg}}(\Lambda)\rightarrow \mathsf{K}_{\mathsf{ac}}(\Inj\Lambda)^{\mathsf{c}}$ which fully faithful and every object in $\mathsf{K}_{\mathsf{ac}}(\Inj\Lambda)^{\mathsf{c}}$ is a summand of an object in $\mathsf{ImF}$. This justifies the name \emph{big singularity category} for $\mathsf{K}_{\mathsf{ac}}(\Inj\Lambda)$. 
\end{krauserecollement}
\begin{adjointpairs}
    We fix some notation that follows. First, we remind the reader of the existence of the following adjoint pair: 
\[
\begin{tikzcd}
\mathsf{K}(\Inj\Lambda) \arrow[r, hook] & \mathsf{K}(\Mod \Lambda) \arrow[l, "\lambda"', bend right].
\end{tikzcd}
\]
This is a consequence of a theorem of Neeman, see \cite[Theorem 8.6.1]{neeman2}, since the inclusion $\mathsf{K}(\Inj\Lambda)\hookrightarrow\mathsf{K}(\Mod\Lambda)$ is product preserving and $\mathsf{K}(\Inj\Lambda)$ is compactly generated (see also \cite{neeman4}). For a bounded complex $X$, in view of (8.1), the object $\lambda(X)$ is given by an injective resolution of $X$.

Recall also that the abelian category $\mathsf{C}(\Mod \Lambda)$ of unbounded complexes of $\Lambda$-modules has enough injectives (for projectives see Section 7). In particular, for every $X \in\mathsf{C}(\Mod \Lambda)$, there is a contractible complex $I_{X}\in\mathsf{C}(\Inj\Lambda)$, with a monomorphism $X\hookrightarrow I_X$. We denote $\mathsf{coker}(X\rightarrow I_{X})$ by $\mho(X)$. We use $\tilde{\mho}$ to denote $\mho\circ[-1]$. 

\begin{prop} \label{skulos} \textnormal{(\!\!\cite[Proposition 2.5]{OPS})}
Let $X$ be a complex of (right) $\Lambda$-modules. If $\id{{X^i}_{\Lambda}}\leq d$ for some $d$, then $\lambda(X)=\tilde{\mho}^d(X)$.
\end{prop}

\begin{prop} \label{apthn_kolash} \textnormal{(\!\!\cite[Proposition 2.6]{OPS})}
    Let $X$ and $Y$ be complexes of $\Lambda$-modules and assume that $\mathsf{Ext}_{\Lambda}^{d+1}(X^m,Y^n)=0$ for all $m,n\in\mathbb{Z}$. There is an isomorphism 
    \[
    \mathsf{Hom}_{\mathsf{K}(\Mod \Lambda)}(\Omega^dX,Y)\cong \mathsf{Hom}_{\mathsf{K}(\Mod \Lambda)}(X,\mho^dY),
    \]
    which is functorial in $X$ and $Y$. 
\end{prop}
\end{adjointpairs}

\subsection{Adjoint functors and big singularity categories}
We now explain how to obtain adjoint pairs of triangle functors between homotopy categories of acyclic complexes of injectives that arise from adjoint pairs of functors of module categories. These results have been established in \cite{OPS}, but since we present slight variations of them, we provide proofs (following essentially the same arguments). Consider an adjoint triple of module categories 
\begin{equation}
     \begin{tikzcd}
\Mod \Lambda \arrow[rr, "\mathsf{e}"] &  & \Mod \Gamma \arrow[ll, "\mathsf{l}"', bend right] \arrow[ll, "\mathsf{r}", bend left]
\end{tikzcd}
\end{equation}

The following is essentially \cite[Proposition 4.1]{OPS}.

\begin{prop} \label{adjoint_triple_and_KINJ} Assume the setting of \textnormal{(8.2)}. If $\id\ \!{\mathsf{e}(I)_{\Gamma}}\leq d$ for some $d$ and for all $I\in\Inj\Lambda$, then there is an adjoint triple of triangle functors as below
    \[
    \begin{tikzcd}
\mathsf{K}(\Inj\Lambda) \arrow[rr, "\lambda_{\Gamma}\mathsf{e}"] &  & \mathsf{K}(\Inj \Gamma) \arrow[ll, "\lambda_{\Lambda}\mathsf{l}(\tilde{\Omega}^d(-))"', bend right] \arrow[ll, "\mathsf{r}", bend left]
\end{tikzcd}
    \]
\end{prop}
\begin{proof}
The situation is summarised in the following diagram:
\[
\begin{tikzcd}
\mathsf{K}(\Mod \Lambda) \arrow[rr, "\mathsf{e}"] \arrow[ddd, "\lambda_{\Lambda}"', bend right] &  & \mathsf{K}(\Mod \Gamma) \arrow[ll, "\mathsf{r}", bend left] \arrow[ll, "\mathsf{l}"', bend right] \arrow[ddd, "\lambda_{\Gamma}"', bend right]         \\
                                                                                              &  &                                                                                                                                                          \\
                                                                                              &  &                                                                                                                                                          \\
\mathsf{K}(\Inj\Lambda) \arrow[rr, "\lambda_{\Gamma}\mathsf{e}"] \arrow[uuu, hook]            &  & \mathsf{K}(\Inj\Gamma) \arrow[ll, "\lambda_{\Lambda}\mathsf{l}(\tilde{\Omega}^d(-))"', bend right] \arrow[ll, "\mathsf{r}", bend left] \arrow[uuu, hook]
\end{tikzcd}
\]
    On the upper part $(\mathsf{l},\mathsf{e},\mathsf{r})$ is an adjoint triple, thus by construction, $(\lambda_{\Gamma}\mathsf{e},\mathsf{r})$ is an adjoint pair. Moreover, for any $X\in\mathsf{K}(\Inj\Gamma)$ and $Y\in\mathsf{K}(\Inj\Lambda)$, we have the following isomorphisms:
    \begin{align*}
         \mathsf{Hom}_{\mathsf{K}(\Inj \Lambda)}(\lambda_{\Lambda}\mathsf{l}(\tilde{\Omega}^d(X)),Y)&=\mathsf{Hom}_{\mathsf{K}(\Mod \Lambda)}(\mathsf{l}(\tilde{\Omega}^d(X)),Y) \\ 
                                 &\cong \mathsf{Hom}_{\mathsf{K}(\Mod \Gamma)}(\tilde{\Omega}^d(X),\mathsf{e}(Y)) \\
                                 & \cong \mathsf{Hom}_{\mathsf{K}(\Mod \Gamma)}(X,\tilde{\mho}^d\mathsf{e}(Y)) & \text{ Proposition \ref{apthn_kolash}} \\ 
                                 & \cong \mathsf{Hom}_{\mathsf{K}(\Mod\Gamma)}(X,\lambda_{\Gamma}\mathsf{e}(Y)) & \text{ Proposition \ref{skulos}} \\ 
                                 & = \mathsf{Hom}_{\mathsf{K}(\Inj\Gamma)}(X,\lambda_{\Gamma}\mathsf{e}(Y)).
    \end{align*}
This shows that $\lambda_{\Lambda}\mathsf{l}(\tilde{\Omega}^d(-))$ is left adjoint to $\lambda_{\Gamma}\mathsf{e}$, which completes the proof. 
\end{proof}

The following was proved in \cite[Proposition 4.4(i)]{OPS}.

\begin{prop} \label{adjoints_of_Kac}
    Assume the setting of \textnormal{(8.2)}. The following hold: 
    \begin{itemize}
        \item[(i)] If $\mathsf{R}^n\mathsf{r}=0$ for $n>\!\!>0$, then the functor $\mathsf{r}\colon\mathsf{K}(\Inj\Gamma)\rightarrow \mathsf{K}(\Inj\Lambda)$ restricts to a functor $\mathsf{K}_{\mathsf{ac}}(\Inj\Gamma)\rightarrow \mathsf{K}_{\mathsf{ac}}(\Inj \Lambda)$.
        \item[(ii)] If $\id\ \!{\mathsf{e}(I)_{\Gamma}}\leq d$ for some $d$, then the functor $\lambda_{\Gamma}\mathsf{e}\colon\mathsf{K}(\Inj\Lambda)\rightarrow \mathsf{K}(\Inj\Gamma)$ restricts to a functor $\mathsf{K}_{\mathsf{ac}}(\Inj\Lambda)\rightarrow \mathsf{K}_{\mathsf{ac}}(\Inj\Gamma)$. 
    \end{itemize}
\end{prop}
\begin{proof}
(i) For every $I\in\mathsf{K}_{\mathsf{ac}}(\Inj\Gamma)$ we have $\mathsf{H}^i(\mathsf{r}(I))\cong \mathbb{R}^{i+j}\mathsf{r}(\mathsf{B}^{-j}(I))$ for $j>\!\!>0$, from which the result follows.

(ii) By the assumption and Proposition \ref{skulos}, we infer that $\lambda_{\Gamma}\mathsf{e}\simeq \tilde{\mho}^d\mathsf{e}$ and since both $\mathsf{e}$ and $\tilde{\mho}^d$ preserve acyclicity, the result follows. 
\end{proof}

We may now employ the left adjoint of the inclusion $\mathsf{K}_{\mathsf{ac}}(\Inj\Lambda)\hookrightarrow \mathsf{K}(\Inj \Lambda)$, in order to obtain an extra (left) adjoint of $\lambda_{\Gamma}\mathsf{e}\colon\mathsf{K}_{\mathsf{ac}}(\Inj\Lambda)\rightarrow \mathsf{K}_{\mathsf{ac}}(\Inj\Gamma)$. The situation, under the assumptions of Proposition \ref{adjoints_of_Kac}, is summarised in the following diagram:

\[
\begin{tikzcd}
\mathsf{K}(\Inj\Lambda) \arrow[rr, "\lambda_{\Gamma}\mathsf{e}"] \arrow[ddd, "\mathsf{L}_{\Lambda}"', bend right] &  & \mathsf{K}(\Inj\Gamma) \arrow[ll, "\mathsf{r}", bend left] \arrow[ll, "\lambda_{\Lambda}\mathsf{l}(\tilde{\Omega}^d(-))"', bend right] \arrow[ddd, "\mathsf{L}_{\Gamma}"', bend right]   \\
                                                                                                                 &  &                                                                                                                                                                                           \\
                                                                                                                 &  &                                                                                                                                                                                           \\
\mathsf{K}_{\mathsf{ac}}(\Inj\Lambda) \arrow[uuu, hook] \arrow[rr, "\lambda_{\Gamma}\mathsf{e}"]                 &  & \mathsf{K}_{\mathsf{ac}}(\Inj\Gamma) \arrow[uuu, hook] \arrow[ll, "\mathsf{r}", bend left] \arrow[ll, "\mathsf{L}_{\Lambda}\lambda_{\Lambda}\mathsf{l}(\tilde{\Omega}^d(-))"', bend right]
\end{tikzcd}
\]

It is evident that on the bottom part of the above diagram we have an adjoint triple. By all the above we have the following.  

\begin{prop} \label{cleft_of_big_singularity}
    Let $(\Mod\Gamma,\Mod\Lambda,\mathsf{i},\mathsf{e},\mathsf{l})$ be a cleft extension of module categories of Noetherian rings and assume that the following are satisfied: 
    \begin{itemize}
        \item[(i)] $\mathbb{R}^{n}\mathsf{r}=0$ for $n>\!\!>0$ and $\id\ \!{\mathsf{e}(I)_{\Gamma}}\leq d$ for every $I\in\Inj\Lambda$ and some $d$. 
        \item[(ii)] $\mathbb{R}^n\mathsf{p}=0$ for $n>\!\!>0$ and $\id\ \!{\mathsf{i}(I)_{\Lambda}}\leq d'$ for every $I\in\Inj\Gamma$ and some $d'$. 
    \end{itemize}
    Then, there is a diagram of big singularity categories as below 
    \[
 \begin{tikzcd}
\mathsf{K}_{\mathsf{ac}}(\Inj \Gamma) \arrow[rr, "\lambda_{\Lambda}\mathsf{i}"] &  & \mathsf{K}_{\mathsf{ac}}(\Inj \Lambda) \arrow[rr, "\lambda_{\Gamma}\mathsf{e}"] \arrow[ll, "\mathsf{L}_{\Gamma}\lambda_{\Gamma}\mathsf{q}(\tilde{\Omega}^{d'}(-))"', bend right] \arrow[ll, "\mathsf{p}", bend left] &  & \mathsf{K}_{\mathsf{ac}}(\Inj \Gamma) \arrow[ll, "\mathsf{L}_{\Lambda}\lambda_{\Lambda}\mathsf{l}(\tilde{\Omega}^d(-))"', bend right] \arrow[ll, "\mathsf{r}", bend left]
\end{tikzcd}
    \]
    where $(\mathsf{L}_{\Gamma}\lambda_{\Gamma}\mathsf{q}(\tilde{\Omega}^{d'}(-)),\lambda_{\Lambda}\mathsf{i},\mathsf{p})$ and $(\mathsf{L}_{\Lambda}\lambda_{\Lambda}\mathsf{l}(\tilde{\Omega}^d(-)),\lambda_{\Gamma}\mathsf{e},\mathsf{r})$ are adjoint triples. Moreover, there are equivalences of functors $\mathsf{pr}\simeq \mathsf{Id}_{\mathsf{K}_{\mathsf{ac}}(\Inj\Gamma)}$ and $\lambda_{\Gamma}\mathsf{e}\lambda_{\Lambda}\mathsf{i}\simeq \mathsf{Id}_{\mathsf{K}_{\mathsf{ac}}(\Inj\Gamma)}$.  
\end{prop}
\begin{proof}
The existence of the adjoint triples follows from Proposition \ref{adjoints_of_Kac} and the discussion below it. Recall that $\mathsf{pr}\colon\mathsf{K}_{\mathsf{ac}}(\Inj\Gamma)\rightarrow \mathsf{K}_{\mathsf{ac}}(\Inj\Gamma)$ occurs as a restriction of $\mathsf{pr}\colon\mathsf{K}(\Mod\Gamma)\rightarrow \mathsf{K}(\Mod\Gamma)$, which is isomorphic to $\mathsf{Id}_{\mathsf{K}(\Mod\Gamma)}$. This shows that $\mathsf{pr}\simeq \mathsf{Id}_{\mathsf{K}_{\mathsf{ac}}(\Inj\Gamma)}$. Further, since $\lambda_{\Lambda}\mathsf{i}$ is left adjoint to $\mathsf{p}$ and $\lambda_{\Gamma}\mathsf{e}$ is left adjoint to $\mathsf{r}$, it follows that $\lambda_{\Gamma}\mathsf{e}\lambda_{\Lambda}\mathsf{i}$ is left adjoint to $\mathsf{pr}$, so necessarily $\lambda_{\Gamma}\mathsf{e}\lambda_{\Lambda}\mathsf{i}\simeq \mathsf{Id}_{\mathsf{K}_{\mathsf{ac}}(\Inj\Gamma)}$.
\end{proof}

By combining the above together with Proposition \ref{cleft_of_modules_is_cocleft}, Proposition \ref{preserves_and_reflects_inj_dim}, Remark \ref{bounds_for_injective_dimensions} and Proposition \ref{vanishing of right derived}, we conclude the following, which proves the existence of the lower part of the commutative diagram of Theorem A. 

\begin{cor} \label{cleft_of_big_singularity_for_perfect}
     Let $(\Mod \Gamma,\Mod \Lambda,\mathsf{i},\mathsf{e},\mathsf{l})$ be a cleft extension of module categories such that $\mathsf{F}$ is perfect and nilpotent. Then the assumptions of Proposition \ref{cleft_of_big_singularity} are satisfied. 
\end{cor}

The commutativity part is discussed below. 

\begin{rem} \label{duck}
Consider an exact functor $\mathsf{e}\colon\Mod \Lambda\rightarrow \Mod \Gamma$ of module categories of Noetherian rings and assume that it admits a right adjoint $\mathsf{r}\colon\Mod\Gamma\rightarrow \Mod\Lambda$. Moreover, assume that $\mathsf{e}$ and $\mathsf{r}$ give rise to triangle functors as follows:
\begin{equation}
    \begin{tikzcd}
\mathsf{K}_{\mathsf{ac}}(\Inj\Lambda) \arrow[r, "\lambda_{\Gamma}\mathsf{e}"] & \mathsf{K}_{\mathsf{ac}}(\Inj\Gamma) \arrow[l, "\mathsf{r}", bend left] \\
\mathsf{D}_{\mathsf{sg}}(\Lambda) \arrow[r, "\mathsf{e}"] \arrow[u, hook]     & \mathsf{D}_{\mathsf{sg}}(\Gamma) \arrow[u, hook]                       
\end{tikzcd}
\end{equation}
Then, the above diagram is commutative. Indeed, first recall that the inclusion $\mathsf{D}_{\mathsf{sg}}(\Lambda)\hookrightarrow \mathsf{K}_{\mathsf{ac}}(\Inj\Lambda)$ is obtained by restricting the upper functors of Krause's recollement. By \cite[Corollary 5.4]{krause}, the above admits an explicit description, as the following composition: 
\[
\begin{tikzcd}
\mathsf{D}_{\mathsf{sg}}(\Gamma) \arrow[rr, "\text{pick preimage}"] &  & \mathsf{D}^{\mathsf{b}}(\smod\Gamma) \arrow[rr, "\mathsf{L}_{\Gamma}\mathsf{Q}_{\rho}"] &  & \mathsf{K}_{\mathsf{ac}}(\Inj\Gamma)^{\mathsf{c}},
\end{tikzcd}
\]
and similarly for $\mathsf{D}_{\mathsf{sg}}(\Lambda)\hookrightarrow\mathsf{K}_{\mathsf{ac}}(\Inj\Lambda)$. Consider the following functors: 
\[
\begin{tikzcd}
\mathsf{K}(\Inj\Gamma) \arrow[r, "I"] & \mathsf{K}(\Mod\Gamma) \arrow[r, "\mathsf{Q}'=\mathsf{can}"] \arrow[l, "\lambda"', bend right] & \mathsf{D}(\Gamma) \arrow[ll, "\mathsf{Q}_{\rho}", bend left]
\end{tikzcd}
\]

By \cite[Remark 3.7]{krause} we have $\mathsf{Q}'\simeq \mathsf{Q}'I\lambda$. Further, by \cite[Remark 3.8]{krause}, $\mathsf{Q}_{\rho}$ induces an equivalence $\mathsf{D}^{\mathsf{b}}(\smod\Gamma)\rightarrow \mathsf{K}(\Inj\Gamma)^{\mathsf{c}}$ with inverse given by $\mathsf{Q}'I=\mathsf{Q}$. Consequently, for $X$ a bounded complex of finitely generated modules, we have 
\[
\lambda (X)\cong \mathsf{Q}_{\rho}\mathsf{Q}'I\lambda(X)\cong \mathsf{Q}_{\rho}\mathsf{Q}'(X),
\]
meaning that the inclusion $\mathsf{D}_{\mathsf{sg}}(\Gamma)\hookrightarrow \mathsf{K}_{\mathsf{ac}}(\Inj\Gamma)$ may be interpreted as the following composition:
\[
\begin{tikzcd}
\mathsf{D}_{\mathsf{sg}}(\Gamma) \arrow[rr, "\text{pick preimage}"] &  & \mathsf{D}^{\mathsf{b}}(\smod\Gamma) \arrow[rr, "\text{pick preimage}"] &  & \mathsf{K}(\Mod\Gamma) \arrow[rr, "\mathsf{L}_{\Gamma}\lambda_{\Gamma}"] &  & \mathsf{K}_{\mathsf{ac}}(\Inj\Gamma)^{\mathsf{c}}.
\end{tikzcd}
\]
Thus, in order to prove the commutativity of (8.3), it is enough to show that the following diagram is commutative: 
\[
\begin{tikzcd}
\mathsf{K}_{\mathsf{ac}}(\Inj\Gamma) \arrow[d, hook] \arrow[rr, "\lambda_{\Lambda}\mathsf{e}"] &  & \mathsf{K}_{\mathsf{ac}}(\Inj\Lambda) \arrow[d, hook] \arrow[ll, "\mathsf{r}", bend left]               \\
\mathsf{K}(\Inj\Gamma) \arrow[d, hook] \arrow[u, "\mathsf{L}_{\Gamma}"', bend right]           &  & \mathsf{K}(\Inj\Lambda) \arrow[d, hook] \arrow[u, "\mathsf{L}_{\Lambda}"', bend right]                  \\
\mathsf{K}(\Mod\Gamma) \arrow[rr, "\mathsf{e}"] \arrow[u, "\lambda_{\Gamma}"', bend right]     &  & \mathsf{K}(\Mod\Lambda) \arrow[ll, "\mathsf{r}", bend left] \arrow[u, "\lambda_{\Lambda}"', bend right]
\end{tikzcd}
\]
It is evident that composing the right adjoints gives a commutative diagram and consequently the same holds for the left adjoints. 
\end{rem}

\begin{rem}  \label{artin_algebras}
    Over an Artin algebra $\Lambda$, we have $\mathsf{K}_{\mathsf{ac}}(\Inj\Lambda)^{\mathsf{c}}\simeq \mathsf{D}_{\mathsf{sg}}(\Lambda)$ (not only up to direct summands). Indeed, by \cite[Corollary 2.4]{chen2} we know that $\mathsf{D}_{\mathsf{sg}}(\Lambda)$ is idempotent complete, so by \cite[Proposition 1.3]{rickard} it must be a thick subcategory of $\mathsf{K}_{\mathsf{ac}}(\Inj\Lambda)^{\mathsf{c}}$, thus coincide with $\mathsf{K}_{\mathsf{ac}}(\Inj\Lambda)^{\mathsf{c}}$ - since every object of the latter occurs as a direct summand of an object of the former. Hence, in the setting of Remark \ref{duck}, it follows that the functor $\lambda_{\Gamma}\mathsf{e}$ restricts to the compact objects. Consequently, by a result of Neeman, see \cite[Theorem 5.1]{neeman}, it follows that the functor $\mathsf{r}$ admits a right adjoint. This applies in particular to the setting of Proposition \ref{cleft_of_big_singularity} and shows that if we assume $\Gamma$ and $\Lambda$ to be Artin algebras, then the functors $\mathsf{p}$ and $\mathsf{r}$ admit right adjoints. Such situations (also involving homotopy categories of projectives) are summarized beautifully in \cite[Section 4]{OPS}. 
\end{rem}

\subsection{An equivalence of big singularity categories}

In this section we investigate equivalences of big singularity categories in a cleft extension diagram. Since injectives are involved, we work with the cleft coextension part of a cleft extension of module categories (whose existence is automatic - see Proposition \ref{cleft_of_modules_is_cocleft}). Consider a cleft extension $(\Mod\Gamma,\Mod \Lambda,\mathsf{i},\mathsf{e},\mathsf{l})$ where $\Gamma$ and $\Lambda$ are Noetherian. Under the assumptions that $\mathbb{R}^n\mathsf{r}=0$ for $n$ large enough and $\id\ \!{\mathsf{e}(I)_{\Gamma}}\leq d$ for all $I\in\Inj\Gamma$ and some $d$, follows the existence of the following diagram of triangle functors: 
\[
\begin{tikzcd}
\mathsf{K}(\Mod\Lambda) \arrow[rr, "\mathsf{e}"]                                                &  & \mathsf{K}(\Mod\Gamma) \arrow[ll, "\mathsf{r}", bend left] \arrow[dd, "\lambda_{\Gamma}"', bend right] \arrow["\mathsf{F}'"', loop, distance=2em, in=125, out=55]     \\
                                                                                                &  &                                                                                                                                                                       \\
\mathsf{K}(\Inj\Lambda) \arrow[uu, hook] \arrow[rr, "\lambda_{\Gamma}\mathsf{e}"]               &  & \mathsf{K}(\Inj\Gamma) \arrow[uu, hook] \arrow[ll, "\mathsf{r}", bend left] \arrow["\lambda_{\Gamma}\mathsf{F}'"', loop, distance=2em, in=35, out=325]                \\
                                                                                                &  &                                                                                                                                                                       \\
\mathsf{K}_{\mathsf{ac}}(\Inj\Lambda) \arrow[uu, hook] \arrow[rr, "\lambda_{\Gamma}\mathsf{e}"] &  & \mathsf{K}_{\mathsf{ac}}(\Inj\Gamma) \arrow[uu, hook] \arrow[ll, "\mathsf{r}", bend left] \arrow["\lambda_{\Gamma}\mathsf{F}'"', loop, distance=2em, in=305, out=235]
\end{tikzcd}
\]
where $\mathsf{F}'$ denotes the endofunctor associated to the cleft coextension, i.e.\ the right adjoint of $\mathsf{F}$ (see Fact \ref{fact}). For the above to hold, it is enough that $\mathbb{R}^n\mathsf{F}'=0$ for $n$ large enough and $\id{\mathsf{F}'(I)_{\Gamma}}\leq d$ for all $I\in\Inj\Gamma$, respectively. The existence of the functor $\lambda_{\Gamma}\mathsf{F}'\colon\mathsf{K}_{\mathsf{ac}}(\Inj\Gamma)\rightarrow \mathsf{K}_{\mathsf{ac}}(\Inj\Gamma)$, under the given assumptions, amounts to an easy check left to the reader (one can also prove this indirectly, making use of the equivalence $\lambda_{\Gamma}\mathsf{e}\mathsf{r}\simeq \mathsf{Id}_{\mathsf{K}(\Inj\Gamma)}\oplus \lambda_{\Gamma}\mathsf{F}'$ in $\mathsf{K}(\Inj\Gamma)$). 

The following is dual to Lemma \ref{cone}. 

\begin{lem} \label{cone2}
    Let $X$ be an acyclic complex of $\Lambda$-modules. If $\id{{X^n}_{\Gamma}}\leq d$ for all $n$ and some $d$, then, for all $i$, there is a morphism $\phi\colon\mathsf{B}^i(X)\rightarrow \mathsf{B}^i(\tilde{\mho}^d(X))$
    such that $\mathsf{cone}(\phi)$ is quasi-isomorphic to a complex of injectives in degrees 0 to $d$. 
\end{lem}

\begin{cor} \label{injective_dimension_of_boundaries}
    Let $X$ be an acyclic complex of $\Lambda$-modules with $\id{{X^n}_{\Lambda}}\leq d$ for all $n$ and some $d$. If $\lambda(X)$ is contractible, then $\id{\mathsf{B}^i(X)_{\Lambda}}\leq d+1$ for all $i$. 
\end{cor}
\begin{proof}
    By Lemma \ref{skulos} and Lemma \ref{cone2}, for every $i$, follows the existence of a triangle  
    \[
    \mathsf{B}^i(X)\rightarrow \mathsf{B}^i(\lambda(X))\rightarrow I\rightarrow \mathsf{B}^i(X)[1]
    \]
    in $\mathsf{D}(\Lambda)$, where $I$ is a complex (depending on $i$) of injectives concentrated in degrees $0$ to $d$. The module $\mathsf{B}^i(\lambda(X))$ is injective, since $\lambda(X)$ is a contractible complex of injectives. Consequently, for any $Y\in\Mod\Lambda$, applying the functor $\mathsf{Hom}_{\mathsf{D}(\Lambda)}(Y,-)$ to the above triangle shows that $\mathsf{Hom}_{\mathsf{D}(\Lambda)}(Y,\mathsf{B}^i(X)[n])=0$ for all $n\geq d+2$.
\end{proof}

\begin{thm} \label{fox}
    Let $(\Mod\Gamma,\Mod\Lambda,\mathsf{i},\mathsf{e},\mathsf{l})$ be a cleft extension of module categories of Noetherian rings. Assume that $\mathsf{F}'$ satisfies the following: 
    \begin{itemize}
        \item[(a)] $\id{\mathsf{F}'(I)_{\Gamma}}\leq d$ for every $I\in\Inj\Gamma$ and some $d$. 
        \item[(b)] $\mathbb{R}^n\mathsf{F}'=0$ for $n$ large enough.  
        \item[(c)] $\id{X_{\Lambda}}\leq \id\ \!{\mathsf{e}(X)_{\Gamma}}+d'$ for all $X\in\Mod \Lambda$ and some $d'$.
    \end{itemize}
    Then the following are equivalent:
    \begin{itemize}
        \item[(i)] The functor $\lambda_{\Gamma}\mathsf{e}\colon\mathsf{K}_{\mathsf{ac}}(\Inj\Lambda)\rightarrow \mathsf{K}_{\mathsf{ac}}(\Inj\Gamma)$ is an equivalence. 
        \item[(ii)] $\lambda_{\Gamma}\mathsf{F}'\simeq0$ in $\mathsf{K}_{\mathsf{ac}}(\Inj\Gamma)$. 
    \end{itemize}
\end{thm}
\begin{proof}
    Assumptions (a) and (b) ensure that the functors $\mathsf{r}$, $\lambda_{\Gamma}\mathsf{e}$ and $\lambda_{\Gamma}\mathsf{F}'$ exists on the level of big singularity categories. By the equivalence $\mathsf{er}\simeq \mathsf{Id}_{\Mod\Gamma}\oplus \mathsf{F}'$ in $\Mod \Gamma$, it follows that 
    \[
    \lambda_{\Gamma}\mathsf{er}\simeq \mathsf{Id}_{\mathsf{K}_{\mathsf{ac}}(\Inj\Gamma)}\oplus\lambda_{\Gamma}\mathsf{F}'
    \]
    in $\mathsf{K}_{\mathsf{ac}}(\Inj\Gamma)$.
    
    (i)$\Longrightarrow$(ii): If $\lambda_{\Gamma}\mathsf{e}$ is an equivalence, since $(\lambda_{\Gamma}\mathsf{e},\mathsf{r})$ is an adjoint pair, it follows that $\lambda_{\Gamma}\mathsf{er}\simeq \mathsf{Id}_{\mathsf{K}_{\mathsf{ac}}(\Inj\Gamma)}$ and so $\lambda_{\Gamma}\mathsf{F}'\simeq 0$ in $\mathsf{K}_{\mathsf{ac}}(\Inj\Gamma)$. 

    (ii)$\Longrightarrow$(i): If $\lambda_{\Gamma}\mathsf{F}'\simeq 0$ in $\mathsf{K}_{\mathsf{ac}}(\Inj\Gamma)$, then it follows that $\lambda_{\Gamma}\mathsf{e}\mathsf{r}\simeq \mathsf{Id}_{\mathsf{K}_{\mathsf{ac}}(\Inj\Gamma)}$. Consequently, $\mathsf{r}\colon\mathsf{K}_{\mathsf{ac}}(\Inj\Gamma)\rightarrow \mathsf{K}_{\mathsf{ac}}(\Inj\Lambda)$ is fully faithful and therefore the functor $\lambda_{\Gamma}\mathsf{e}$ induces an equivalence 
    \[
    \mathsf{K}_{\mathsf{ac}}(\Inj\Lambda)/\mathsf{ker}\lambda_{\Gamma}\mathsf{e}\simeq \mathsf{K}_{\mathsf{ac}}(\Inj\Gamma).
    \]
    We claim that $\mathsf{ker}\lambda_{\Gamma}\mathsf{e}$ is trivial. Indeed, if $\lambda_{\Gamma}\mathsf{e}(X)$ is contractible, then it follows by Corollary \ref{injective_dimension_of_boundaries} that $\id{ \mathsf{B}^i(\mathsf{e}(X))_{\Gamma}}\leq d+1$ for all $i$. Therefore $\id\ \!{\mathsf{e}(\mathsf{B}^i(X))_{\Gamma}}\leq d+1$, so by the assumption on $\mathsf{e}$ it follows that $\id{\mathsf{B}^i(X)_{\Lambda}}$ admit a common bound for all $i$. We claim that the latter, together with the fact that $X$ is acyclic, suffice for $X$ to be contractible. Indeed, consider the following exact sequence: 
    \[
    0\rightarrow \mathsf{B}^i(X)\rightarrow X^i\rightarrow \cdots\rightarrow X^j\rightarrow \mathsf{B}^{j+1}(X)\rightarrow 0,
    \]
    from which, for $j-i$ large enough, we infer that $\id{ \mathsf{B}^i(X)_{\Lambda}}=0$. Consequently, the image of every differential is injective, making $X$ a contractible complex. 
\end{proof}

\begin{cor} \label{equivalence_of_big_singularity}
    Let $(\Mod \Gamma,\Mod \Lambda,\mathsf{i},\mathsf{e},\mathsf{l})$ be a cleft extension of modules categories of Noetherian rings $\Gamma$ and $\Lambda$. Assume that $\mathsf{F}$ is perfect and nilpotent. The following are equivalent: 
    \begin{itemize}
        \item[(i)] The functor $\mathsf{e}\colon\mathsf{K}_{\mathsf{ac}}(\Inj\Lambda)\rightarrow \mathsf{K}_{\mathsf{ac}}(\Inj\Gamma)$ is an equivalence. 
        \item[(ii)] $\lambda_{\Gamma}\mathsf{F}'\simeq 0$ in $\mathsf{K}_{\mathsf{ac}}(\Inj\Gamma)$.
    \end{itemize} 
\end{cor}
\begin{proof}
    Since $\mathsf{F}$ is perfect and nilpotent, it follows by Proposition \ref{cleft_of_modules_is_cocleft} that $\mathsf{F}'$ is coperfect and nilpotent. Consequently, it follows by Lemma \ref{basic_properties_of_coperfect_endofunctor_on_cleft}, Proposition \ref{preserves_and_reflects_inj_dim} and in particular Remark \ref{bounds_for_injective_dimensions}, that the conditions of Theorem \ref{fox} are satisfied.
\end{proof}

\begin{exmp} \label{nice_example}
    Consider a triangular matrix ring $\Lambda=\big(\begin{smallmatrix}   A & N\\   0 & B \end{smallmatrix}\big)$ where $A$ and $B$ are Noetherian rings and $N$ is an $A$-$B$-bimodule that is finitely generated on both sides. Recall, from Example \ref{perfect_for_triangular}, that $\Mod\Lambda$ is a cleft extension of $\Mod A\!\times\! B$. Moreover, the functor $\mathsf{F}$ is given by $(X,Y)\mapsto (0,X\otimes_A N)$ and so the functor $\mathsf{F}'$ is given by $(X,Y)\mapsto (\mathsf{Hom}_B(N,Y),0)$. Assume that $\pd{_AN}<\infty$ and $\pd N_B<\infty$. Then, it follows from Example \ref{perfect_for_triangular} that $\mathsf{F}$ is perfect and nilpotent. If, additionally, $\gd B<\infty$, then $\lambda_{A\times B}\mathsf{F}'\simeq 0$ in $\mathsf{K}_{\mathsf{ac}}(\Inj A\!\times\! B)$. Indeed, since $\gd B<\infty$, it follows that every acyclic complex of injective $B$-modules is contractible, so for $(X,Y)=(X,0)\in\mathsf{K}_{\mathsf{ac}}(\Inj A\!\times \! B)$, 
    \[ 
    \lambda_{A\times B}\mathsf{F}'(X,Y)=\lambda_{A\times B}(\mathsf{Hom}_B(N,Y),0)=0.
    \]
    Summing up, using Corollary \ref{equivalence_of_big_singularity}, we infer that
    \[
    \mathsf{K}_{\mathsf{ac}}(\Inj\begin{pmatrix}   A & N\\   0 & B \end{pmatrix})\simeq \mathsf{K}_{\mathsf{ac}}(\Inj A \!\times\! B)\simeq \mathsf{K}_{\mathsf{ac}}(\Inj A),
    \]
    provided that $\pd{_AN}<\infty$ and $\gd B<\infty$. 
\end{exmp}

\begin{rem} \label{final_remark}
In this remark we explain how for Artin algebras, the equivalences of Corollary \ref{equivalence_of_big_singularity} occur as a consequence of the equivalences of Corollary \ref{equivalence_singularity_for_mod}. More generally, consider an exact functor $\mathsf{e}\colon \Mod\Lambda\rightarrow \Mod\Gamma$ of module categories of Artin algebras that admits a right adjoint $\mathsf{r}\colon\Mod\Gamma\rightarrow \Mod\Lambda$. Assume that the functors of the following diagram exist: 
\[
\begin{tikzcd}
\mathsf{K}_{\mathsf{ac}}(\Inj\Lambda) \arrow[r, "\lambda_{\Gamma}\mathsf{e}"] & \mathsf{K}_{\mathsf{ac}}(\Inj\Gamma) \arrow[l, "\mathsf{r}", bend left] \\
\mathsf{D}_{\mathsf{sg}}(\Lambda) \arrow[r, "\mathsf{e}"] \arrow[u, hook]     & \mathsf{D}_{\mathsf{sg}}(\Gamma) \arrow[u, hook]                    \end{tikzcd}
\]
By Remark \ref{duck}, the above diagram is commutative. We will explain that $\lambda_{\Gamma}\mathsf{e}$ is an equivalence, provided that $\mathsf{e}$ is an equivalence. Since $\Gamma$ and $\Lambda$ are assumed to be Artin algebras, it follows that $\mathsf{K}_{\mathsf{ac}}(\Inj\Gamma)^{\mathsf{c}}\simeq \mathsf{D}_{\mathsf{sg}}(\Gamma)$ and $\mathsf{K}_{\mathsf{ac}}(\Inj\Lambda)^{\mathsf{c}}\simeq \mathsf{D}_{\mathsf{sg}}(\Lambda)$ (not only up to direct summands - see Remark \ref{artin_algebras}). Consequently, by a standard dévissage argument the claim follows. We obtain more consequences which are of independent interest: 
\begin{itemize}
    \item[(i)] In a cleft extension $(\Mod\Gamma,\Mod\Lambda,\mathsf{i},\mathsf{e},\mathsf{l})$, under the assumption that $\mathsf{F}$ is perfect and nilpotent: 
    \[
    \mathbb{L}_{\mathsf{sg}}\mathsf{F}=0 \text{ in }\mathsf{D}_{\mathsf{sg}}(\Gamma) \iff\lambda_{\Gamma}\mathsf{F}'=0 \text{ in }\mathsf{K}_{\mathsf{ac}}(\Inj\Gamma).
    \]
    This follows by the above, Corollary \ref{equivalence_singularity_for_mod} and Corollary \ref{equivalence_of_big_singularity}.
    \item[(ii)] If $\Gamma$ and $\Lambda$ are finite dimensional algebras over a field that are singular equivalent of Morita type with level, then $\mathsf{K}_{\mathsf{ac}}(\Inj\Gamma)\simeq \mathsf{K}_{\mathsf{ac}}(\Inj\Lambda)$. In fact, we may describe the above equivalence explicitly. If $(_{\Lambda}M_{\Gamma},_{\Gamma}N_{\Lambda})$ is a pair of bimodules that define a singular equivalence of Morita type with level, we have $\mathsf{e}=-\otimes_{\Lambda}M$ and $\mathsf{r}=\mathsf{Hom}_{\Gamma}(M,-)$. Hence, the equivalence between the big singularity categories is given by $\mathsf{Hom}_{\Gamma}(M,-)$. This is the content of \cite[Lemma 4.7]{chen_liu_wang}. 
\end{itemize}
\end{rem}

\end{document}